\documentclass[printer]{gOMS2e}

\usepackage{amsfonts,amsmath,amssymb,graphicx}
\usepackage{al_macros,float} 
\RequirePackage[colorlinks]{hyperref} 
\usepackage{algorithm} 
\usepackage{algpseudocode}
\usepackage{etoolbox}
\usepackage{multirow}
\usepackage{rotating}
\usepackage{fixltx2e}
\usepackage{color}
\usepackage{longtable}
\usepackage[table]{xcolor}
\definecolor{lightgray}{gray}{0.9}

\definecolor{nickcol}{rgb}{0.580392157, 0.0, 0.82745098}

\definecolor{extlink}{cmyk}{0,0,0,1}
\definecolor{intlink}{cmyk}{0,0,0,1}
\definecolor{softpage}{cmyk}{0,0,0,1}
\def\@linkcolor{intlink}
\def\@anchorcolor{black}
\def\@citecolor{intlink}
\def\@filecolor{cyan}
\def\@urlcolor{extlink}
\def\@menucolor{red}

\newtoggle{fullpaper}
\toggletrue{fullpaper}
\togglefalse{fullpaper}

\newtoggle{arxivpaper}
\toggletrue{arxivpaper}

\newtheorem{theorem}{Theorem}[section]
\newtheorem{assumption}[theorem]{Assumption}

\newtheorem{lemma}[theorem]{Lemma}

\begin{document}

\doi{}
\jvol{}
\jnum{}
\jyear{}
\received{}

\iftoggle{arxivpaper}{
\title{Adaptive Augmented Lagrangian Methods:\\ Algorithms and Practical Numerical Experience---Detailed Version}
}{
\title{Adaptive Augmented Lagrangian Methods:\\ Algorithms and Practical Numerical Experience}
}

\author{Frank E.~Curtis$^{\dagger}$\thanks{$^{\dagger}$Department of Industrial and Systems Engineering, Lehigh University, Bethlehem, PA, USA. E-mail: \href{mailto:frank.e.curtis@gmail.com}{frank.e.curtis@gmail.com}.  Supported by Department of Energy grant DE--SC0010615.},
        Nicholas I.~M.~Gould$^{\ddagger}$\thanks{$^{\ddagger}$STFC-Rutherford Appleton Laboratory, Numerical Analysis Group, R18, Chilton, OX11 0QX, UK.  E-mail: \href{mailto:nick.gould@stfc.ac.uk}{nick.gould@stfc.ac.uk}.  Supported by Engineering and Physical Sciences Research Council grant EP/I013067/1.},
        Hao Jiang${^\S}$, and Daniel P.~Robinson${^\S}$\thanks{${^\S}$Department of Applied Mathematics and Statistics, Johns Hopkins University, Baltimore, MD, USA.  E-mail: \href{mailto:hjiang13@jhu.edu}{hjiang13@jhu.edu}, \href{mailto:daniel.p.robinson@gmail.com}{daniel.p.robinson@gmail.com}.  Supported by National Science Foundation grant DMS-1217153.}
        }

\date{\today}

\maketitle

\begin{abstract}
  In this paper, we consider augmented Lagrangian (AL) algorithms for solving large-scale nonlinear optimization problems that execute adaptive strategies for updating the penalty parameter.  Our work is motivated by the recently proposed adaptive AL trust region method by Curtis, Jiang, and Robinson~[\href{http://link.springer.com/article/10.1007/s10107-014-0784-y}{Math.~Prog., DOI:~10.1007/s10107-014-0784-y, 2013}].  The first focal point of this paper is a new variant of the approach that employs a line search rather than a trust region strategy, where a critical algorithmic feature for the line search strategy is the use of convexified piecewise quadratic models of the AL function for computing the search directions.  We prove global convergence guarantees for our line search algorithm that are on par with those for the previously proposed trust region method.  A second focal point of this paper is the practical performance of the line search and trust region algorithm variants in \Matlab{} software, as well as that of an adaptive penalty parameter updating strategy incorporated into the \lancelot{} software.  We test these methods on problems from the \CUTEst{} and \COPS{} collections, as well as on challenging test problems related to optimal power flow.  Our numerical experience suggests that the adaptive algorithms outperform traditional AL methods in terms of efficiency and reliability.  As with traditional AL algorithms, the adaptive methods are matrix-free and thus represent a viable option for solving extreme-scale problems.
\end{abstract}


\begin{classcode}
  49M05, 49M15, 49M29, 49M37, 65K05, 65K10, 90C06, 90C30, 93B40 
\end{classcode}

\section{Introduction}\label{sec.introduction}

Augmented Lagrangian (AL) methods \cite{Hes69,Pow69} have recently regained popularity due to growing interests in solving extreme-scale nonlinear optimization problems.  These methods are attractive in such settings as they can be implemented matrix-free~\cite{AndBMS08,BirM12,ConGT91c,KocS03} and have global and local convergence guarantees under relatively weak assumptions~\cite{FerS10b,IzmS11}.  Furthermore, certain variants of AL methods~\cite{GabaMerc76,GlowMarr75} have proved to be very efficient for solving certain structured problems~\cite{BoyPCPE11,QinGM11,YanZY10}.

A new AL trust region method was recently proposed and analyzed in \cite{curtis12}.  The novel feature of that algorithm is an adaptive strategy for updating the penalty parameter inspired by techniques for performing such updates in the context of exact penalty methods \cite{ByrdLopeNoce10,ByrNW08,MS95automaticdecrease}.  This feature is designed to overcome a potentially serious drawback of traditional AL methods, which is that they may be ineffective during some (early) iterations due to poor choices of the penalty parameter and/or Lagrange multiplier estimates.  In such situations, the poor choices of these quantities may lead to little or no improvement in the primal space and, in fact, the iterates may diverge from even a well-chosen initial iterate.  The key idea for avoiding this behavior in the algorithm proposed in \cite{curtis12} is to adaptively update the penalty parameter \emph{during} the step computation in order to ensure that the trial step yields a sufficiently large reduction in linearized constraint violation, thus \emph{steering} the optimization process steadily toward constraint satisfaction.

The contributions of this paper are two-fold.  First, we present an AL line search method based on the same framework employed for the trust region method in~\cite{curtis12}.  The main difference between our new approach and that in \cite{curtis12}, besides the differences inherent in using line searches instead of a trust region strategy, is that we utilize a convexified piecewise quadratic model of the AL function to compute the search direction in each iteration.  With this modification, we prove that our line search method achieves global convergence guarantees on par with those proved for the trust region method in \cite{curtis12}.  The second contribution of this paper is that we perform extensive numerical experiments with a \Matlab{} implementation of the adaptive algorithms (i.e., both line search and trust region variants) and an implementation of an adaptive penalty parameter updating strategy in the \lancelot{} software~\cite{ConGT92a}.  We test these implementations on problems from the \CUTEst{}~\cite{GouOT03} and \COPS{}~\cite{BonDM98} collections, as well as on test problems related to optimal power flow~\cite{zimmerman2011matpower}.  Our results indicate that our adaptive algorithms outperform traditional AL methods in terms of efficiency and reliability.

The remainder of the paper is organized as follows.  In \S\ref{sec.basic}, we present our adaptive AL line search method and state convergence results.  Details and proofs of these results, which draw from those in \cite{curtis12}, can be found in Appendices~\ref{subsec-wellposed} and \ref{sec-convergence}\iftoggle{arxivpaper}{}{ and in \cite{CJJR-arxiv}}.  We then provide numerical results in \S\ref{sec-numerical} to illustrate the effectiveness of our implementations of our adaptive AL algorithms.  We give conclusions in \S\ref{sec-conclusions}.
\medskip

\noindent
\emph{Notation.}  We often drop function arguments once a function is defined.  We also use a subscript on a function name to denote its value corresponding to algorithmic quantities using the same subscript.  For example, for a function $f : \Re^n \to \Re$, if $x_k$ is the value for the variable $x$ during iteration $k$ of an algorithm, then $f_k := f(x_k)$.  We also often use subscripts for constants to indicate the algorithmic quantity to which they correspond. For example, $\gamma_\mu$ denotes a parameter corresponding to the algorithmic quantity $\mu$. 

\section{An Adaptive Augmented Lagrangian Line Search Algorithm}\label{sec.basic}

\subsection{Preliminaries}

We assume that all problems under our consideration are formulated as
\begin{equation} \label{nep}
  \minimizearg{x}{n}\ f(x)\ \subject\ c(x) = 0, \mgap l \leq x \leq u.
\end{equation}
Here, we assume that the objective function $f:\Re^n \to \Re$ and constraint function $c:\Re^n\to\Re^m$ are twice continuously differentiable, and that the variable lower bound vector $l \in \Re^n$ and upper bound vector $u \in \Re^n$ satisfy $l \leq u$.  Our goal is to design an algorithm that will compute a first-order primal-dual stationary point for problem \eqref{nep}.  However, in order for the algorithm to be suitable as a general-purpose approach, it should have mechanisms for terminating and providing useful information when an instance of \eqref{nep} is (locally) infeasible.  In such cases, we have designed our algorithm so that it transitions to finding an infeasible first-order stationary point for the nonlinear feasibility problem
\begin{equation}\label{vep}
  \minimizearg{x}{n}\ v(x)\ \subject\ l\leq x\leq u,
\end{equation}
where $v : \Re^n \to \Re$ is defined as $v(x) = \half \twonorm{c(x)}^2$.

As implied by the previous paragraph, our algorithm requires first-order stationarity conditions for problems \eqref{nep} and \eqref{vep}, which can be stated in the following manner.  First, introducing a Lagrange multiplier vector $y \in \Re^m$, we define the Lagrangian for problem~\eqref{nep}, call it $\ell : \Re^n \times \Re^m \to \Re$, by
\begin{equation*}
  \ell(x,y) = f(x) - c(x)\T y.
\end{equation*}
Then, defining the gradient of the objective function $g : \Re^n \to \Re^n$ by $g(x) = \Grad f(x)$, the Jacobian of the constraint functions $J : \Re^n \to \Re^{m \times n}$ by $J(x) = \Grad c(x)$, and the projection operator onto the bounds $P : \Re^n \to \Re^n$, component-wise for $i \in \{1,\dots,n\}$, by
\begin{equation*}
  [P(x)]_i=
  \left\{
  \begin{aligned}
    &l_i && \text{if } x_i \leq l_i \\
    &u_i && \text{if } x_i \geq u_i \\
    &x_i && \text{otherwise}
  \end{aligned}
  \right.
\end{equation*}
we may introduce the primal-dual stationarity measure
$F\sub{L} : \Re^n \times \Re^m \to \Re^n$ given by
\begin{equation*}
  F\sub{L}(x,y) = P(x - \Grad_x \ell (x,y)) - x = P(x - (g(x) - J(x)\T y)) - x.
\end{equation*}
First-order primal-dual stationary points for \eqref{nep} can then be characterized as zeros of the primal-dual stationarity measure $F\sub{OPT} : \Re^n \times \Re^m \to \Re^{n+m}$ defined by stacking the stationarity measure $F\sub{L}$ and the constraint function $-c$, i.e., a first-order primal-dual stationary point for \eqref{nep} is any pair $(x,y)$ with $l\leq x\leq u$ satisfying
\begin{equation}\label{F-zero}
  0 = F\sub{OPT}(x,y) = \pmat{ F\sub{L}(x,y) \\ -c(x) } = \pmat{ P(x - \Grad_x \ell(x,y))-x \\ \Grad_y \ell(x,y) }.
\end{equation} 
Similarly, a first-order primal stationary point for \eqref{vep} is any $x$ with $l \leq x \leq u$ satisfying
\begin{equation}\label{Finf-zero}
  0 = F\sub{FEAS}(x),
\end{equation}
where $F\sub{FEAS} : \Re^n \to \Re^n$ is defined by
\begin{equation*}
  F\sub{FEAS}(x) = P(x - \Grad_x v(x)) - x = P(x - J(x)\T c(x)) - x.
\end{equation*}
In particular, if $l \leq x \leq u$, $v(x) > 0$, and \eqref{Finf-zero} holds, then $x$ is an infeasible stationary point for problem~\eqref{nep}.

Over the past decades, a variety of effective numerical methods have been proposed for solving large-scale bound-constrained optimization problems.  Hence, the critical issue in solving problem \eqref{nep} is how to handle the presence of the equality constraints.  As in the wide variety of penalty methods that have been proposed, the strategy adopted by AL methods is to remove these constraints, but push the algorithm to satisfy them through the addition of influential terms in the objective.  In this manner, problem \eqref{nep} (or at least \eqref{vep}) can be solved via a sequence of bound-constrained subproblems---thus allowing AL methods to exploit the methods that are available for subproblems of this type.  Specifically, AL methods consider a sequence of subproblems in which the objective is a weighted sum of the Lagrangian $\ell$ and the constraint violation measure $v$.  By scaling $\ell$ by a penalty parameter $\mu \geq 0$, each subproblem involves the minimization of a function $\Lscr : \Re^n \times \Re^m \times \Re \to \Re$, called the augmented Lagrangian (AL), defined by
\begin{equation*}
  \Lscr(x,y,\mu) = \mu\ell(x,y) + v(x) = \mu(f(x) - c(x)\T y) + \half\twonorm{c(x)}^2.
\end{equation*}
Observe that the gradient of the AL with respect to $x$, evaluated at $(x,y,\mu)$, is given by
\begin{equation*}
  \Grad_x\Lscr(x, y,\mu)
  = \mu\big(g(x) - J(x)\T \pi(x,y,\mu)\big), 
\end{equation*}
where we define the function $\pi : \Re^n \times \Re^m \times \Re \to \Re^m$ by
\begin{equation}\label{def-pi}
  \pi(x,y,\mu) = y - \tfrac{1}{\mu}c(x).
\end{equation}
Hence, each subproblem to be solved in an AL method has the form
\begin{equation} \label{prob-al}
  \minimize{x\in\Re^n}\ \Lscr(x,y,\mu)
  \ \subject\ l \leq x \leq u.
\end{equation}
Given a pair $(y,\mu)$, a first-order stationary point for problem~\eqref{prob-al} is any zero of the primal-dual stationarity measure $F\sub{AL} : \Re^n \times \Re^m \times \Re \to \Re^n$, defined similarly to $F\sub{L}$ but with the Lagrangian replaced by the augmented Lagrangian; i.e., given $(y,\mu)$, a first-order stationary point for \eqref{prob-al} is any $x$ satisfying
\begin{equation} \label{foo-al}
   0 = F\sub{AL}(x,y,\mu) = P(x - \Grad_x\Lscr(x,y,\mu))-x.
\end{equation}

Given a pair $(y,\mu)$ with $\mu > 0$, a traditional AL method proceeds by (approximately) solving \eqref{prob-al}, which is to say that it finds a point, call it $x(y,\mu)$, that (approximately) satisfies \eqref{foo-al}.  If the resulting pair $(x(y,\mu),y)$ is not a first-order primal-dual stationary point for \eqref{nep}, then the method would modify the Lagrange multiplier $y$ or penalty parameter $\mu$ so that, hopefully, the solution of the subsequent subproblem (of the form \eqref{prob-al}) yields a better primal-dual solution estimate for \eqref{nep}.  The function $\pi$ plays a critical role in this procedure.  In particular, observe that if $c(x(y,\mu)) = 0$, then $\pi(x(y,\mu),y,\mu) = y$ and \eqref{foo-al} would imply $F\sub{OPT}(x(y,\mu),y) = 0$, i.e., that $(x(y,\mu),y)$ is a first-order primal-dual stationary point for~\eqref{nep}.  Hence, if the constraint violation at $x(y,\mu)$ is sufficiently small, then a traditional AL method would set the new value of $y$ as $\pi(x,y,\mu)$.  Otherwise, if the constraint violation is not sufficiently small, then the penalty parameter is decreased to place a higher priority on reducing it during subsequent iterations.

\subsection{Algorithm Description}\label{sec-LSalgorithm}

Our AL line search algorithm is similar to the AL trust region method proposed in \cite{curtis12}, except for two key differences: it executes line searches rather than using a trust region framework, and it employs a convexified piecewise quadratic model of the AL function for computing the search direction in each iteration.  The main motivation for utilizing a convexified model is to ensure that each computed search direction is a direction of strict descent for the AL function from the current iterate, which is necessary to ensure the well-posedness of the line search.  However, it should be noted that, practically speaking, the convexification of the model does not necessarily add any computational difficulties when computing each direction; see \S\ref{sec-implement-matlab}.  Similar to the trust region method proposed in~\cite{curtis12}, a critical component of our algorithm is the adaptive strategy for updating the penalty parameter $\mu$ during the search direction computation.  This is used to ensure steady progress---i.e., steer the algorithm---toward solving~\eqref{nep} (or at least \eqref{vep}) by monitoring predicted improvements in linearized feasibility.

The central component of each iteration of our algorithm is the search direction computation.  In our approach, this computation is performed based on local models of the constraint violation measure $v$ and the AL function $\Lscr$ at the current iterate, which at iteration $k$ is given by $(x_k,y_k,\mu_k)$.  The local models that we employ for these functions are, respectively, $\qv : \Re^n \to \Re$ and $\tilde{q} : \Re^n \to \Re$, defined as follows:
\begin{alignat*}{2}\label{dpr-test}
  \qv(s;x) &= \half\twonorm{c(x)+J(x)s}^2 \\
  \tilde{q}(s;x,y,\mu) &= \Lscr(x,y) + \Grad_x \Lscr(x,y)\T s + \max\{\half s\T (\mu \Hess_{xx} \ell(x,y) + J(x)\T J(x))s,0\}.
\end{alignat*}
We note that $\qv$ is a typical Gauss-Newton model of the constraint violation measure~$v$, and $\tilde{q}$ is a convexification of a second-order approximation of the augmented Lagrangian.  (We use the notation $\tilde{q}$ rather than simply $q$ to distinguish between the model above and the second-order model---without the max---that appears extensively in \cite{curtis12}.)

Our algorithm computes two types of steps during each iteration.  The purpose of the first step, which we refer to as the steering step, is to gauge the progress towards linearized feasibility that may be achieved (locally) from the current iterate.  This is done by (approximately) minimizing our model $\qv$ of the constraint violation measure $v$ within the bound constraints and a trust region.  Then, a step of the second type is computed by (approximately) minimizing our model $\tilde{q}$ of the AL function $\Lscr$ within the bound constraints and a trust region.  If the reduction in the model $\qv$ yielded by the latter step is sufficiently large---say, compared to that yielded by the steering step---then the algorithm proceeds using this step as the search direction.  Otherwise, the penalty parameter may be reduced, in which case a step of the latter type is recomputed.  This process repeats iteratively until a search direction is computed that yields a sufficiently large (or at least not too negative) reduction in $\qv$.  As such, the iterate sequence is intended to make steady progress toward (or at least approximately maintain) constraint satisfaction throughout the optimization process, regardless of the initial penalty parameter value.

We now describe this process in more detail.  During iteration $k$, the steering step $r_k$ is computed via the optimization subproblem given by
\begin{equation} \label{prob-normal}
  \minimize{r\in\Re^n}\ \qv(r;x_k)\ \subject\ l\leq x_k+r \leq u, \,\twonorm{r} \leq \TRk,
\end{equation}
where, for some constant $\delta > 0$, the trust region radius is defined to be
\begin{equation}\label{sandy}
   \TRk := \delta\twonorm{F\sub{FEAS}(x_k)} \geq 0.
\end{equation}
A consequence of this choice of trust region radius is that it forces the steering step to be smaller in norm as the iterates of the algorithm approach any stationary point of the constraint violation measure~\cite{Toi13}.  This prevents the steering step from being too large relative to the progress that can be made toward minimizing $v$.  While \eqref{prob-normal} is a convex optimization problem for which there are efficient methods, in order to reduce computational expense our algorithm only requires $r_k$ to be an approximate solution of~\eqref{prob-normal}.  In particular, we merely require that $r_k$ yields a reduction in $\qv$ that is proportional to that yielded by the associated Cauchy step (see \eqref{cond1} later on), which is defined to be
\begin{equation} \label{def-cauchy}
  \rCk := \rC(x_k,\TRk) := P(x_k-\betaC_k J_k^Tc_k) - x_k
\end{equation}
for $\betaC_k:=\betaC(x_k,\TRk)$ such that, for some $\eps_r\in(0,1)$, the step $\rCk$ satisfies
\begin{equation} \label{c1}
  \delmod_v(\rCk;x_k) := \qv(0;x_k) - \qv(\rCk;x_k) \geq - \eps_r \rCk\T J_k\T c_k \words{and} \twonorm{\rCk}\leq\TRk.
\end{equation}

Appropriate values for $\betaC_k$ and $\rCk$---along with auxiliary nonnegative scalar quantities $\eps_k$ and $\Gamma_k$ to be used in subsequent calculations in our method---are computed by Algorithm~\ref{alg-rCk}.  The quantity $\delmod_v(\rCk;x_k)$ representing the predicted reduction in constraint violation yielded by $\rCk$ is guaranteed to be positive at any $x_k$ that is not a first-order stationary point for $v$ subject to the bound constraints; see part (i) of Lemma~\ref{keylemma1}.  We define a similar reduction $\delmodv(r_k;x_k)$ for the steering step $r_k$. 

\begin{algorithm}
  \caption{Cauchy step computation for the feasibility subproblem~\eqref{prob-normal}}
  \label{alg-rCk}
 \begin{algorithmic}[1]
    \Procedure{Cauchy\_feasibility}{$x_k,\theta_k$}
     \State \textbf{restrictions} : $\theta_k \geq 0$.
      \State \textbf{available constants} : $\{\eps_r,\gamma\} \subset (0,1)$.
      \State Compute the smallest integer $l_k \geq 0$ satisfying $\twonorm{P(x_k - \gamma^{l_k} J_k\T c_k)- x_k} \leq \theta_k$.
      \If{$l_k > 0$}
        \State Set $\Gamma_k \gets \min\{2,\half(1+\twonorm{P(x_k - \gamma^{l_k-1} J_k\T c_k)-x_k}/\theta_k)\}$. 
      \Else
\State Set $\Gamma_k \gets 2$.
  \EndIf
  \State Set $\betaC_k \gets \gamma^{l_k}$, $\rCk \gets P(x_k - \betaC_k J_k\T c_k) - x_k$, and $\eps_k \gets 0$.
  \While{$\rCk$ does not satisfy~\eqref{c1}}\label{line-rCk-while}
\State Set $\eps_k \gets \max(\eps_k,-\delmod_v(\rCk;x_k) / \rCk\T J_k\T c_k)$.
     \State Set $\betaC_k \gets \gamma\betaC_k$ and $\rCk \gets P(x_k - \betaC_k J_k\T c_k) - x_k$.
  \EndWhile                
        \State \Return: $(\betaC_k,\rCk,\eps_k,\Gamma_k)$
        \EndProcedure
\end{algorithmic}
\end{algorithm}

After computing a steering step $r_k$, we proceed to compute a trial step $s_k$ via
\begin{equation} \label{subprob-al}
   \minimize{s\in\Re^n}\ \tilde{q}(s;x_k,y_k,\mu_k)\ \subject\ l \leq x_k+s \leq u, \,\twonorm{s} \leq \TRALk,
\end{equation}
where, given $\Gamma_k > 1$ from the output of Algorithm~\ref{alg-rCk}, we define the trust region radius
\begin{equation}\label{sandy-AL}
  \TRALk:= \TRAL(x_k,y_k,\mu_k,\Gamma_k) = \Gamma_k \delta \twonorm{F\sub{AL}(x_k,y_k,\mu_k)} \geq 0.
\end{equation}
As for the steering step, we allow inexactness in the solution of \eqref{subprob-al} by only requiring the step $s_k$ to satisfy a Cauchy decrease condition (see \eqref{cond2} later on), where the Cauchy step for problem~\eqref{subprob-al} is
\begin{equation}\label{def-cauchy2}
  \sCk := \sC(x_k,y_k,\mu_k,\TRALk,\eps_k) := P(x_k-\alphaC_k\Grad_x \Lscr(x_k,y_k,\mu_k))-x_k
\end{equation}
for $\alphaC_k=\alphaC(x_k,y_k,\mu_k,\TRALk,\eps_k)$ such that, for $\eps_k \geq 0$ returned from Algorithm~\ref{alg-rCk}, $\sCk$ yields
\begin{equation} \label{c2}
  \begin{aligned}
    \Deltait\tilde{q}(\sCk;x_k,y_k,\mu_k) :=&\ \tilde{q}(0;x_k,y_k,\mu_k) - \tilde{q}(\sCk;x_k,y_k,\mu_k) \\
    \geq&\ - \frac{(\eps_k+\eps_r)}{2} \sCk\T \Grad_x\Lscr(x_k,y_k,\mu_k) \words{and} \twonorm{\sCk}\leq\TRALk.
  \end{aligned}
\end{equation}

Algorithm~\ref{alg-sCk} describes our procedure for computing $\alphaC_k$ and $\sCk$.  (The importance of incorporating $\Gamma_k$ in~\eqref{sandy-AL} and $\eps_k$ in~\eqref{c2} is revealed in the proofs of Lemmas~\ref{lem-algo} and~\ref{lem-models}\iftoggle{arxivpaper}{}{; see~\cite{CJJR-arxiv}}.)  The quantity $\Deltait\tilde{q}(\sCk;x_k,y_k,\mu_k)$ representing the predicted reduction in $\Lscr(\cdot,y_k,\mu_k)$ yielded by $\sCk$ is guaranteed to be positive at any $x_k$ that is not a first-order stationary point for $\Lscr(\cdot,y_k,\mu_k)$ subject to the bound constraints; see part~(ii) of Lemma~\ref{keylemma1}.  A similar quantity $\Deltait\tilde{q}(s_k;x_k,y_k,\mu_k)$ is also used for the search direction $s_k$.

\begin{algorithm}
  \caption{Cauchy step computation for the Augmented Lagrangian subproblem~\eqref{subprob-al}.}
  \label{alg-sCk}
\begin{algorithmic}[1]
          \Procedure{Cauchy\_AL}{$x_k,y_k,\mu_k,\Theta_k,\eps_k$}
  \State \textbf{restrictions} : $\mu_k > 0$, $\Theta_k > 0$, and $\eps_k \geq 0$.
          \State \textbf{available constant} : $\gamma\in(0,1)$.
                   \State Set $\alphaC_k \gets 1$ and $\sCk \gets P(x_k - \alphaC_k \Grad_x \Lscr(x_k,y_k,\mu_k))-x_k$.\label{temp1}
                   \While{\eqref{c2} is not satisfied}\label{temp2}
   \State Set $\alphaC_k \gets \gamma\alphaC_k$ and $\sCk \gets P(x_k - \alphaC_k \Grad_x\Lscr(x_k,y_k,\mu_k))-x_k$.\label{temp3}
                   \EndWhile
        \State \Return: $(\alphaC_k,\sCk)$
        \EndProcedure
\end{algorithmic}
\end{algorithm}

Our complete algorithm is given as Algorithm~\ref{alg-aal} on page~\pageref{alg-aal}.  In particular, the $k$th iteration proceeds as follows.  Given the $k$th iterate tuple $(x_k,y_k,\mu_k)$, the algorithm first determines whether the first-order primal-dual stationarity conditions for \eqref{nep} or the first-order stationarity condition for \eqref{vep} are satisfied.  If either is the case, then the algorithm terminates, but otherwise the method enters the \textbf{while} loop in line~\ref{while-1} to check for stationarity with respect to the AL function.  This loop is guaranteed to terminate finitely; see Lemma~\ref{lem-wd-while-1}.  Next, after computing appropriate trust region radii and Cauchy steps, the method enters a block for computing the steering step $r_k$ and trial step $s_k$.  Through the \textbf{while} loop on line~\ref{while-2}, the overall goal of this block is to compute (approximate) solutions of subproblems \eqref{prob-normal} and \eqref{subprob-al} satisfying
\begin{subequations}\label{conds}
  \begin{align}
    \Deltait\tilde{q}(s_k;x_k,y_k,\mu_k) &\geq \kappa_1\Deltait\tilde{q}(\sCk;x_k,y_k,\mu_k) > 0, \bgap
    l \leq x_k+s_k \leq u, \bgap
    \twonorm{s_k} \leq \TRALk, \label{cond1} \\
    \delmodv(r_k;x_k) &\geq \kappa_2\delmodv(\rCk;x_k) \geq 0, \bgap\qquad\,\hspace{0.1cm}
    l \leq x_k+r_k \leq u, \bgap
    \twonorm{r_k} \leq \TRk, \label{cond2}\\
    \words{and} \delmodv(s_k;x_k) &\geq \min\{\kappa_3\delmodv(r_k;x_k), v_k-\half(\kappa_t t_j)^2\}. \label{cond2b}
  \end{align}
\end{subequations}
In these conditions, the method employs user-provided constants $\{\kappa_1,\kappa_2,\kappa_3, \kappa_t\} \subset(0,1)$ and the algorithmic quantity $t_j > 0$ representing the $j$th constraint violation target.  It should be noted that, for sufficiently small $\mu > 0$, many approximate solutions to \eqref{prob-normal} and \eqref{subprob-al} satisfy~\eqref{conds}, but for our purposes (see Theorem~\ref{thm-wd}) it is sufficient that, for sufficiently small $\mu > 0$, they are at least satisfied by $r_k = \rCk$ and $s_k = \sCk$.  A complete description of the motivations underlying conditions \eqref{conds} can be found in~\cite[Section~3]{curtis12}.  In short, \eqref{cond1} and \eqref{cond2} are Cauchy decrease conditions while~\eqref{cond2b} ensures that the trial step predicts progress toward constraint satisfaction, or at least predicts that any increase in constraint violation is limited (when the right-hand side is negative).

\begin{algorithm}
  \caption{Adaptive Augmented Lagrangian Line Search Algorithm}
  \label{alg-aal}
  \begin{algorithmic}[1]
    \State Choose $\{\gamma,\gamma_\mu,\gamma_\alpha,\gamma_t,\gamma_T,\kappa_1,\kappa_2,\kappa_3,\eps_r,\kappa_t,\etasuccessful,\etaverysuccessful\} \subset (0,1)$ and $\{\delta,\epsilon,Y\} \subset (0,\infty)$ such that $\etaverysuccessful \geq \etasuccessful$.
    \State Choose initial primal-dual pair $(x_0,y_0)$ and initialize $\{\mu_0,\delta_0,t_0,t_1,T_1,Y_1\} \subset (0,\infty)$ such that $Y_1 \geq Y$ and $\twonorm{y_0} \leq Y_1$.
    \State Set $k \gets 0$, $k_0 \gets 0$, and $j \gets 1$.
    \Loop
      \If{$F\sub{OPT}(x_k,y_k) = 0$,}\label{term-kkt-m1}
        \State\label{term-kkt} \Return the first-order stationary solution $(x_k,y_k)$.
      \EndIf
      \If{$\twonorm{c_k} > 0$ and $F\sub{FEAS}(x_k) = 0$,}\label{term-isp-m1}
        \State\label{term-isp} \Return the infeasible stationary point $x_k$.
      \EndIf
      \While{$F\sub{AL}(x_k,y_k,\mu_k) = 0$,}\label{while-1}
        \State\label{mu-dec-1} Set $\mu_k\gets\gamma_\mu\mu_k$.
      \EndWhile
      \State Define $\TRk$ by~\eqref{sandy}.
      \State Use Algorithm~\ref{alg-rCk} to compute $(\betaC_k,\rCk,\eps_k,\Gamma_k) = \text{\sc Cauchy\_feasibility}(x_k,\theta_k)$.\label{line-rCk}
      \State Define $\TRALk$ by~\eqref{sandy-AL}.
      \State Use Algorithm~\ref{alg-sCk} to compute $(\alphaC_k,\sCk) = \text{\sc Cauchy\_AL}(x_k,y_k,\mu_k,\Theta_k,\eps_k)$.\label{line-sCk-1} 
      \State Compute approximate solutions $r_k$ to~\eqref{prob-normal} and $s_k$ to \eqref{subprob-al} that satisfy~\eqref{cond1}--\eqref{cond2}.
      \While{\eqref{cond2b} is not satisfied or $F\sub{AL}(x_k,y_k,\mu_k) = 0$,}\label{while-2}
        \State Set $\mu_k\gets\gamma_\mu\mu_k$ and define $\TRALk$ by~\eqref{sandy-AL}.\label{steering-decrease}
        \State Use Algorithm~\ref{alg-sCk} to compute $(\alphaC_k,\sCk) = \text{\sc Cauchy\_AL}(x_k,y_k,\mu_k,\Theta_k,\eps_k)$.\label{line-sCk-2} 
        \State Compute an approximate solution $s_k$ to~\eqref{subprob-al} satisfying~\eqref{cond1}.
      \EndWhile
      \State\label{line-alpha} Set $\alpha_k \gets \gamma_\alpha^l$ where $l \geq 0$ is the smallest integer satisfying~\eqref{dec-L-ls}.
      \State\label{stepx-k+1} Set $\xnext \gets x_k + \alpha_k s_k$.
      \If{$\twonorm{c_{k+1}} \leq t_j$,}\label{line-c-small}
        \State\label{line-y} Compute any $\yhat_{k+1}$ satisfying \eqref{new-y}.
        \If{$\min\{\twonorm{F\sub{L}(x_{k+1},\yhat_{k+1})}, \twonorm{F\sub{AL}(x_{k+1},y_k,\mu_k)}\} \leq T_j$,}\label{crunchy}  
          \State\label{line-kj} Set $k_j \gets k+1$ and $Y_{j+1} \gets \max\{Y, t_{j-1}^{-\epsilon}\}$.
          \State\label{aaahhh} Set $t_{j+1} \gets \min\{\gamma_t t_j,t_j^{1+\epsilon}\}$ and $T_{j+1} \gets \gamma_T T_j$.
          \State Set $y_{k+1}$ from~\eqref{y-combo} where $\alpha_y$ satisfies \eqref{new-y-2}.\label{line-y-bd}
          \State Set $j\gets j+1$.
        \Else
          \State Set $y_{k+1} \gets y_k$.
        \EndIf
      \Else
        \State Set $y_{k+1} \gets y_k$.
      \EndIf
      \State Set $\mu_{k+1} \gets \mu_k$.
      \State Set $k\gets k+1$.
    \EndLoop
  \end{algorithmic}
\end{algorithm}

With the search direction $s_k$ in hand, the method proceeds to perform a backtracking line search along the strict descent direction $s_k$ for $\Lscr(\cdot,y_k,\mu_k)$ at $x_k$.  Specifically, for a given $\gamma_\alpha\in(0,1)$, the method computes the smallest integer $l\geq 0$ such that
\begin{equation} \label{dec-L-ls}
  \Lscr(x_k + \gamma_\alpha^l s_k,y_k,\mu_k) \leq \Lscr(x_k,y_k,\mu_k) - \etasuccessful\gamma_\alpha^l\Deltait\widetilde{q}(s_k;x_k,y_k,\mu_k),
\end{equation}
and then sets $\alpha_k \gets \gamma_\alpha^l$ and $\xnext \gets x_k + \alpha_k s_k$.  The remainder of the iteration is then composed of potential modifications of the Lagrange multiplier vector and target values for the accuracies in minimizing the constraint violation measure and AL function subject to the bound constraints.  First, the method checks whether the constraint violation at the next primal iterate $\xnext$ is sufficiently small compared to the target $t_j > 0$.  If this requirement is met, then a multiplier vector $\yhat_{k+1}$ that satisfies
\begin{equation}\label{new-y}
  \twonorm{F\sub{L}(x_{k+1},\yhat_{k+1})} \leq \min\left\{\twonorm{F\sub{L}(x_{k+1},y_k)},\twonorm{F\sub{L}(x_{k+1},\pi(x_{k+1},y_k,\mu_k))} \right\}
\end{equation}
is computed. Two obvious potential choices for $\yhat_{k+1}$ are $y_k$ and $\pi(x_{k+1},y_k,\mu_k)$, but another viable candidate would be an approximate least-squares multiplier estimate (which may be computed via a linearly constrained optimization subproblem).  The method then checks if either $\twonorm{F\sub{L}(x_{k+1}, \yhat_{k+1})}$ or $\twonorm{F\sub{AL}(x_{k+1}, y_k, \mu_k)}$ is sufficiently small with respect to the target value $T_j > 0$.  
If so, then new target values $t_{j+1} < t_j$ and $T_{j+1} < T_j$ are set, 
$Y_{j+1} \geq Y_j$ is chosen, and a new Lagrange multiplier vector is set as
\begin{equation}\label{y-combo}
  y_{k+1} \gets (1-\alpha_y)y_k + \alpha_y \yhat_{k+1},
\end{equation}
where $\alpha_y$ is the largest value in $[0,1]$ such that
\begin{equation} \label{new-y-2}
  \twonorm{(1-\alpha_y)y_k + \alpha_y \yhat_{k+1}} \leq Y_{j+1}.
\end{equation}
This updating procedure is well-defined since the choice $\alpha_y \gets 0$ results in $y_{k+1} \gets y_k$, for which \eqref{new-y-2} is satisfied 
since $\|y_k\|_2 \leq Y_j \leq Y_{j+1}$.  If either line~\ref{line-c-small} or line \ref{crunchy} in Algorithm~\ref{alg-aal} tests false, then the method simply sets $y_{k+1} \gets y_k$.  We note that unlike more traditional augmented Lagrangian approaches \cite{AndBMS08,ConGT91c}, the penalty parameter is not adjusted on the basis of a test like that on line~\ref{line-c-small}, but instead relies on our steering procedure.  Moreover, in our approach we decrease the target values at a linear rate for simplicity, but more sophisticated approaches may be used \cite{ConGT91c}.


\subsection{Well-posedness and global convergence}

In this section, we state two vital results, namely that 
Algorithm~\ref{alg-aal} is well posed, and that limit points of the iterate sequence have desirable 
properties. 
\iftoggle{arxivpaper}{
Proofs of these results, which are similar to those in \cite{curtis12}, are given in Appendices A and B.
}{
Vital components of these results are given in Appendices A and B.  (The proofs of these results are similar to the corresponding results in \cite{curtis12}; for reference, complete details can be found in \cite{CJJR-arxiv}.)
}
In order to show well-posedness of the algorithm, we make the following 
formal assumption.

\begin{assumption} \label{ass-posed0}
  At each given $x_k$, the objective function $f$ and constraint function $c$ are both twice-continuously differentiable.
\end{assumption}

Under this assumption, we have the following theorem.

\begin{theorem} \label{thm-wd}
Suppose that Assumption~\ref{ass-posed0} holds.
Then the $k$th iteration of Algorithm~\ref{alg-aal} is well posed.  That is, either the algorithm will terminate in line~\ref{term-kkt} or~\ref{term-isp}, or it will compute $\mu_k > 0$ such that $F\sub{AL}(x_k,y_k,\mu_k) \neq 0$ and for the steps $s_k = \sCk$ and $r_k = \rCk$ the conditions in~\eqref{conds} will be satisfied, in which case $(x_{k+1},y_{k+1},\mu_{k+1})$ will be computed.
\end{theorem}

According to Theorem~\ref{thm-wd}, we have that Algorithm~\ref{alg-aal} will either terminate finitely or produce an infinite sequence of iterates.  If it terminates finitely---which can only occur if line~\ref{term-kkt} or~\ref{term-isp} is executed---then the algorithm has computed a first-order stationary solution or an infeasible stationary point and there is nothing else to prove about the algorithm's performance in such cases.  Therefore, it remains to focus on the global convergence properties of Algorithm~\ref{alg-aal} under the assumption that the sequence $\{(x_k,y_k,\mu_k)\}$ is infinite.  For such cases, we make the following additional assumption.

\begin{assumption} \label{ass-0}
  The primal sequences $\{x_k\}$ and $\{x_k+s_k\}$ are contained in a convex compact set over which the objective function $f$ and constraint function $c$ are both twice-continuously differentiable. 
\end{assumption}



Our main global convergence result for Algorithm~\ref{alg-aal} is as follows.

\begin{theorem} \label{thm-main}
  Suppose that Assumptions~\ref{thm-wd} and \ref{ass-0} hold. 
  Then one of the following must hold:
  \begin{itemize}
    \item[(i)] every limit point $\xstar$ of $\{x_k\}$ is an infeasible stationary point;
    \item[(ii)] $\mu_k \nrightarrow 0$ and there exists an infinite ordered set $\Kscr \subseteq \mathbb{N}$ such that every limit point of $\{(x_k,\yhat_k)\}_{k\in\Kscr}$ is first-order stationary for~\eqref{nep}; or
  \item[(iii)] $\mu_k \to 0$, every limit point of $\{x_k\}$ is feasible, and if there exists a positive integer $p$ such that $\mu_{k_j-1} \geq \gamma_\mu^p \mu_{k_{j-1}-1}$ for all sufficiently large $j$, then there exists an infinite ordered set $\Jscr \subseteq \mathbb{N}$ such that any limit point of either 
$\{(x_{k_j},\yhat_{k_j})\}_{j\in\Jscr}$ 
or
$\{(x_{k_j},y_{k_j-1})\}_{j\in\Jscr}$
is first-order stationary for \eqref{nep}.
  \end{itemize}
\end{theorem}

Observe that the conclusions are exactly the same as in \cite[Theorem~3.14]{curtis12}.  We also call the readers attention to the comments following \cite[Theorem~3.14]{curtis12}, which discuss the consequences of these results.  In particular, these comments suggest how Algorithm~\ref{alg-aal} may be modified to guarantee convergence to first-order stationary points, even in case~(iii) of the theorem.  However, as mentioned in \cite{curtis12}, we do not consider these modifications to the algorithm to have practical benefits. This perspective is supported by the numerical tests presented in the following section.

\section{Numerical Experiments} \label{sec-numerical}

In this section, we provide evidence that steering can have a positive
effect on the performance of AL algorithms.  To best illustrate the
influence of steering, we implemented and tested algorithms in two
pieces of software.  First, in \Matlab{}, we implemented our adaptive AL
line search algorithm, i.e., Algorithm~\ref{alg-aal}, and the adaptive
AL trust region method given as~\cite[Algorithm~4]{curtis12}.  Since
these methods were implemented from scratch, we had control over every
aspect of the code, which allowed us to implement all features described
in this paper and in \cite{curtis12}.  Second, we implemented a simple
modification of the AL trust region algorithm in the \lancelot{}
software package~\cite{ConGT92a}.  Our only modification to \lancelot{}
was to incorporate a basic form of steering; i.e., we did not change
other aspects of \lancelot{}, such as the mechanisms for triggering a
multiplier update.  In this manner, we were also able to isolate the
effect that steering had on numerical performance, though it should be
noted that there were differences between Algorithm~\ref{alg-aal} and
our implemented algorithm in \lancelot{} in terms of, e.g., the
multiplier updates.

While we provide an extensive amount of information about the results of our experiments in this section, further information can be found in~\iftoggle{arxivpaper}{Appendix~\ref{app-results}.}{\cite[Appendix~\ref{app-results}]{CJJR-arxiv}.}

\subsection{\Matlab{} implementation} \label{sec:matlab}

\subsubsection{Implementation details} \label{sec-implement-matlab}

Our \Matlab{} software was comprised of six algorithm variants.  The
algorithms were implemented as part of the same package so that most of
the algorithmic components were exactly the same; the primary
differences related to the step acceptance mechanisms and the manner in
which the Lagrange multiplier estimates and penalty parameter were
updated.  First, for comparison against algorithms that utilized our
steering mechanism, we implemented line search and trust region variants
of a basic augmented Lagrangian method, given as
\cite[Algorithm~1]{curtis12}.  We refer to these algorithms as
\balls\ (\textbf{b}asic \textbf{a}ugmented \textbf{L}agrangian,
\textbf{l}ine \textbf{s}earch) and \baltr\ (\textbf{t}rust
\textbf{r}egion), respectively.  These algorithms clearly differed in
that one used a line search and the other used a trust region strategy
for step acceptance, but the other difference was that, like
Algorithm~\ref{alg-aal} in this paper, \balls\ employed a convexified
model of the AL function.  (We discuss more details about the use of
this convexified model below.)  The other algorithms implemented in our
software included two variants of Algorithm~\ref{alg-aal} and two
variants of~\cite[Algorithm~4]{curtis12}.  The first variants of each,
which we refer to as \aalls\ and \aaltr\ (\textbf{a}daptive, as opposed
to \textbf{b}asic), were straightforward implementations of these algorithms,
whereas the latter variants, which we refer to as \aallssafe\ and
\aaltrsafe, included an implementation of a safeguarding procedure for
the steering mechanism.  The safeguarding procedure will be described in
detail shortly.

The main per-iteration computational expense for each algorithm variant
can be attributed to the search direction computations.  For computing a
search direction via an approximate solve of \eqref{subprob-al} or
\cite[Prob.~(3.8)]{curtis12}, all algorithms essentially used the same
procedure.  For simplicity, all algorithms considered variants of these
subproblems in which the $\ell_2$-norm trust region was replaced by an
$\ell_\infty$-norm trust region so that the subproblems were
bound-constrained.  (The same modification was used in the Cauchy step
calculations.)  Then, starting with the Cauchy step as the initial solution estimate and defining the initial working set by the bounds
identified as active by the Cauchy step, a projected conjugate gradient
(PCG) method was used to compute an improved solution.  During the PCG
routine, if a new trial solution violated a bound constraint that was
not already part of the working set, then this bound was added to the
working set and the PCG routine was reinitialized.  By contrast,
if the reduced subproblem (corresponding to the current working set) was
solved sufficiently accurately, then a check for termination was
performed.  In particular, multiplier estimates were computed for the
working set elements.  If these multiplier estimates were all
nonnegative (or at least larger than a small negative number), then the subproblem was
deemed to be solved and the routine terminated.  Otherwise, an element
corresponding to the most negative multiplier estimate was removed from
the working set and the PCG routine was reinitialized.  We do not claim
that the precise manner in which we implemented this approach guaranteed
convergence to an exact solution of the subproblem.  However, the
approach just described was based on well-established methods for
solving bound-constrained quadratic optimization problems (QPs), and we
found that it worked very well in our experiments.  It should be noted that if, at any time, negative curvature was encountered in the PCG routine, then the solver terminated with the current PCG iterate.  In this manner, the solutions were generally less accurate when negative curvature was encountered, but we claim that this did not have too adverse an effect on the performance of any of the algorithms.

A few additional comments are necessary to describe our search direction
computation procedures.  First, it should be noted that for the line
search algorithms, the Cauchy step calculation in
Algorithm~\ref{alg-sCk} was performed with \eqref{c2} as stated (i.e.,
with $\tilde{q}$), but the above PCG routine to compute the search
direction was applied to \eqref{subprob-al} \emph{without} the
convexification for the quadratic term.  However, we claim that this
choice remains consistent with the stated algorithms since, for all
algorithm variants, we performed a sanity check after the computation of
the search direction.  In particular, the reduction in the model of the
AL function yielded by the search direction was compared against that
yielded by the corresponding Cauchy step.  If the Cauchy step actually
provided a better reduction in the model, then the computed search
direction was replaced by the Cauchy step.  In this sanity check for the
line search algorithms, we computed the model reductions \emph{with} the
convexification of the quadratic term (i.e., with $\tilde{q}$), which
implies that, overall, our implemented algorithm guaranteed Cauchy
decrease in the appropriate model for all algorithms.  Second, we remark
that for the algorithms that employed a steering mechanism, we did not
employ the same procedure to approximately solve \eqref{prob-normal} or
\cite[Prob.~(3.4)]{curtis12}.  Instead, we simply used the Cauchy steps
as approximate solutions of these subproblems.  Finally, we note that in
the steering mechanism, we checked condition~\eqref{cond2b} with the
Cauchy steps for each subproblem, despite the fact that the search
direction was computed as a more accurate solution of \eqref{subprob-al}
or \cite[Prob.~(3.8)]{curtis12}.  This had the effect that the
algorithms were able to modify the penalty parameter via the steering
mechanism prior to computing the search direction; only Cauchy steps for
the subproblems were needed for steering.

Most of the other algorithmic components were implemented similarly to
the algorithm in \cite{curtis12}.  As an example, for the computation of
the estimates $\{\widehat{y}_{k+1}\}$ (which are required to
satisfy~\eqref{new-y}), we checked whether
$\twonorm{F\sub{L}(x_{k+1},\pi(x_{k+1},y_k,\mu_k))} \leq
\twonorm{F\sub{L}(x_{k+1},y_k)}$; if so, then we set $\widehat{y}_{k+1}
\gets \pi(x_{k+1},y_k,\mu_k)$, and otherwise we set $\widehat{y}_{k+1}
\gets y_k$.  Furthermore, for prescribed tolerances
$\{\kappa\sub{opt},\kappa\sub{feas},\mu\sub{min}\} \subset (0,\infty)$,
we terminated an algorithm with a declaration that a stationary point
was found if
\begin{equation}\label{terminate1}
  \|F\sub{L}(x_k,y_k)\|_\infty \leq \kappa\sub{opt}\ \ \text{and}\ \ \|c_k\|_\infty \leq \kappa\sub{feas},
\end{equation}
and terminated with a declaration that an infeasible stationary point was found if
\begin{equation}\label{terminate2}
  \|F\sub{FEAS}(x_k)\|_\infty \leq \kappa\sub{opt},\ \ \|c_k\|_\infty > \kappa\sub{feas},\ \ \text{and}\ \ \mu_k \leq \mu\sub{min}.
\end{equation}
As in \cite{curtis12}, this latter set of conditions shows that we did
not declare that an infeasible stationary point was found unless the
penalty parameter had already been reduced below a prescribed tolerance.
This helps in avoiding premature termination when the algorithm could
otherwise continue and potentially find a point
satisfying~\eqref{terminate1}, which was always the preferred outcome.
Each algorithm terminated with a message of failure if neither
\eqref{terminate1} nor \eqref{terminate2} was satisfied within
$k\sub{max}$ iterations.  It should also be noted that the problems were
pre-scaled so that the $\ell_\infty$-norms of the gradients of the problem functions at
the initial point would be less than or equal to a prescribed constant
$G > 0$.  The values for all of these parameters, as well as other input
parameter required in the code, are summarized in
Table~\ref{tab.params}.  (Values for parameters related to updating the
trust region radii required by~\cite[Algorithm~4]{curtis12} were set as
in \cite{curtis12}.)

\begin{table}[h]
  \caption{Input parameter values used in our \Matlab{} software.}
  \label{tab.params}
  \centering
  \begin{tabular}{|cc|cc|cc|cc|}
    \hline
    \strt Parameter & Value & Parameter & Value & Parameter & Value & Parameter & Value \\
    \hline
    \strt$\gamma$     & $0.5$ & $\kappa_2$       & $1$       & $\etaverysuccessful$ & $0.9$     & $\mu\sub{min}$ & $10^{-8}$ \\
    \strt$\gamma_\mu$ & $0.1$ & $\kappa_3$       & $10^{-4}$ & $\epsilon$           & $0.5$     & $k\sub{max}$   & $10^4$    \\
    \strt$\gamma_t$   & $0.1$ & $\eps_r$         & $10^{-4}$ & $\mu_0$              & $1$       & $G$            & $10^2$    \\
    \strt$\gamma_T$   & $0.1$ & $\kappa_t$       & $0.9$     & $\kappa\sub{opt}$    & $10^{-5}$ &                &           \\
    \strt$\kappa_1$   & $1$   & $\etasuccessful$ & $10^{-4}$ & $\kappa\sub{feas}$   & $10^{-5}$ &                &           \\
    \hline                                                                                                            
  \end{tabular}
\end{table}

We close this subsection with a discussion of some additional differences between the algorithms as stated in this paper and in \cite{curtis12} and those implemented in our software.  We claim that none of these differences represents a significant departure from the stated algorithms; we merely made some adjustments to simplify the implementation and to incorporate features that we found to work well in our experiments.  First, while all algorithms use the input parameter $\gamma_\mu$ given in Table~\ref{tab.params} for decreasing the penalty parameter, we decrease the penalty parameter less significantly in the steering mechanism.  In particular, in line~\ref{steering-decrease} of Algorithm~\ref{alg-aal} and line~20 of~\cite[Algorithm~4]{curtis12}, we replace $\gamma_\mu$ with $0.7$.  Second, in the line search algorithms, rather than set the trust region radii as in \eqref{sandy} and \eqref{sandy-AL} where $\delta$ appears as a constant value, we defined a dynamic sequence, call it $\{\delta_k\}$, that depended on the step-size sequence $\{\alpha_k\}$.  In this manner, $\delta_k$ replaced $\delta$ in \eqref{sandy} and \eqref{sandy-AL} for all $k$.  We initialized $\delta_0 \gets 1$.  Then, for all $k$, if $\alpha_k = 1$, then we set $\delta_{k+1} \gets \tfrac53\delta_k$, and if $\alpha_k < 1$, then we set $\delta_{k+1} \gets \tfrac12\delta_k$.  Third, to simplify our implementation, we effectively ignored the imposed bounds on the multiplier estimates by setting $Y \gets \infty$ and $Y_1 \gets \infty$.  This choice implies that we always chose $\alpha_y \gets 1$ in \eqref{y-combo}.  Fourth, we initialized the target values as
\begin{align}
              t_1 &\gets \max\{10^2,\min\{10^4,\|c_k\|_\infty\}\} \\
  \words{and} T_1 &\gets \max\{10^0,\min\{10^2,\|F\sub{L}(x_k,y_k)\|_\infty\}\}.
\end{align}
Finally, in \aallssafe\ and \aaltrsafe, we safeguard the steering procedure by shutting it off whenever the penalty parameter was smaller than a prescribed tolerance.  Specifically, we considered the \textbf{while} condition in line~\ref{while-2} of Algorithm~\ref{alg-aal} and line~19 of \cite[Algorithm~4]{curtis12} to be satisfied whenever $\mu_k \leq 10^{-4}$.

\subsubsection{Results on \CUTEst{} test problems} \label{sec:matlab-cutest}

We tested our \Matlab{} algorithms on the subset of problems from the \CUTEst{}~\cite{gould2013cutest} collection that have at least one general constraint and at most $1000$ variables and $1000$ constraints.  This set contains $383$ test problems.  However, the results that we present in this section are only for those problems for which at least one of our six solvers obtained a successful result, i.e., where \eqref{terminate1} or \eqref{terminate2} was satisfied.  This led to a set of $323$ problems that are represented in the numerical results in this section.

To illustrate the performance of our \Matlab{} software, we use performance profiles as introduced by Dolan and Mor\'{e}~\cite{DolM02} to provide a visual comparison of different measures of performance.  Consider a performance profile that measures performance in terms of required iterations until termination.  For such a profile, if the graph associated with an algorithm passes through the point $(\alpha,0.\beta)$, then, on $\beta\%$ of the problems, the number of iterations required by the algorithm was less than $2^\alpha$ times the number of iterations required by the algorithm that required the fewest number of iterations.  At the extremes of the graph, an algorithm with a higher value on the vertical axis may be considered a more efficient algorithm, whereas an algorithm on top at the far right of the graph may be considered more reliable.  Since, for most problems, comparing values in the performance profiles for large values of $\alpha$ is not enlightening, we truncated the horizontal axis at 16 and simply remark on the numbers of failures for each algorithm.

Figures~\ref{fig:ppAALIter_LS} and \ref{fig:ppAALFunc_LS} show the results for the three line search variants, namely \balls, \aalls, and \aallssafe.  The numbers of failures for these algorithms were 25, 3, and 16, respectively.  The same conclusion may be drawn from both profiles: the steering variants (with and without safeguarding) were both more efficient and more reliable than the basic algorithm, where efficiency is measured by either the number of iterations (Figure~\ref{fig:ppAALIter_LS}) or the number of function evaluations (Figure~\ref{fig:ppAALFunc_LS}) required.  We display the profile for the number of function evaluations required since, for a line search algorithm, this value is always at least as large as the number of iterations, and will be strictly greater whenever backtracking is required to satisfy \eqref{dec-L-ls} (yielding $\alpha_k < 1$).  From these profiles, one may observe that unrestricted steering (in \aalls) yielded superior performance to restricted steering (in \aallssafe) in terms of both efficiency and reliability; this suggests that safeguarding the steering mechanism may diminish its potential benefits.
    
\begin{figure}[ht]
  \begin{minipage}[b]{0.45\linewidth}
    \centering
    \includegraphics[scale=0.40,clip=true,trim=80 60 300 240]{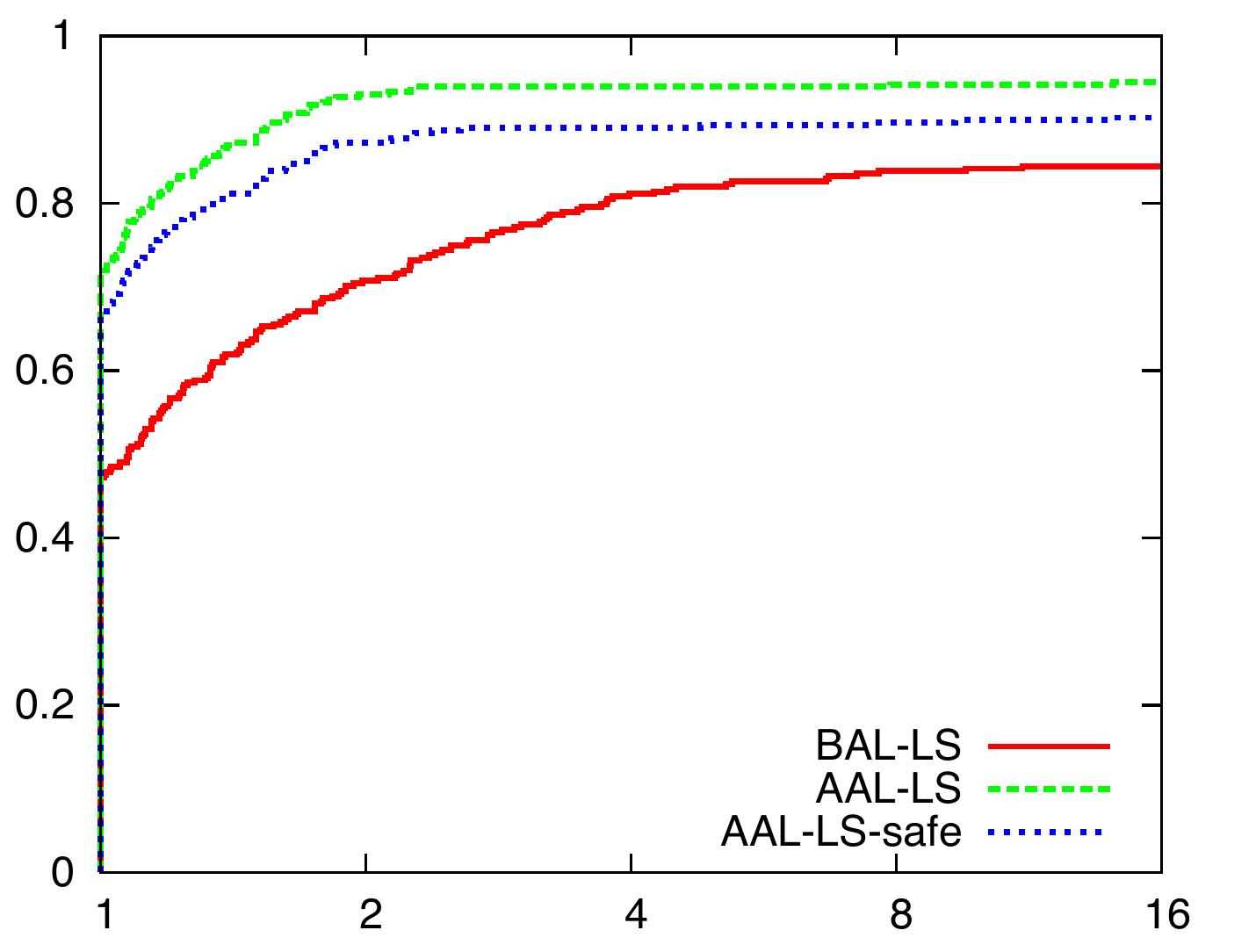}
    \caption{Performance profile for iterations: line search algorithms on the \CUTEst{} set.}
    \label{fig:ppAALIter_LS}
  \end{minipage}
  \hspace{1.0cm}
  \begin{minipage}[b]{0.45\linewidth}
    \centering
    \includegraphics[scale=0.40,clip=true,trim=80 60 300 240]{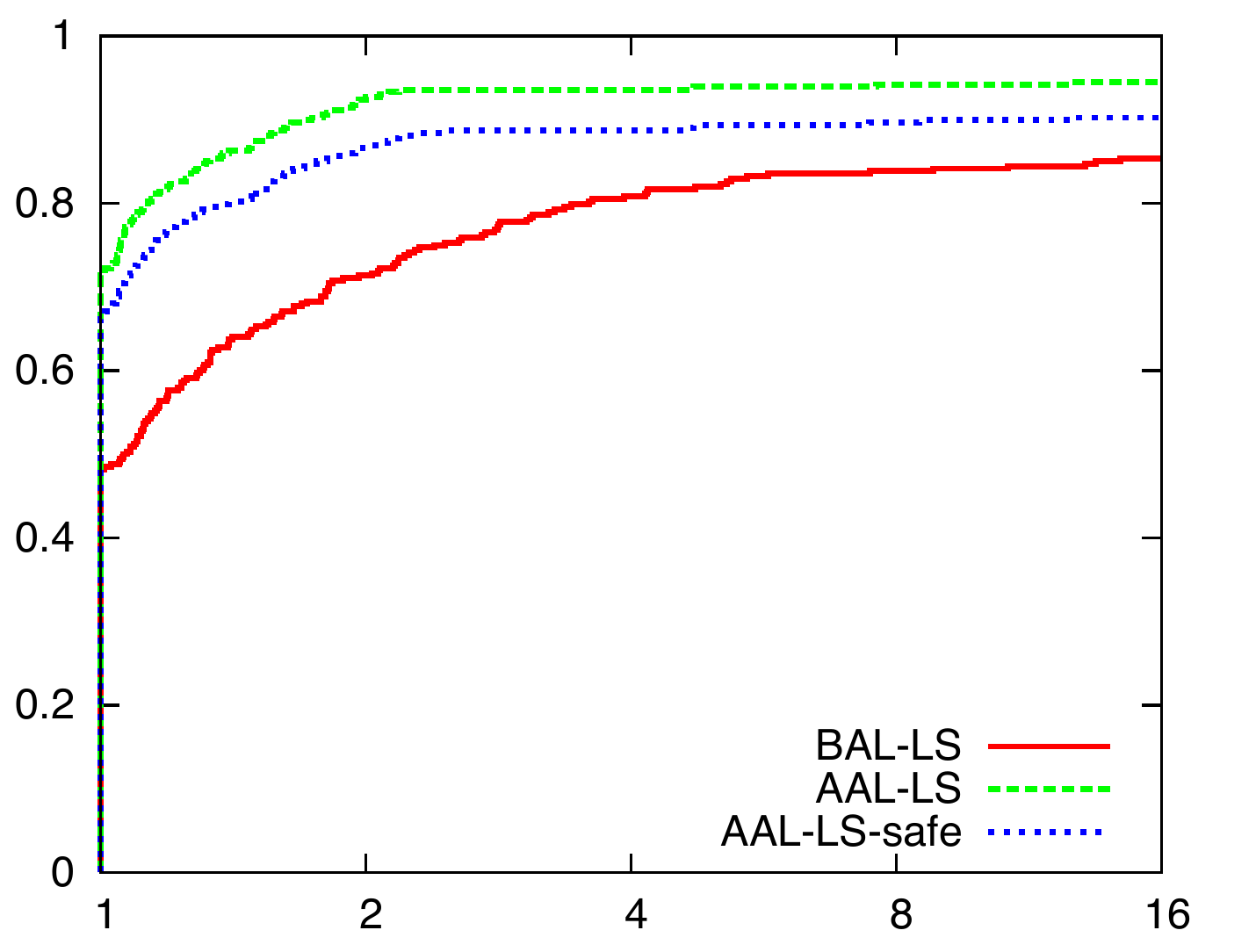}
    \caption{Performance profile for function evaluations: line search algorithms on the \CUTEst{} set.}
    \label{fig:ppAALFunc_LS}
  \end{minipage}
\end{figure}

Figures~\ref{fig:ppAALIter_TR} and \ref{fig:ppAALFunc_TR} show the results for the three trust region variants, namely \baltr, \aaltr, and \aaltrsafe, the numbers of failures for which were 30, 12, and 20, respectively.  Again, as for the line search algorithms, the same conclusion may be drawn from both profiles: the steering variants (with and without safeguarding) are both more efficient and more reliable than the basic algorithm, where now we measure efficiency by either the number of iterations (Figure~\ref{fig:ppAALIter_TR}) or the number of gradient evaluations (Figure~\ref{fig:ppAALFunc_TR}) required before termination.  We observe the number of gradient evaluations here (as opposed to the number of function evaluations) since, for a trust region algorithm, this value is never larger than the number of iterations, and will be strictly smaller whenever a step is rejected and the trust-region radius is decreased because of insufficient decrease in the AL function.  These profiles also support the other observation that was made by the results for our line search algorithms, i.e., that unrestricted steering may be superior to restricted steering in terms of efficiency and reliability.

\begin{figure}[ht]
  \begin{minipage}[b]{0.45\linewidth}
    \centering
    \includegraphics[scale=0.40,clip=true,trim=80 60 300 240]{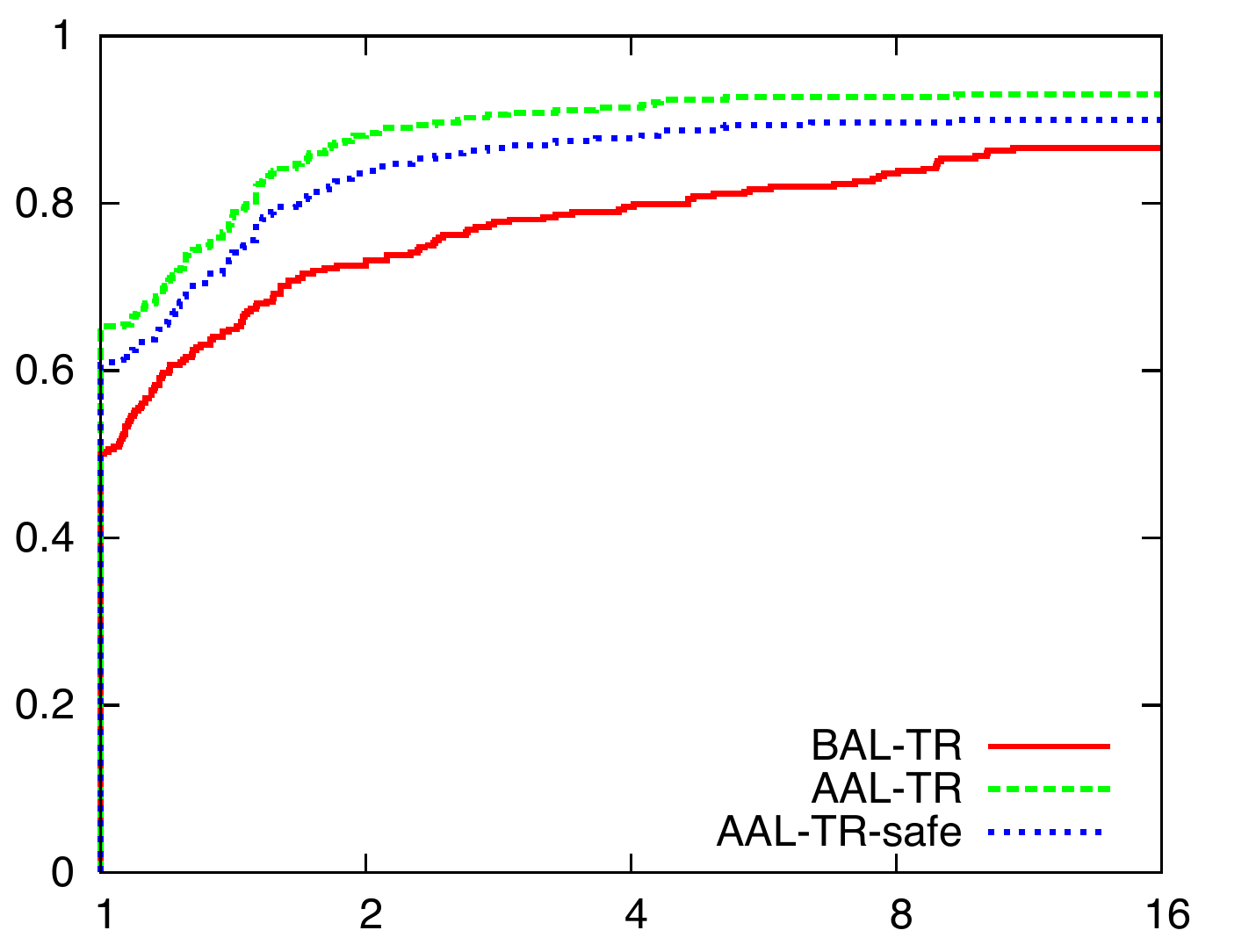}
    \caption{Performance profile for iterations: trust region algorithms on the \CUTEst{} set.}
    \label{fig:ppAALIter_TR}
  \end{minipage}
  \hspace{1.0cm}
  \begin{minipage}[b]{0.45\linewidth}
    \centering
    \includegraphics[scale=0.40,clip=true,trim=80 60 300 240]{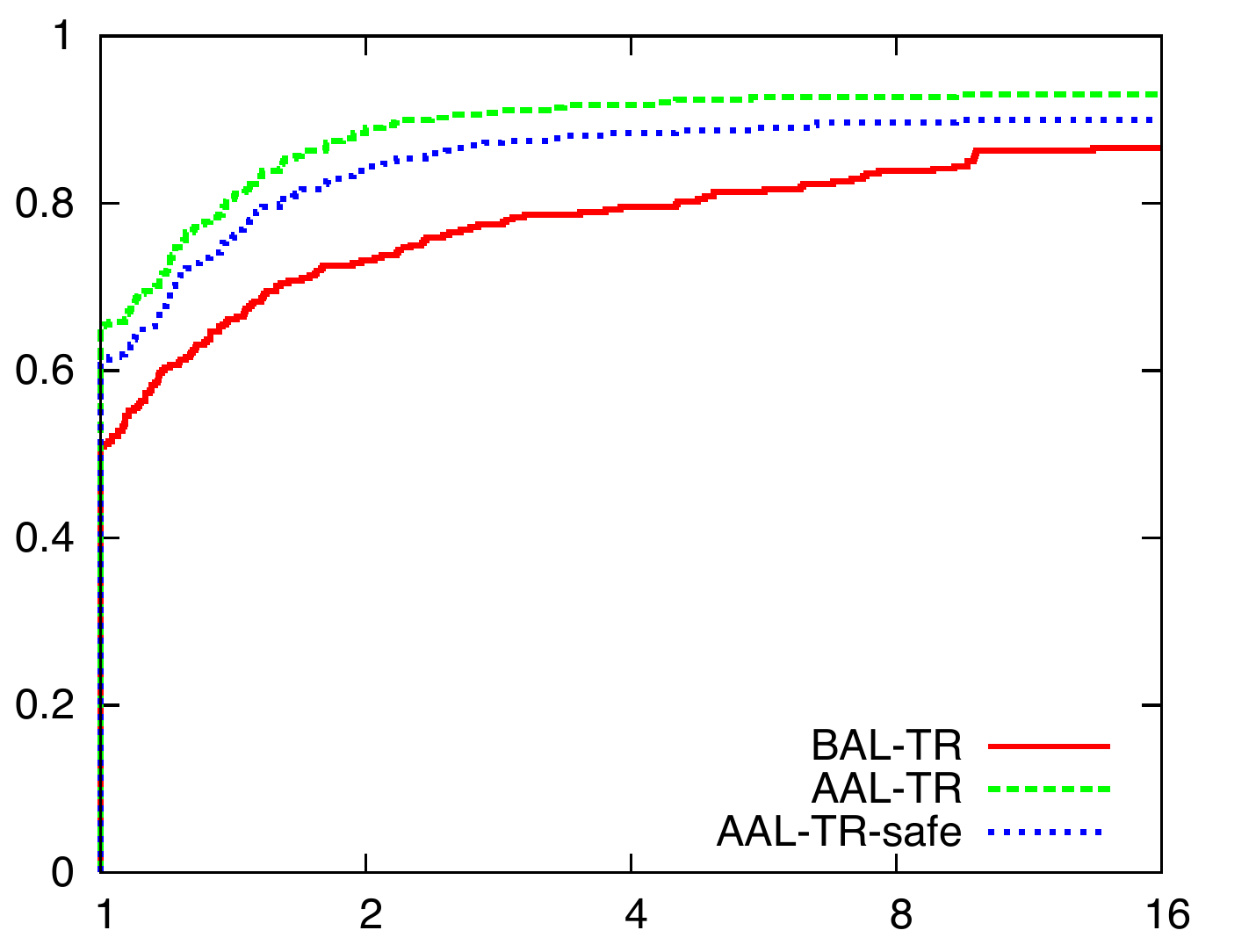}
     \caption{Performance profile for gradient evaluations: trust region algorithms on the \CUTEst{} set.}  
    \label{fig:ppAALFunc_TR}
  \end{minipage}
\end{figure}

The performance profiles in Figures~\ref{fig:ppAALIter_LS}--\ref{fig:ppAALFunc_TR} suggest that steering has practical benefits, and that safeguarding the procedure may limit its potential benefits.  However, to be more confident in these claims, one should observe the final penalty parameter values typically produced by the algorithms.  These observations are important since one may be concerned whether the algorithms that employ steering yield final penalty parameter values that are often significantly smaller than those yielded by basic AL algorithms.  To investigate this possibility in our experiments, we collected the final penalty parameter values produced by all six algorithms; the results are in Table~\ref{tab-muvals-AAL}.  The column titled $\mu\sub{final}$ gives a range for the final value of the penalty parameter.  (For example, the value 27 in the $\balls$ column indicates that the final penalty parameter value computed by our basic line search AL algorithm fell in the range $[10^{-2},10^{-1})$ for 27 of the problems.)

\begin{table}[ht]
\centering
\caption{Numbers of \CUTEst{} problems for which the final penalty parameter values were in the given ranges.}
\label{tab-muvals-AAL}
\begin{tabular}{|lcccccc|}
  \hline
  \strt $\mu\sub{final}$     & \balls & \aalls & \aallssafe & \baltr & \aaltr & \aaltrsafe \\ \hline 
  \strt $1$                  & 139    & 87     & 87         & 156    & 90     & 90         \\ 
  \strt $[10^{-1},1)$        & 43     & 33     & 33         & 35     & 46     & 46         \\ 
  \strt $[10^{-2},10^{-1})$  & 27     & 37     & 37         & 28     & 29     & 29         \\ 
  \strt $[10^{-3}, 10^{-2})$ & 17     & 42     & 42         & 19     & 49     & 49         \\ 
  \strt $[10^{-4}, 10^{-3})$ & 22     & 36     & 36         & 18     & 29     & 29         \\ 
  \strt $[10^{-5}, 10^{-4})$ & 19     & 28     & 42         & 19     & 25     & 39         \\ 
  \strt $[10^{-6}, 10^{-5})$ & 15     & 19     & 11         & 9      & 11     & 9          \\ 
  \strt $(0, 10^{-6})$       & 46     & 46     & 40         & 44     & 49     & 37         \\ 
  \hline
\end{tabular}
\end{table}

We remark on two observations about the data in Table~\ref{tab-muvals-AAL}.  First, as may be expected, the algorithms that employ steering typically reduce the penalty parameter below its initial value on some problems on which the other algorithms do not reduce it at all.  This, in itself, is not a major concern, since a reasonable reduction in the penalty parameter may cause an algorithm to locate a stationary point more quickly.  Second, we remark that the number of problems for which the final penalty parameter was very small (say, less than $10^{-4}$) was similar for all algorithms, even those that employed steering.  This suggests that while steering was able to aid in guiding the algorithms toward constraint satisfaction, the algorithms did not reduce the value to such a small value that feasibility became the only priority.  Overall, our conclusion from Table~\ref{tab-muvals-AAL} is that steering typically decreases the penalty parameter more than does a traditonal updating scheme, but one should not expect that the final penalty parameter value will be reduced unnecessarily small due to steering; rather, steering can have the intended benefit of improving efficiency and reliability by guiding a method toward constraint satisfaction more quickly.

\subsubsection{Results on \COPS{} test problems} \label{sec:cops}

We also tested our \Matlab{} software on the large-scale constrained problems available in the \COPS{}~\cite{BonDM98} collection.  This test set was designed to provide difficult test cases for nonlinear optimization software; the problems include examples from fluid dynamics, population dynamics, optimal design, mesh smoothing, and optimal control.  For our purposes, we solved the smallest versions of the \AMPL{} models \cite{AMPL,AMPL2} provided in the collection.  All of our solvers failed to solve the problems named
{\it chain}, 
{\it dirichlet}, 
{\it henon}, 
{\it lane\_emden}, 
and
{\it robot1}, 
so these problems were excluded.  The remaining set consisted of the following 17 problems:
{\it bearing}, 
{\it camshape}, 
{\it catmix}, 
{\it channel}, 
{\it elec}, 
{\it gasoil}, 
{\it glider}, 
{\it marine}, 
{\it methanol}, 
{\it minsurf}, 
{\it pinene}, 
{\it polygon}, 
{\it rocket}, 
{\it steering}, 
{\it tetra}, 
{\it torsion}, and 
{\it triangle}.
Since the size of this test set is relatively small, we have decided to display pair-wise comparisons of algorithms in the manner suggested in~\cite{morales2002numerical}.  That is, for a performance measure of interest (e.g., number of iterations required until termination), we compare solvers, call them $A$ and $B$, on problem $j$ with the logarithmic \emph{outperforming factor}
\begin{equation}
	r^j_{AB} := -\log_2 ({\textrm{m}_A^j}/{\textrm{m}_B^j}),\ \ \text{where}\ \ \begin{cases} \textrm{m}_A^j & \text{is the measure for $A$ on problem $j$} \\ \textrm{m}_B^j & \text{is the measure for $B$ on problem $j$.} \end{cases}
\end{equation}
Therefore, if the measure of interest is iterations required, then $r^j_{AB} = p$ would indicate that solver $A$ required $2^{-p}$ the iterations required by solver $B$.  For all plots, we focus our attention on the range $p \in [-2,2]$.

The results of our experiments are given in Figures~\ref{fig:ppAALIter_LS_cops}--\ref{fig:ppAALFunc_TR_cops}.  For the same reasons as discussed in \S\ref{sec:matlab-cutest}, we display results for iterations and function evaluations for the line search algorithms, and display results for iterations and gradient evaluations for the trust region algorithms.  In addition, here we ignore the results for \aallssafe\ and \aaltrsafe\ since, as in the results in \S\ref{sec:matlab-cutest}, we did not see benefits in safeguarding the steering mechanism.  In each figure, a positive (negative) bar indicates that the algorithm whose name appears above (below) the horizontal axis yielded a better value for the measure on a particular problem.  The results are displayed according to the order of the problems listed in the previous paragraph.  In Figures~\ref{fig:ppAALIter_LS_cops} and \ref{fig:ppAALFunc_LS_cops} for the line search algorithms, the red bars for problems {\it catmix} and {\it polygon} indicate that \aalls\ failed on the former and \balls\ failed on the latter; similarly, in Figures~\ref{fig:ppAALIter_TR_cops} and \ref{fig:ppAALFunc_TR_cops} for the trust region algorithms, the red bar for {\it catmix} indicates that \aaltr\ failed on it.

The results in Figures~\ref{fig:ppAALIter_LS_cops} and \ref{fig:ppAALFunc_LS_cops} indicate that \aalls\ more often outperforms \balls\ in terms of iterations and functions evaluations, though the advantage is not overwhelming.  On the other hand, it is clear from Figures~\ref{fig:ppAALIter_TR_cops} and \ref{fig:ppAALFunc_TR_cops} that, despite the one failure, \aaltr{} is generally superior to \baltr{}.  We conclude from these results that steering was beneficial on this test set, especially in terms of the trust region methods.

\begin{figure}[ht]
  \begin{minipage}[b]{0.45\linewidth}
    \centering
    \includegraphics[scale=0.40,clip=true,trim=70 190 10 200]{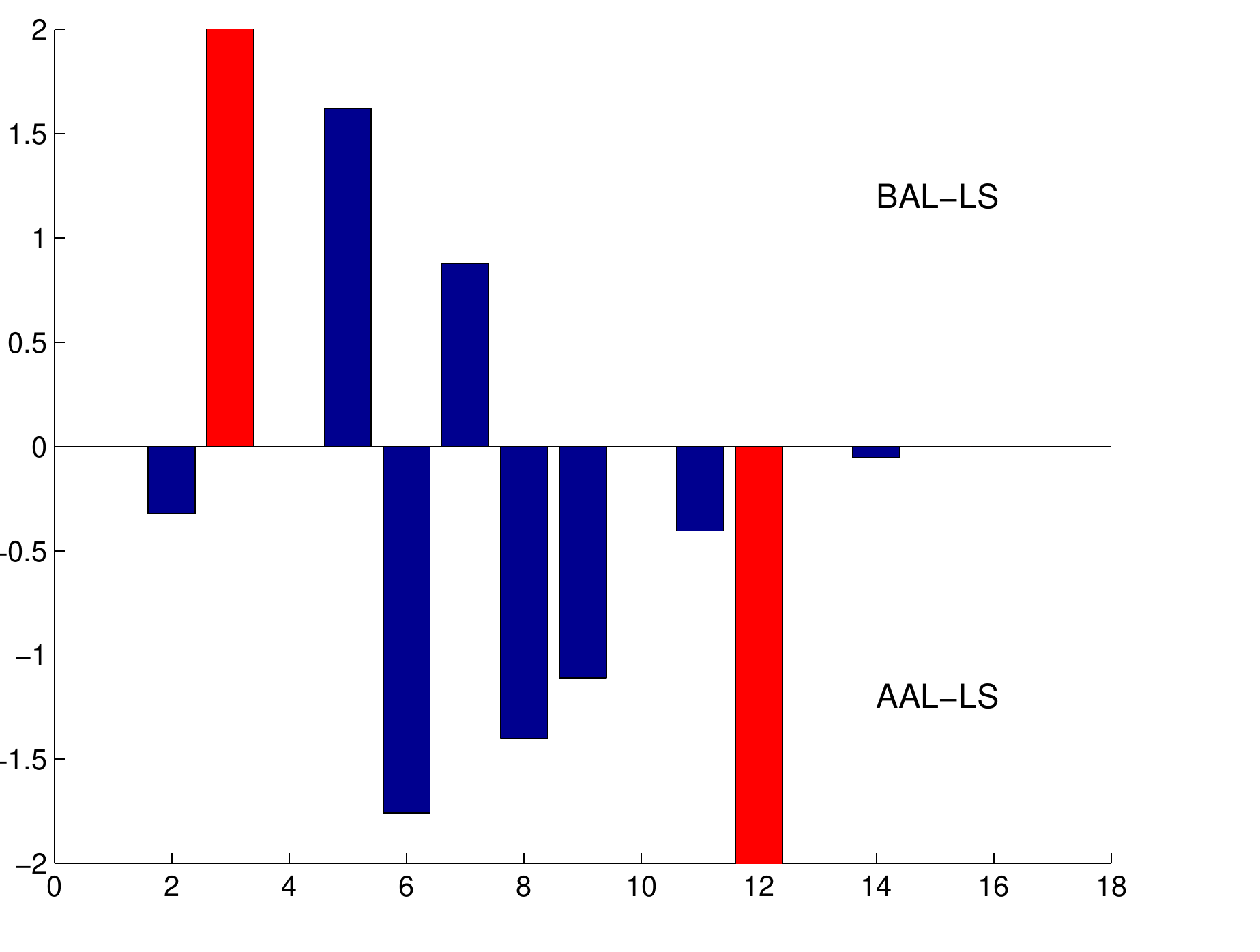}
    \caption{Outperforming factors for iterations: line search algorithms on the \COPS{} set.}
    \label{fig:ppAALIter_LS_cops}
  \end{minipage}
  \hspace{1.0cm}
  \begin{minipage}[b]{0.45\linewidth}
    \centering
     \includegraphics[scale=0.40,clip=true,trim=70 190 10 200]{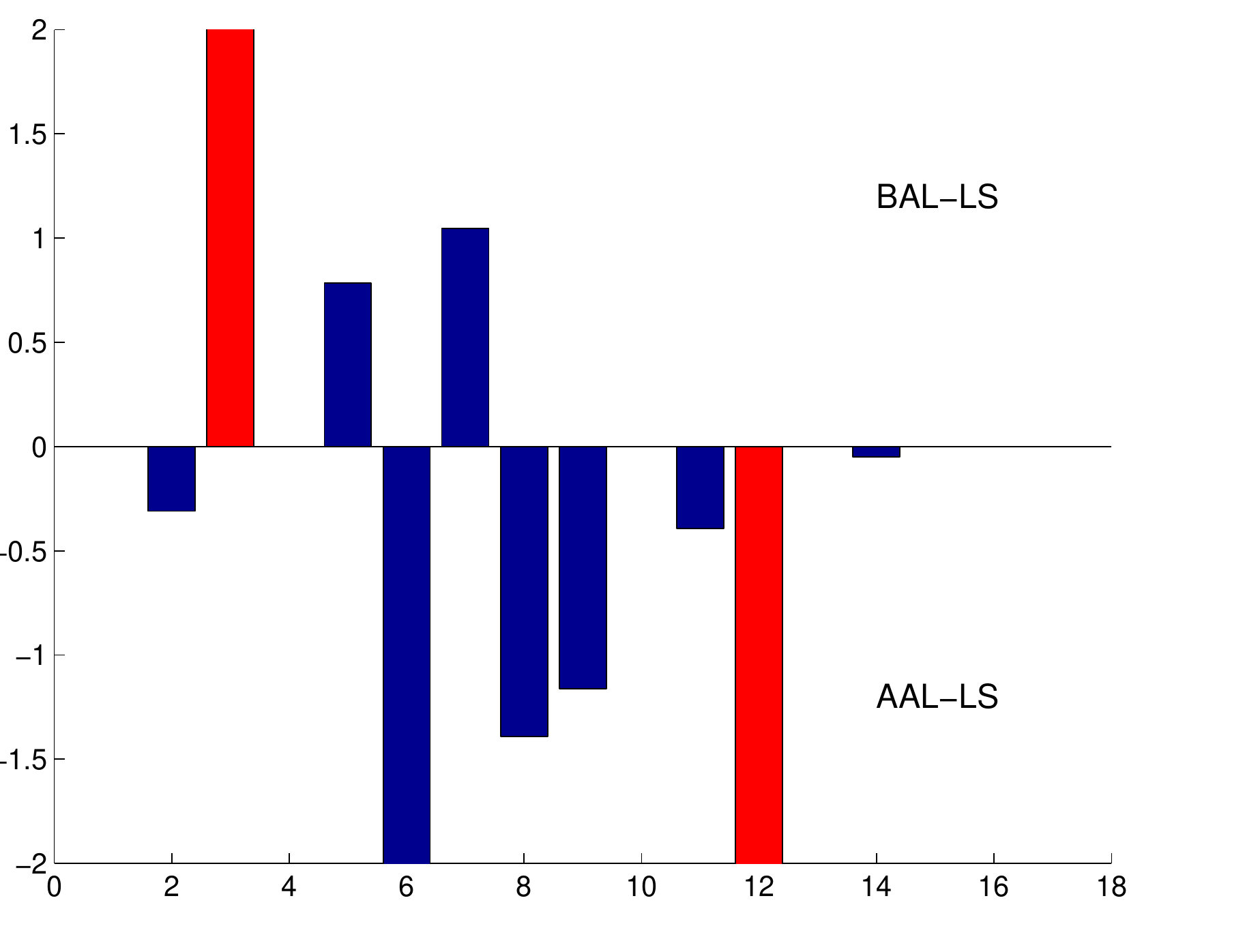}
    \caption{Outperforming factors for function evaluations: line search algorithms on the \COPS{} set.}
    \label{fig:ppAALFunc_LS_cops}
  \end{minipage}
\end{figure}

\begin{figure}[ht]
  \begin{minipage}[b]{0.45\linewidth}
    \centering
    \includegraphics[scale=0.40,clip=true,trim=70 190 10 200]{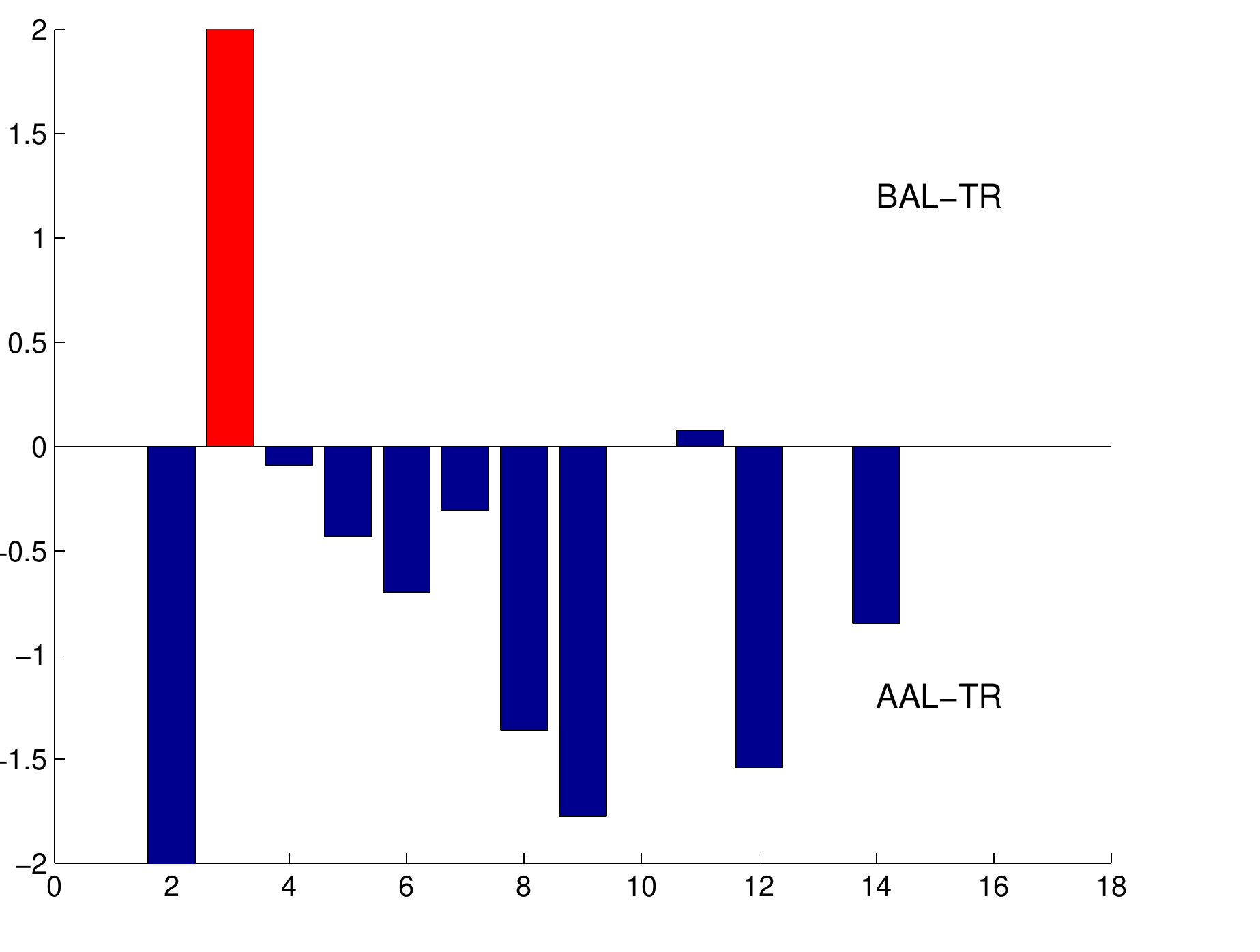}
    \caption{Outperforming factors for iterations: trust region algorithms on the \COPS{} set.}
    \label{fig:ppAALIter_TR_cops}
  \end{minipage}
  \hspace{1.0cm}
  \begin{minipage}[b]{0.45\linewidth}
    \centering
    \includegraphics[scale=0.40,clip=true,trim=70 190 10 200]{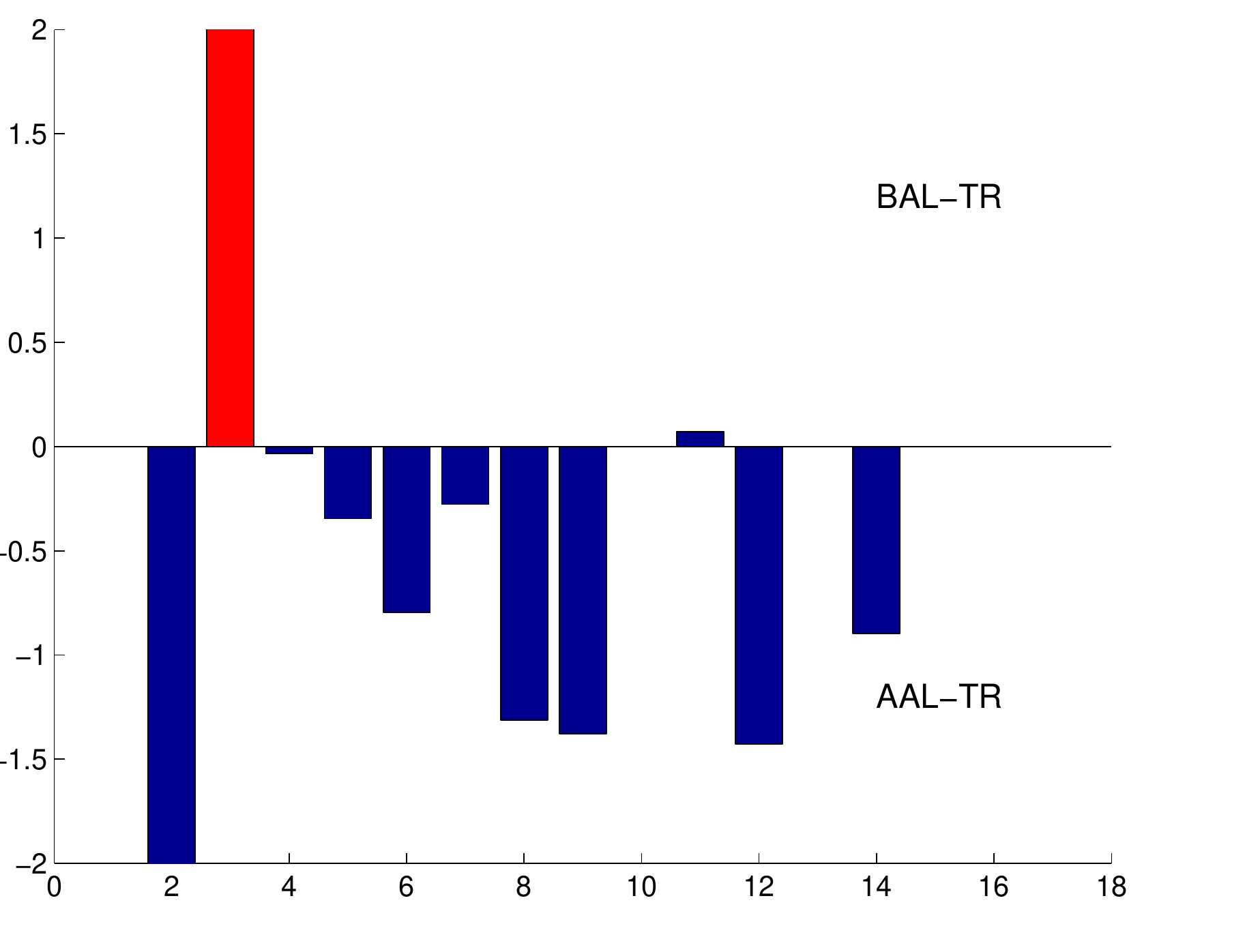}
     \caption{Outperforming factors for gradient evaluations: trust region algorithms on the \COPS{} set.}  
    \label{fig:ppAALFunc_TR_cops}
  \end{minipage}
\end{figure}

\subsubsection{Results on optimal power flow (OPF) test problems} \label{sec:opf}

As a third and final set of experiments for our \Matlab{} software, we tested our algorithms on a collection of
optimal power flow (OPF) problems modeled in \AMPL{} using data sets obtained from {\sc MATPOWER}~\cite{zimmerman2011matpower}.
OPF problems represent a challenging set of nonconvex problems. The active and reactive power flow and the network balance equations give
rise to equality constraints involving nonconvex functions while the inequality constraints are linear and result from placing operating limits on
quantities such as flows, voltages, and various control variables.  The control variables include the voltages
at generator buses and the active-power output of the generating units. The state variables consist of the
voltage magnitudes and angles at each node as well as reactive and active flows in each link. Our test set was comprised of $28$ problems modeled on systems having 14 to 662 nodes from the IEEE test set. In particular, there are
seven IEEE systems, each modeled in four different ways: (i) in Cartesian coordinates; (ii) in polar coordinates;
(iii) with basic approximations to the sin and cos functions in the problem functions; and (iv) with
linearized constraints based on DC power 
flow equations (in place of AC power 
flow). It should be noted
that while linearizing the constraints in formulation (iv) led to a set of linear optimization problems, we still find it interesting
to investigate the possible effect that steering may have in this context. All of the test problems were solved
by all of our algorithm variants.

%

We provide outperforming factors in the same manner as in \S\ref{sec:cops}.  Figures~\ref{fig:ppAALIter_LS_opf} and \ref{fig:ppAALFunc_LS_opf} reveal that \aalls\ typically outperforms \balls\ in terms of both iterations and function evaluations, and Figures~\ref{fig:ppAALIter_TR_opf} and \ref{fig:ppAALFunc_TR_opf} reveal that \aaltr{} more often than not outperforms \baltr\ in terms of iterations and gradient evaluations.  Interestingly, these results suggest more benefits for steering in the line search algorithm than in the trust region algorithm, which is the opposite of that suggested by the results in~\S\ref{sec:cops}.  However, in any case, we believe that we have presented convincing numerical evidence that steering often has an overall beneficial effect on the performance of our \Matlab{} solvers.

\begin{figure}[ht]
  \begin{minipage}[b]{0.45\linewidth}
    \centering
    \includegraphics[scale=0.40,clip=true,trim=70 190 10 200]{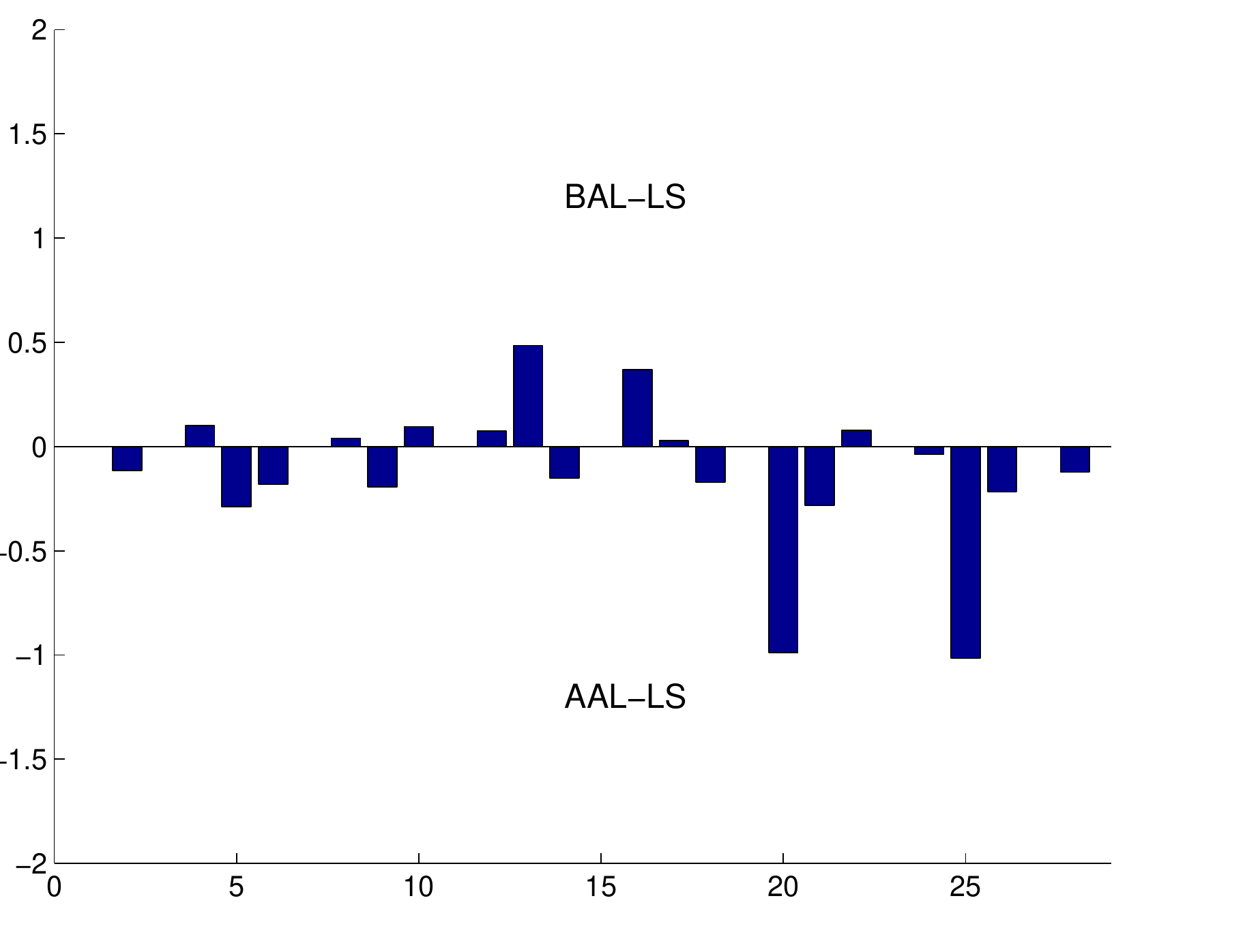}
    \caption{Outperforming factors for iterations: line search algorithms on OPF tests.}
    \label{fig:ppAALIter_LS_opf}
  \end{minipage}
  \hspace{1.0cm}
  \begin{minipage}[b]{0.45\linewidth}
    \centering
    \includegraphics[scale=0.40,clip=true,trim=70 190 10 200]{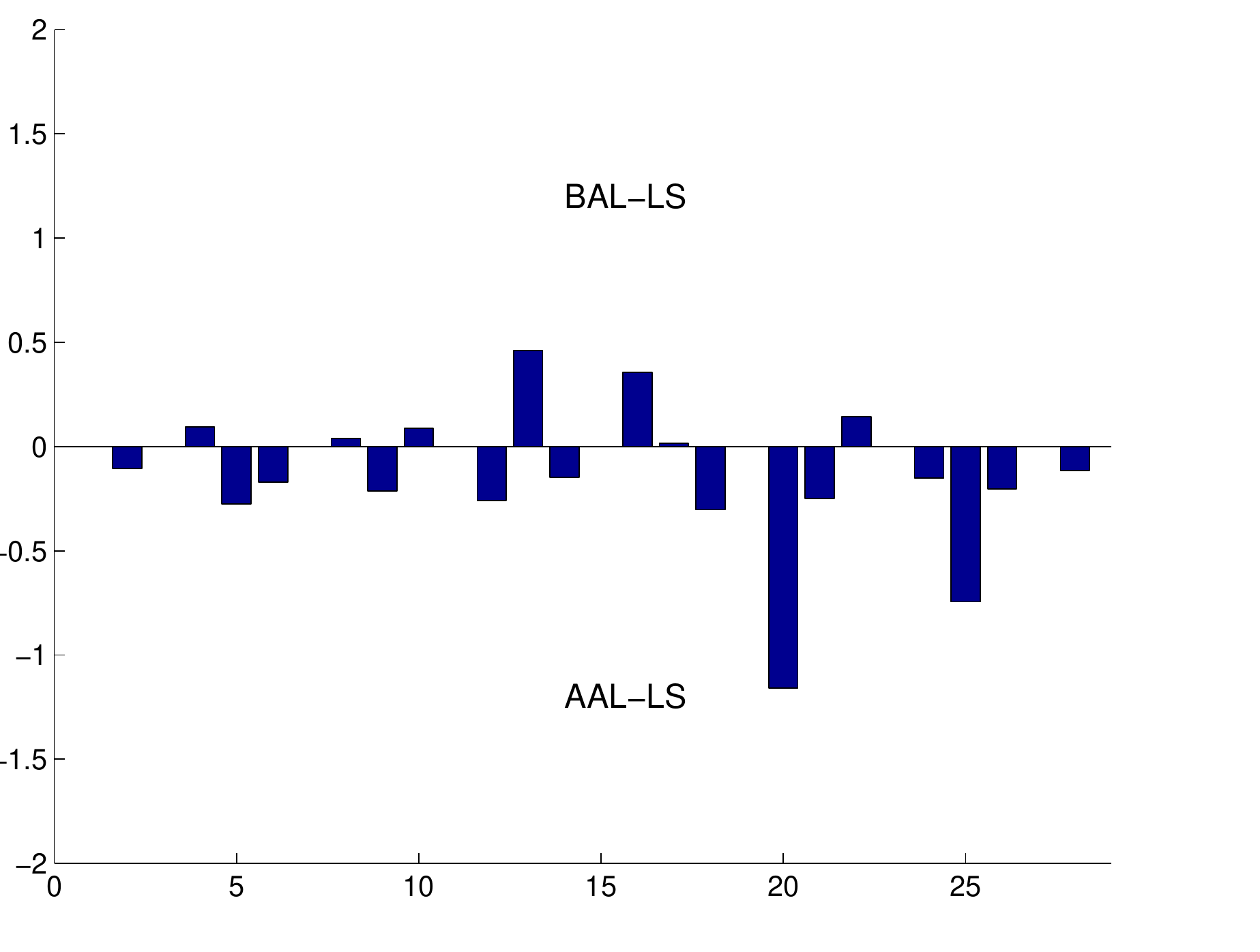}
    \caption{Outperforming factors for function evaluations: line search algorithms on OPF tests.}
    \label{fig:ppAALFunc_LS_opf}
  \end{minipage}
\end{figure}

\begin{figure}[ht]
  \begin{minipage}[b]{0.45\linewidth}
    \centering
    \includegraphics[scale=0.40,clip=true,trim=70 190 10 200]{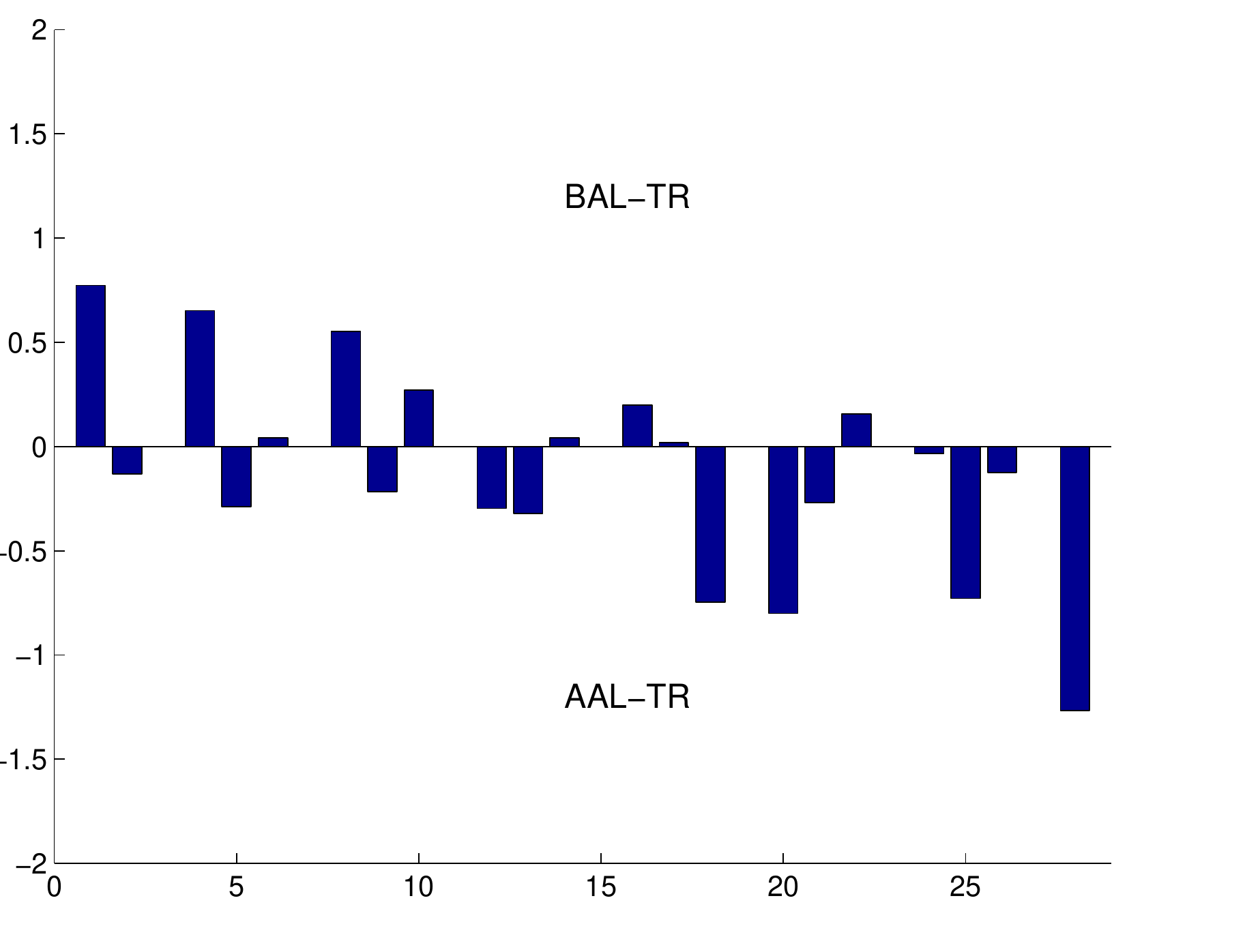}
    \caption{Outperforming factors for iterations: trust region algorithms on OPF tests.}
    \label{fig:ppAALIter_TR_opf}
  \end{minipage}
  \hspace{1.0cm}
  \begin{minipage}[b]{0.45\linewidth}
    \centering
    \includegraphics[scale=0.40,clip=true,trim=70 190 10 200]{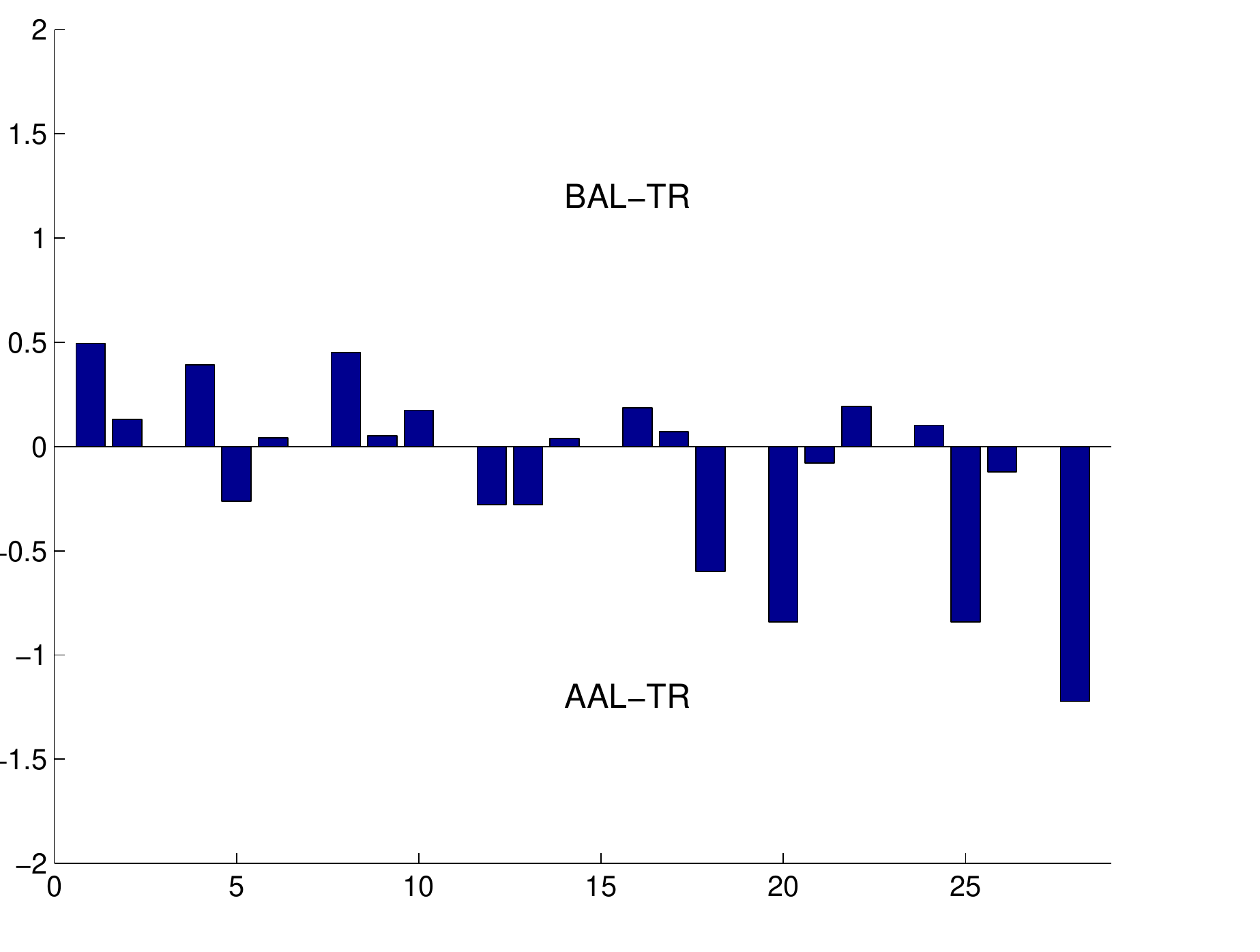}
     \caption{Outperforming factors for gradient evaluations: trust region algorithms on OPF tests.}
    \label{fig:ppAALFunc_TR_opf}
  \end{minipage}
\end{figure}

\subsection{An implementation of \lancelot{} that uses steering}

\subsubsection{Implementation details} \label{sec:Lancelot-imp}

The results for our \Matlab{} software in the previous section illustrate that our adaptive line search AL algorithm and the adaptive trust region AL algorithm from \cite{curtis12} are often more efficient and reliable than basic AL algorithms that employ traditional penalty parameter and Lagrange multiplier updates.  Recall, however, that our adaptive methods are different from their basic counterparts in two key ways.  First, the steering conditions~\eqref{conds} are used to dynamically decrease the penalty parameter during the optimization process for the AL function.  Second, our mechanisms for updating the Lagrange multiplier estimate are different than the basic algorithm outlined in \cite[Algorithm~1]{curtis12} since they use
optimality measures
for both the Lagrangian and the AL functions
(see line~\ref{crunchy} of Algorithm~\ref{alg-aal}) rather than only that for the AL function.  We believe this strategy is more adaptive since it allows for updates to the Lagrange multipliers when the primal estimate is still far from a first-order stationary point for the AL function subject to the bounds. 

In this section, we isolate the effect of the first of these differences by incorporating a steering strategy in the \lancelot{} \cite{ConGT92a,ConnGoulToin96:mp} package that is available in the \galahad\ library \cite{GoulOrbaToin03:toms}.  Specifically, we made three principle enhancements in \lancelot{}.  First, along the lines of the model $q$ in \cite{curtis12} and the convexified model $\tilde{q}$ defined in this paper, we defined the model $\hat{q} : \Re^n \to \Re$ of the AL function given by
\[
  \hat{q}(s;x,y,\mu) = s^T \nabla_x \ell\big(x,y+ c(x)/\mu\big) + \half s^T \big( \nabla_{xx} \ell(x,y) + J(x)\T J(x)/\mu \big) s
\]
as an alternative to the Newton model $q\sub{N} : \Re^n \to \Re$, originally used in \lancelot{},
\[
  q\sub{N}(s;x,y,\mu) = s^T \nabla_x \ell(x,y+ c(x)/\mu) + \half s^T ( \nabla_{xx} \ell(x,y+ c(x)/\mu) + J(x)^T J(x)/\mu ) s.
\]
As in our adaptive algorithms, the purpose of employing such a model was to ensure that $\hat{q} \to q_v$ (pointwise) as $\mu \to 0$, which was required to ensure that our steering procedure was well-defined; see \eqref{models-converge}.  Second, we added routines to compute generalized Cauchy points \cite{ConGT88} for both the constraint violation measure model $q_v$ and $\hat{q}$ during the loop in which $\mu$ was decreased until the steering test \eqref{cond2b} was satisfied; recall the \textbf{while} loop starting on line~\ref{while-2} of Algorithm~\ref{alg-aal}.  Third, we used the value for $\mu$ determined in the steering procedure to compute a generalized Cauchy point for the Newton model $q\sub{N}$, which was the model employed to compute the search direction.  For each of the models just discussed, the generalized Cauchy point was computed using either an efficient sequential search along the piece-wise Cauchy arc \cite{ConnGoulToin88:mc} or via a backtracking Armijo search along the same arc \cite{More88}.  We remark that this third enhancement would not have been needed if the model $\hat{q}$ were used to compute the search directions.  However, in our experiments, it was revealed that using the Newton model typically led to better performance, so the results in this section were obtained using this third enhancement.  In our implementation, the user was allowed to control which model was used via control parameters.  We also added control parameters that allowed the user to restrict the number of times that the penalty parameter may be reduced in the steering procedure in a given iteration, and that disabled steering once the penalty parameter was reduced below a given tolerance (as in the safeguarding procedure implemented in our \Matlab{} software).

The new package was tested with three different control parameter settings.  We refer to algorithm with the first setting, which did not allow any steering to occur, simply as \Lancelot{}.  The second setting allowed steering to be used initially, but turned it off whenever $\mu \leq 10^{-4}$ (as in our safeguarded \Matlab{} algorithms).  We refer to this variant as \LancelotSS{}.  The third setting allowed for steering to be used without any safeguards or restrictions; we refer to this variant as \LancelotS{}.  As in our \Matlab{} software, the penalty parameter was decreased by a factor of $0.7$ until the steering test \eqref{cond2b} was satisfied.  All other control parameters were set to their default \Lancelot\ values.  The new package will be re-branded as \lancelot{} in the next official release, \galahad~2.6.

\galahad\ was compiled with gfortran-4.7 with optimization -O and using Intel MKL BLAS. The code was executed on a single core of an Intel Xeon E5620 (2.4GHz) CPU with 23.5 GiB of RAM.

\subsubsection{Results on \CUTEst{} test problems}

We tested \Lancelot{}, \LancelotS{}, and \LancelotSS{} on the subset of \CUTEst{} problems that have at least one general constraint and at most $10,\!000$ variables and $10,\!000$ constraints.  This amounted to $457$ test problems.  The results are displayed as performance profiles in
Figures~\ref{fig:ppLanIter} and~\ref{fig:ppLanGrad}, which were created from the 364 of these problems that were solved by at least one of the algorithms.  As in the previous sections, since the algorithms are trust region methods, we use the number of iterations and gradient evaluations required as the performance measures of interest.

We can make two important observations from these profiles.  First, it is clear that \LancelotS{} and \LancelotSS{} yielded similar performance in terms of iterations and gradient evaluations, which suggests that safeguarding the steering mechanism is not necessary in practice.  Second, \LancelotS{} and \LancelotSS{} were both more efficient and reliable than~\Lancelot{} on these tests, thus showing the positive influence that steering can have on performance.

\begin{figure}[ht]
\begin{minipage}[b]{0.45\linewidth}
\centering
  \includegraphics[scale=0.40,angle=-90]{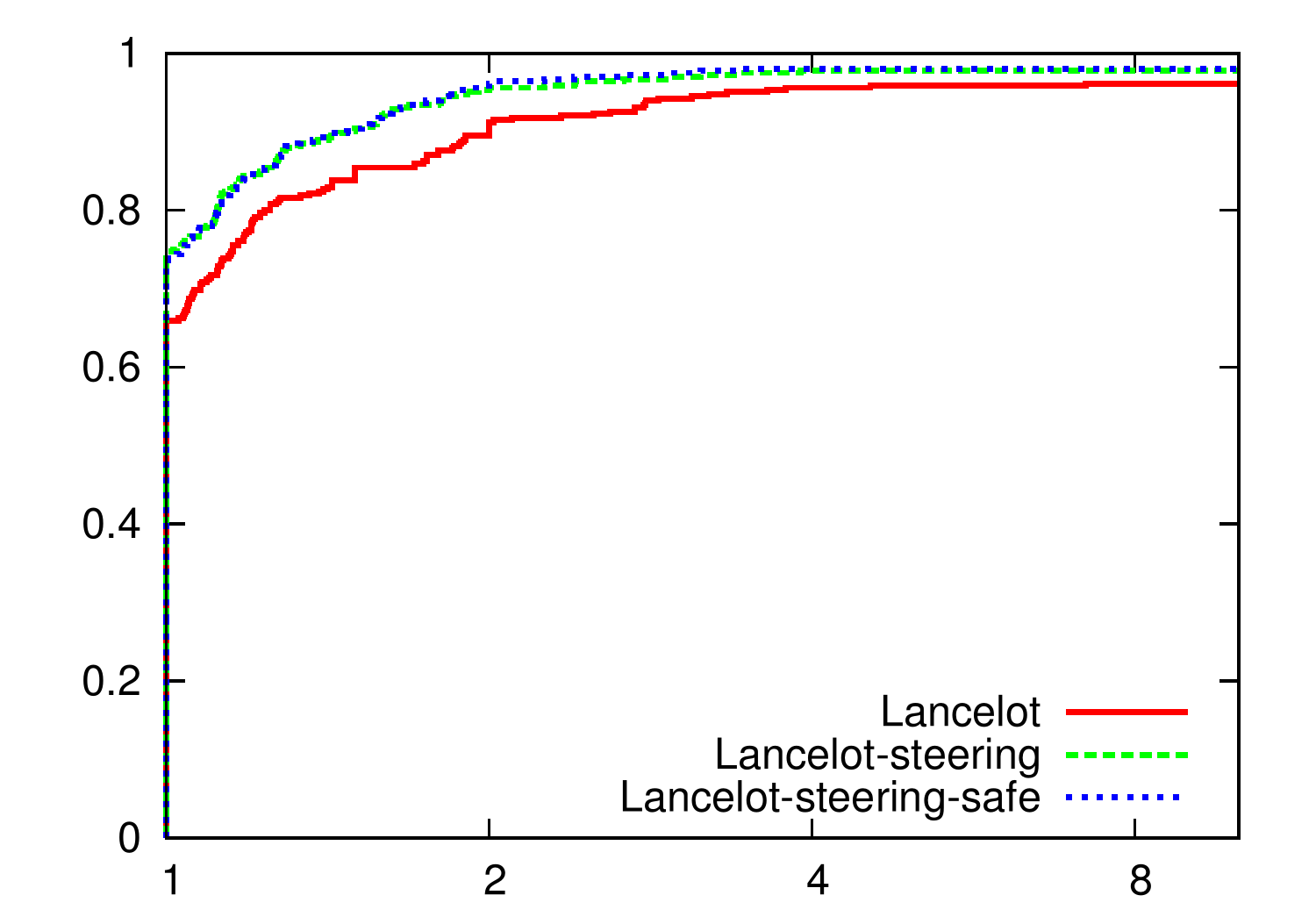}
\caption{Performance profile for iterations: \lancelot{} algorithms on the \CUTEst{} set.}
  \label{fig:ppLanIter}
\end{minipage}
\hspace{1.0cm}
\begin{minipage}[b]{0.45\linewidth}
\centering
  \includegraphics[scale=0.40,angle=-90]{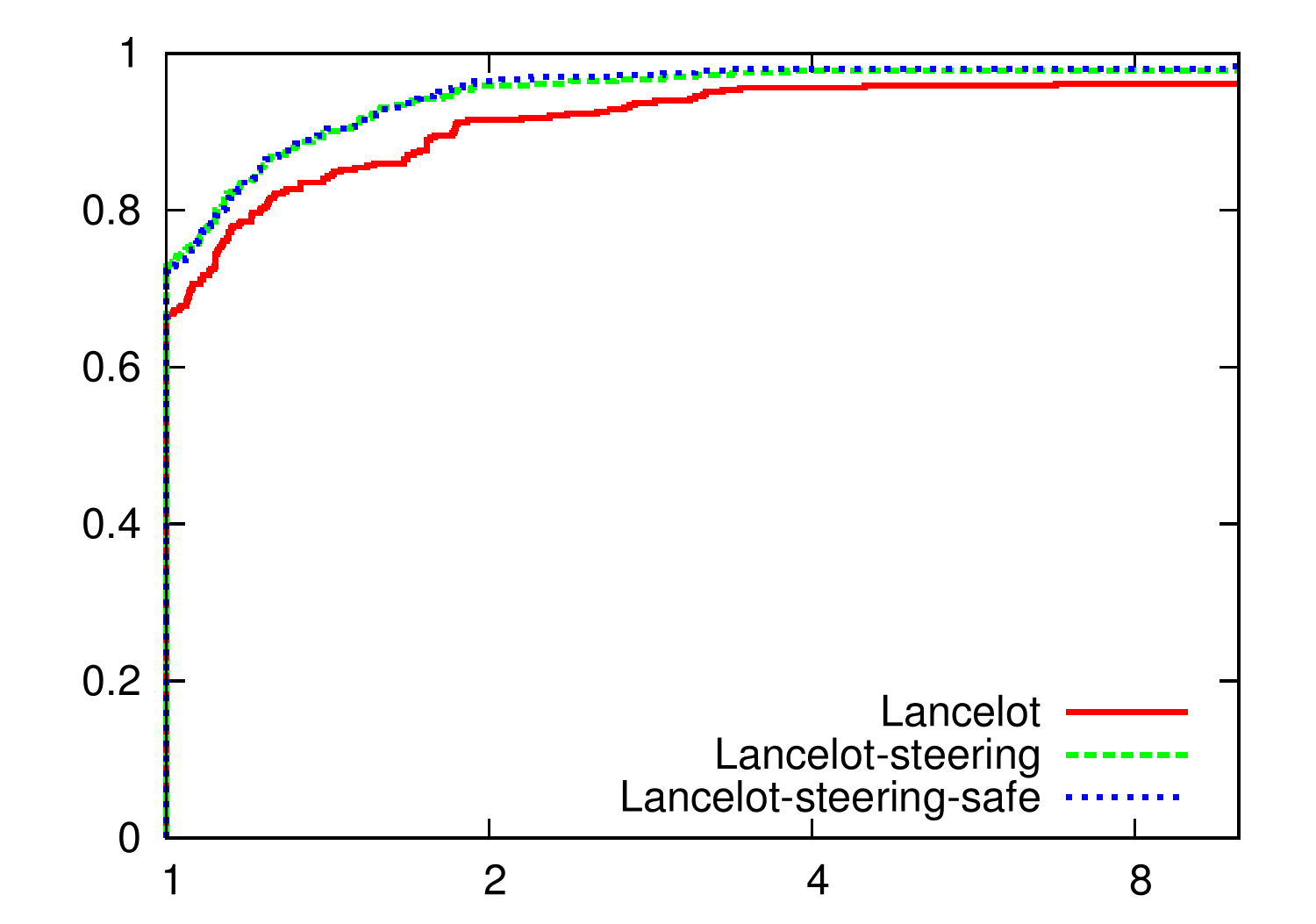}
  \caption{Performance profile for gradient evaluations: \lancelot{} algorithms on the \CUTEst{} set.}
  \label{fig:ppLanGrad}
\end{minipage}
\end{figure}

%

As in \S\ref{sec:matlab-cutest}, it is important to observe the final penalty parameter values yielded by \LancelotS{} and \LancelotSS{} as opposed to those yielded by \Lancelot{}.  For these experiments, we collected this information; see Table~\ref{tab-muvals}.

\begin{table}[ht]
\centering
\caption{Numbers of \CUTEst{} problems for which the final penalty parameter values were in the given ranges.}
\label{tab-muvals}
\begin{tabular}{|lccc|}
  \hline
  \strt $\mu\sub{final}$     & \Lancelot & \LancelotS & \LancelotSS \\ \hline 
  \strt $1$                  & 14        & 1          & 1           \\ 
  \strt $[10^{-1},1)$        & 77        & 1          & 1           \\ 
  \strt $[10^{-2},10^{-1})$  & 47        & 93         & 93          \\ 
  \strt $[10^{-3}, 10^{-2})$ & 27        & 45         & 45          \\ 
  \strt $[10^{-4}, 10^{-3})$ & 18        & 28         & 28          \\ 
  \strt $[10^{-5}, 10^{-4})$ & 15        & 22         & 22          \\ 
  \strt $[10^{-6}, 10^{-5})$ & 12        & 21         & 14          \\ 
  \strt $(0, 10^{-6})$       & 19        & 18         & 25          \\ 
  \hline
\end{tabular}
\end{table}

We make a few remarks about the results in Table~\ref{tab-muvals}.  First, as may have been expected, the \LancelotS{} and \LancelotSS{} algorithms typically reduced the penalty parameter below its initial value, even when \Lancelot{} did not reduce it at all throughout an entire run.  Second, the number of problems for which the final penalty parameter was less than $10^{-4}$ was $171$ for \Lancelot{} and $168$ for \LancelotS{}.  Combining this fact with the previous observation leads us to conclude that steering tended to reduce the penalty parameter from its initial value of $1$, but, overall, it did not decrease it much more aggressively than \Lancelot{}.  Third, it is interesting to compare the final penalty parameter values for \LancelotS{} and \LancelotSS{}.  Of course, these values were equal in any run in which the final penalty parameter was greater than or equal to $10^{-4}$, since this was the threshold value below which safeguarding was activated.  Interestingly, however, \LancelotSS{} actually produced \emph{smaller} values of the penalty parameter compared to \LancelotS{} when the final penalty parameter was smaller than $10^{-4}$.  We initially found this observation to be somewhat counterintuitive, but we believe that it can be explained by observing the penalty parameter updating strategy used by \Lancelot{}.  (Recall that once safeguarding was activated in \LancelotSS{}, the updating strategy became the same used in \Lancelot{}.)  In particular, the decrease factor for the penalty parameter used in \Lancelot{} is $0.1$, whereas the decrease factor used in steering the penalty parameter was $0.7$.  Thus, we believe that \LancelotS{} reduced the penalty parameter more gradually once it was reduced below $10^{-4}$ while \LancelotSS{} could only reduce it in the typical aggressive manner.  (We remark that to (potentially) circumvent this inefficiency in \Lancelot{}, one could implement a different strategy in which the penalty parameter decrease factor is increased as the penalty parameter decreases, but in a manner that still ensures that the penalty parameter converges to zero when infinitely many decreases occur.)  Overall, our conclusion from Table~\ref{tab-muvals} is that steering typically decreases the penalty parameter more than a traditional updating scheme, but the difference is relatively small and we have implemented steering in a way that improves the overall efficiency and reliability of the method.

\section{Conclusion} \label{sec-conclusions}

In this paper, we explored the numerical performance of adaptive updates
to the Lagrange multiplier vector and penalty parameter in AL methods.
Specific to the penalty parameter updating scheme is the use of steering
conditions that guide the iterates toward the feasible region and toward
dual feasibility in a balanced manner.  Similar conditions were first
introduced in~\cite{ByrNW08} for exact penalty functions, but have been
adapted in \cite{curtis12} and this paper to be appropriate for AL-based
methods.  Specifically, since AL methods are not exact (in that, in
general, the trial steps do not satisfy linearized feasibility for any
positive value of the penalty parameter), we allowed for a relaxation of
the linearized constraints.  This relaxation was based on obtaining a
target level of infeasibility that is driven to zero at a modest, but
acceptable, rate.  This approach is in the spirit of AL algorithms since
feasibility and linearized feasibility are only obtained in the limit.
It should be noted that, like other AL algorithms, our adaptive methods
can be implemented matrix-free, i.e., they only require matrix-vector
products.  This is of particular importance when solving large problems
that have sparse derivative matrices.

As with steering strategies designed for exact penalty functions, our
steering conditions proved to yield more efficient and
reliable algorithms than a traditional updating strategy.  This conclusion was
made by performing a variety of numerical tests that involved our own
\Matlab{} implementations and a simple modification of the well-known AL
software \lancelot{}.  To test the potential for the penalty parameter
to be reduced too quickly, we also implemented safeguarded variants of
our steering algorithms.  Across the board, our results indicate that
safeguarding was not necessary and would typically degrade performance
when compared to the unrestricted steering approach.  We feel confident
that these tests clearly show that although our theoretical global
convergence guarantee is weaker than some algorithms (i.e., we cannot
prove that the penalty parameter will remain bounded under a suitable
constraint qualification), this should not be a concern in practice.
Finally, we suspect that the steering strategies described in this paper
would also likely improve the performance of other AL-based methods such
as~\cite{birgin2014practical,kovcvara2003pennon}.



\medskip
\noindent
\emph{Acknowledgments.} 
We would like to thank Sven Leyffer and Victor Zavala from
the Argonne National Laboratory for providing us with the
\AMPL{}~\cite{AMPL,AMPL2} files 
required to test the optimal power flow problems described in \S\ref{sec:opf}. 

\bibliographystyle{gOMS}
\bibliography{al_references}


\appendices

\section{Well-posedness}\label{subsec-wellposed}

Our goal in this appendix is to prove that Algorithm~\ref{alg-aal} is well-posed under Assumption~\ref{ass-posed0}.  Since this assumption is assumed to hold throughout the remainder of this appendix, we do not refer to it explicitly in the statement of each lemma and proof.

\subsection{Preliminary results}

Our proof of the well-posedness of Algorithm~\ref{alg-aal} relies on showing that it will either terminate finitely or will produce an infinite sequence of iterates $\{(x_k,y_k,\mu_k)\}$.  In order to show this, we first require that the \textbf{while} loop that begins at line~\ref{while-1} of Algorithm~\ref{alg-aal} terminates finitely.  Since the same loop appears in the AL trust region method in \cite{curtis12} and the proof of the result in the case of that algorithm is the same as that for Algorithm~\ref{alg-aal}, we need only refer to the result in \cite{curtis12} in order to state the following lemma for Algorithm~\ref{alg-aal}.

\iftoggle{fullpaper}
{
\begin{lemma} \label{lem-is-crit}
  Suppose $l \leq x \leq u$ and let $v$ be any vector in $\Re^n$.  If there exists a scalar $\xi_s > 0$ such that $P[x - \xi_sv] = x$, then $P[x - \xi v] = x$ for all $\xi > 0$.
\end{lemma}
}
\iftoggle{fullpaper}
{
\begin{proof}
  Since $l \leq x \leq u$, it follows the definition of the projection operator $P$ and the facts that $\xi_s > 0$ and $P[x - \xi_sv] = x$ that
  \begin{align*}
    v_i &\geq 0 \words{if $x_i = l_i$;} \mgap
    v_i \leq 0 \words{if $x_i = u_i$;}\words{and}
    v_i = 0 \words{otherwise.}
  \end{align*}
  It follows that $P[x-\xi v] = x$ for all $\xi > 0$, as desired.\qed
\end{proof}
}

\begin{lemma}[\protect{\cite[Lemma~3.2]{curtis12}}] \label{lem-wd-while-1}
  If line~\ref{while-1} is reached, then $F\sub{AL}(x_k,y_k,\mu) \neq 0$ for all sufficiently small $\mu > 0$.
\end{lemma}
\iftoggle{fullpaper}
{
\begin{proof}
  Suppose that line~\ref{while-1} is reached and, to reach a contradiction, suppose also that there exists an infinite positive sequence $\{\xi_h\}_{h\geq0}$ such that $\xi_h \to 0$ and
  \begin{equation} \label{muj}
    F\sub{AL}(x_k,y_k,\xi_h) = P\big[x_k-\xi_h(g_k-J_k\T y_k) - J_k\T c_k\big]-x_k = 0 \words{for all $h\geq0$.}
  \end{equation}
  It follows from~\eqref{muj} and the fact that $\xi_h \to 0$ that
  \begin{equation*} \label{lim-is-zero}
    F\sub{FEAS}(x_k) = P[x_k- J_k\T c_k]-x_k = 0.
  \end{equation*}
  If $c_k \neq 0$, then Algorithm~\ref{alg-aal} would have terminated in line~\ref{term-isp}; hence, since line~\ref{while-1} is reached, we must have $c_k = 0$.  We then may conclude from \eqref{muj}, the fact that $\{\xi_h\}$ is positive for all $h$, and Lemma~\ref{lem-is-crit} that $F\sub{L}(x_k,y_k) = 0$.  Combining this with~\eqref{F-zero} and the fact that $c_k=0$, it follows that $F\sub{OPT}(x_k,y_k) = 0$ so that $(x_k,y_k)$ is a first-order optimal point for \eqref{nep}.  However, under these conditions, Algorithm~\ref{alg-aal} would have terminated in line~\ref{term-kkt}.  Hence, we have a contradiction to the existence of the sequence $\{\xi_h\}$.  \qed 
\end{proof}
}

Next, since the Cauchy steps employed in Algorithm~\ref{alg-aal} are similar to those employed in the method in \cite{curtis12}, we may state the following lemma showing that Algorithms~\ref{alg-rCk} and~\ref{alg-sCk} are well defined when called in lines~\ref{line-rCk}, \ref{line-sCk-1}, and~\ref{line-sCk-2} of Algorithm~\ref{alg-aal}.  It should be noted that a slight difference between Algorithm~\ref{alg-sCk} and the similar procedure in \cite{curtis12} is the use of the convexified model $\tilde{q}$ in \eqref{c2}.  However, we claim that this difference does not affect the veracity of the result.

\begin{lemma}[\protect{\cite[Lemma~3.3]{curtis12}}]\label{lem-algo}
  The following hold true:
  \begin{itemize}
    \item[(i)] The computation of $(\betaC_k,\rCk,\eps_k,\Gamma_k)$ in line~\ref{line-rCk} is well defined and yields $\Gamma_k \in(1,2]$ and $\eps_k\in[0,\eps_r)$.
    \item[(ii)] The computation of $(\alphaC_k,\sCk)$ in lines~\ref{line-sCk-1} and~\ref{line-sCk-2} is well defined. 
  \end{itemize}
\end{lemma}
\iftoggle{fullpaper}
{
\begin{proof}
  We start by proving part (i).  Consider the call to Algorithm~\ref{alg-rCk} made during line~\ref{line-rCk} of Algorithm~\ref{alg-aal}.  If $\theta_k = 0$, then, since $\delta_k > 0$ by construction, we must have $F\sub{FEAS}(x_k) = 0$, which in turn implies that Algorithm~\ref{alg-rCk} trivially computes $l_k = 0$, $\Gamma_k = 2$, $\betaC_k = 1$, and~$\eps_k = 0$.  In this case, (i) clearly holds, so now let us suppose that $\theta_k > 0$.  If $l_k = 0$, then $\Gamma_k = 2$; otherwise, if $l_k > 0$, then it follows that either $\Gamma_k = 2$ or   
  $$
    2 > \Gamma_k = \half\left(1 + \frac{\twonorm{P\big[x_k-\gamma^{l_k-1}J_k\T c_k\big]-x_k}}{\theta_k}\right) > \half\left(1 + \frac{\theta_k}{\theta_k}\right) = 1.
  $$
  Next, the \textbf{while} loop at line~\ref{line-rCk-while} terminates finitely as shown by~\cite[Theorem~4.2]{Mor88} since $\eps_r\in(0,1)$. The fact that $\eps_k\in[0,\eps_r)$ holds since $\eps_r\in(0,1)$ by choice, $\eps_k$ is initialized to zero in Algorithm~\ref{alg-rCk}, and every time $\eps_k$ is updated we have $-\delmod_v(\rCk;x_k) / \rCk\T J_k\T c_k < \eps_r$ (by the condition of the \textbf{while} loop).

  Now consider part (ii).  Regardless of how line~\ref{line-sCk-1} or~\ref{line-sCk-2} is reached in Algorithm~\ref{alg-aal}, we have that $\mu_k > 0$ and $\Theta_k \geq 0$ by construction, and from part (i) we have that $\Gamma_k \in (1,2]$ and $\eps_k\in[0,\eps_r)$.  If $\Theta_k = 0$, then, since $\delta_k > 0$ by construction, it follows that $F\sub{AL}(x_k,y_k,\mu_k) = 0$, and therefore Algorithm~\ref{alg-sCk} terminates with $\alphaC_k = 1$ and $\sCk=0$.  Thus, we may continue supposing that $\Theta_k > 0$.  We may now use the fact that
  $$
  0 < \eps_r/2 \leq (\eps_k+\eps_r)/2 < \eps_r < 1
  $$
  and~\cite[Theorem~4.2]{Mor88} to say that Algorithm~\ref{alg-sCk} will terminate finitely with $(\alphaC_k,\sCk)$ satisfying~\eqref{c2}. \qed
\end{proof}
}

\newcommand{\lemmodels}{Let $(\betaC_k,\rCk,\eps_k,\Gamma_k) \gets
 \text{\sc Cauchy\_feasibility}(x_k,\TRk)$ with $\TRk$ defined
 by~\eqref{sandy} and, as quantities dependent on the penalty parameter
 $\mu > 0$, let $(\alphaC_k(\mu),\sCk(\mu)) \gets \text{\sc
   Cauchy\_AL}(x_k,y_k,\mu,\Theta_k(\mu),\eps_k)$ with
 $\TRALk(\mu):=\Gamma_k\delta\twonorm{F\sub{AL}(x_k,y_k,\mu)}$ (see
 \eqref{sandy-AL}). Then, the following hold true:

  \begin{subequations}\label{lemma1}
    \begin{align}
      \lim_{\mu\to 0} \left( \max_{\twonorm{s}\leq 2\theta_k} 
        |\tilde{q}(s;x_k,y_k,\mu)-\qv(s;x_k) |\right) &= 0, 
         \label{models-converge} \\
      \lim_{\mu\to 0} \Grad_x \Lscr(x_k,y_k,\mu) &= J_k^Tc_k, 
         \label{grads-converge} \\
      \lim_{\mu\to 0} \sCk(\mu) &= \rCk, \label{cauchys-converge} \\
      \words{and} \lim_{\mu\to 0} \delmodv(\sCk(\mu);x_k) &= \delmodv(\rCk;x_k).
          \label{modvals-converge}
    \end{align}
  \end{subequations}
}

\iftoggle{arxivpaper}{
  The next result, similar to \cite[Lemma~3.4]{curtis12}, highlights
  critical relationships between $\qv$ and $\tilde{q}$ as $\mu \to 0$.
  Indeed, much of the proof follows exactly the same logic as
  \cite[Lemma~3.4]{curtis12}, but we provide a complete proof to account
  for our present use of the convexified model $\tilde{q}$ and the
  differences in the trust region radii for the subproblems employed in
  the algorithms.  This result is crucial for showing that the steering
  condition~\eqref{cond2b} is satisfied for all sufficient small $\mu$
  (see Lemma~\ref{keylemma2}).

  \begin{lemma}\label{lem-models}
   \lemmodels
   \end{lemma}
\begin{proof}
  Since $x_k$ and $y_k$ are fixed during iteration $k$, for ease of
  exposition we often drop these quantities from function dependencies
  for the purposes of this proof.  From the definitions of $\qv$ and
  $\tilde{q}$, it follows that for some constants $M_1 > 0$ and $M_2 >
  0$ independent of $\mu$ we have
  \begin{align*}
    \max_{\twonorm{s}\leq 2\theta_k} |\tilde{q}(s;\mu)-\qv(s)| &= \max_{\twonorm{s}\leq 2\theta_k} |\mu \ell_k + \mu \Grad_x \ell_k\T s + \max\{\tfrac{\mu}{2}s\T \Grad_{xx} \ell_k s, -\half\twonorm{J_k s}^2\}| \\
  &\leq \mu M_1 + \max\{\mu M_2, -\half\twonorm{J_k s}^2\}.
  \end{align*}
  Since the right-hand side of this expression vanishes as $\mu \to 0$, we have \eqref{models-converge}.  Further,
  \begin{equation*}
    \Grad_x\Lscr(x_k,y_k,\mu) - J_k^Tc_k = \mu(g_k - J_k^Ty_k),
  \end{equation*}
  which implies that \eqref{grads-converge} holds.

  We now show that~\eqref{cauchys-converge} holds.  Our proof considers two cases depending on the value $F\sub{FEAS}(x_k)$.  Throughout consideration of these cases, it should be observed that all quantities in Algorithm~\ref{alg-rCk} are unaffected by $\mu$, so they can be considered as fixed quantities.

  \begin{enumerate}
    \item[\textbf{Case 1:}] If $F\sub{FEAS}(x_k) = 0$, then $\TRk = \delta\twonorm{F\sub{FEAS}(x_k)} = 0$, from which it follows that $\rCk = 0$ and $\delmodv(\rCk) = 0$.  Furthermore, from \eqref{grads-converge}, we have $\TRALk(\mu) \to 0$ as $\mu\to0$, which means $\sCk(\mu) \to 0 = \rCk$, as desired.
    \item[\textbf{Case 2:}] Now suppose that $F\sub{FEAS}(x_k) \neq 0$.  In the following arguments, we define the following functions of a nonnegative integer $l$ and positive scalar $\mu$:
    \begin{equation*}
      r_k(l) = P(x_k-\gamma^{l} J_k \T c_k) - x_k \wordss{and} s_k(l,\mu) = P(x_k-\gamma^{l} \Grad_x \Lscr(\mu)) - x_k.
    \end{equation*}
    We also define $l_\beta \geq 0$ to be the integer such that $\betaC_k = \gamma^{l_\beta}$ (see Algorithm~\ref{alg-rCk}), which implies that
    \begin{equation} \label{r-need}
      \rCk=r_k(l_\beta).
    \end{equation}
    
    We have as a consequence of~\eqref{grads-converge} that
    \begin{equation*}
      \lim_{\mu\to 0} s_k(l,\mu) = r_k(l) \words{for any} l \geq 0.
    \end{equation*}
    In particular, this implies with \eqref{r-need} that
    \begin{equation} \label{s-to-r}
      \lim_{\mu\to 0} s_k(l_\beta,\mu) = r_k(l_\beta) = \rCk.
    \end{equation}
    Thus, \eqref{cauchys-converge} follows as long as
    \begin{equation} \label{same-cs}
      \sCk(\mu) = s_k(l_\beta,\mu) \wordss{for all sufficiently small} \mu>0.
    \end{equation}
    Since the computation of $\sCk(\mu)$ (via the $\text{\sc Cauchy\_AL}$ routine stated as Algorithm~\ref{alg-sCk}) computes a nonnegative integer $l_{\alpha,\mu}$ such that
    \begin{equation*}
      \sCk(\mu) = P(x_k - \gamma^{l_{\alpha,\mu}}\Grad_x\Lscr(\mu)) - x_k,
    \end{equation*}
    it follows that \eqref{same-cs} can be proved by showing that $l_{\alpha,\mu} = l_\beta$ for all sufficiently small $\mu > 0$.  As a preliminary result in the proof of this fact, we first show that for $l_k$ computed in Algorithm~\ref{alg-rCk} we have
    \begin{equation} \label{l-bound}
       \min\{l_\beta,l_{\alpha,\mu}\} \geq l_k \wordss{for all sufficiently small} \mu > 0.
    \end{equation}
    Indeed, if $l_k = 0$, then~\eqref{l-bound} holds trivially.  Thus,
    let us suppose that $l_k > 0$.  According to the procedures in
    Algorithm~\ref{alg-rCk}, it is clear that $l_\beta \geq l_k$.
    Hence, we may turn our attention to $l_{\alpha,\mu}$.  From the
    definition of $\Theta_k(\mu)$ (in the statement of this lemma),
    \eqref{grads-converge}, the manner in which $\Gamma_k$ is set in
    Algorithm~\ref{alg-rCk}, the fact that $\theta_k > 0$, and since
    $\twonorm{P(x_k - \gamma^{l_k-1}J_k\T c_k) - x_k} > \theta_k$ due to
    the manner in which $l_k$ is set in Algorithm~\ref{alg-rCk}, we have
    that
    \begin{align*}
      \lim_{\mu \to 0} \Theta_k(\mu)
        &=\lim_{\mu \to 0} \Gamma_k \delta\twonorm{F\sub{AL}(x_k,y_k,\mu)} = \Gamma_k \delta\twonorm{F\sub{FEAS}(x_k)} = \Gamma_k \theta_k \\
        &= \min\left\{2, \half\left(1 + \frac{\twonorm{P(x_k-\gamma^{l_k-1}J_k\T c_k)-x_k}}{\theta_k}\right)\right\}\theta_k\\ 
        &= \min\left\{2\theta_k, \half\left(\theta_k + \twonorm{P(x_k-\gamma^{l_k-1}J_k\T c_k)-x_k}\right)\right\} \\
        &\in \big(\theta_k,\twonorm{P(x_k-\gamma^{l_k-1}J_k\T c_k)-x_k}\big).
    \end{align*}
    Along with~\eqref{grads-converge}, this implies that for all sufficiently small $\mu > 0$ we have
    \begin{equation*}
      \lim_{\mu \to 0} \twonorm{P(x_k-\gamma^{l_k-1}\Grad_x\Lscr(\mu))-x_k} = \twonorm{P(x_k-\gamma^{l_k-1}J_k\T c_k)-x_k} > \Theta_k(\mu).
    \end{equation*}
    This shows that $l_{\alpha,\mu} \geq l_k$ holds for all sufficiently small $\mu > 0$.  Consequently, we have~\eqref{l-bound}.

  \smallskip
    Having ensured that \eqref{l-bound} holds, we proceed to prove that
    $l_{\alpha,\mu} = l_\beta$ for all sufficiently small $\mu > 0$.  It
    follows from the definition of $l_\beta$, \eqref{r-need}, the
    procedures of Algorithm~\ref{alg-rCk} (e.g., the manner in which
    $\eps_k$ is set), and part (i) of Lemma~\ref{lem-algo} that
    \begin{equation}\label{ratio-ineq-0}
      \begin{aligned}
        -\frac{\delmodv(\rCk)}{\rCk\T J_k\T c_k} = -\frac{\delmodv\big(r_k(l_\beta)\big)}{r_k(l_\beta)\T J_k\T c_k} &\geq \epsilon_r \\
        \wordss{and} -\frac{\delmodv\big(r_k(l)\big)}{r_k(l)\T J_k\T c_k} \leq \epsilon_k &< \epsilon_r \words{for all integers $l_k \leq l < l_\beta$.}
      \end{aligned}
    \end{equation}
    (Here, it is important to note that~\cite[Theorem~12.1.4]{ConGT00a}
    can be invoked to ensure that all denominators
    in~\eqref{ratio-ineq-0} are negative.)  It follows
    from~\eqref{grads-converge}, \eqref{s-to-r},
    \eqref{models-converge}, \eqref{ratio-ineq-0}, and part (i) of
    Lemma~\ref{lem-algo} that
    \begin{equation}\label{eqn-yes}
      \lim_{\mu\to 0} -\frac{\Deltait\tilde{q}\big(s_k(l_\beta,\mu)\big)}{s_k(l_\beta,\mu)^T \Grad_x\Lscr(\mu)} = -\frac{\delmodv(\rCk)}{\rCk\T J_k\T c_k} \geq \eps_r > \frac{\eps_k+\eps_r}{2}
    \end{equation}
    and
    \begin{equation} \label{eqn-no}
      \lim_{\mu\to 0} -\frac{\Deltait\tilde{q}\big(s_k(l,\mu)\big)}{s_k(l,\mu)^T \Grad_x\Lscr(\mu)} = -\frac{\delmodv\big(r_k(l)\big)}{r_k(l)\T J_k\T c_k} \leq \eps_k < \frac{\eps_k+\eps_r}{2} \wordss{for all integers $l_k \leq l < l_\beta$.}
    \end{equation}
    It now follows from~\eqref{l-bound}, \eqref{eqn-yes},
    \eqref{eqn-no}, and~\eqref{c2} that $l_{\alpha,\mu} = l_\beta$ for
    all sufficiently small $\mu > 0$.  As previously mentioned, this
    proves~\eqref{cauchys-converge}.
  \end{enumerate}

  Finally, we notice that \eqref{modvals-converge} follows from 
  \eqref{cauchys-converge} and continuity of the model $\qv$.
 \end{proof}
}{
  The next result highlights critical relationships between 
  $\qv$ and $\tilde{q}$ as $\mu \to 0$.

\begin{lemma}[\protect{\cite[Lemma~\ref{lem-models}]{CJJR-arxiv}}]\label{lem-models}
 \lemmodels
 \end{lemma}
}

\newcommand{\lemkeylemmaone}{
  Let $\Omega$ be any scalar value such that
  \begin{equation}\label{the-alpha-and-the-omega}
    \Omega \geq \max\{\twonorm{\mu_k\Hess_{xx}\ell(x_k,y_k) + J_k\T J_k},\twonorm{J_k\T J_k}\}.
  \end{equation}
  Then, the following hold true:
  \begin{itemize}
    \item[(i)] For some $\kappa_4 \in (0,1)$, the Cauchy step for subproblem~\eqref{prob-normal} yields
    \begin{equation}\label{decrease-feas}
      \delmodv(\rCk;x_k) \geq \kappa_4 \twonorm{F\sub{FEAS}(x_k)}^2 \min\left\{\delta,\frac{1}{1+\Omega}\right\}.
    \end{equation}
    \item[(ii)] For some $\kappa_5 \in (0,1)$, the Cauchy step for subproblem~\eqref{subprob-al} yields
    \begin{equation}\label{decrease-AL}
      \Deltait\tilde{q}(\sCk;x_k,y_k,\mu_k) \geq \kappa_5 \twonorm{F\sub{AL}(x_k,y_k,\mu_k)}^2 \min\left\{\delta,\frac{1}{1+\Omega}\right\}.
    \end{equation}
  \end{itemize}
}

We also need the following lemma related to Cauchy decreases in the models $\qv$ and $\tilde{q}$.  
\iftoggle{arxivpaper}{
The conclusions of the lemma are similar to~\cite[Lemma~3.5]{curtis12}, but here we account for the convexified model $\tilde{q}$ and other differences in the subproblems employed here as opposed to those in \cite{curtis12}.
\begin{lemma}\label{keylemma1}
 \lemkeylemmaone
 \end{lemma}
\begin{proof}
  Let $\Sigma_k := 1 +\sup\{|\omega_k(r)|:0<\twonorm{r}\leq\TRk\}$, where
  \begin{equation*}
    \omega_k(r) = \frac{-\delmodv(r;x_k)-r\T J_k\T c_k}{\twonorm{r}^2} \words{for all} r \in \Re^n.
  \end{equation*}
  In fact, using \eqref{the-alpha-and-the-omega}, the Cauchy-Schwartz inequality, and standard norm inequalities, we have that
  \begin{equation*}
    \omega_k(r) = \frac{r\T J_k\T J_k r}{2\twonorm{r}^2} \leq \Omega \words{for all} r \in \Re^n.
  \end{equation*}
  Hence, $\Sigma_k \leq 1+\Omega$.  The requirement \eqref{c1} and \cite[Theorem 4.4]{Mor88} then yield, for some $\bar\kappa_4 \in (0,1)$, that
  \begin{equation*}
    \delmodv(\rCk;x_k) \geq \epsilon_r \bar\kappa_4 \twonorm{F\sub{FEAS}(x_k)} \min\left\{\TRk,\frac{1}{\Sigma_k}\twonorm{F\sub{FEAS}(x_k)}\right\}.
  \end{equation*}
  which, with \eqref{sandy}, implies that \eqref{decrease-feas} follows with $\kappa_4 := \epsilon_r \bar\kappa_4$.

  We now show~\eqref{decrease-AL} in a similar manner.  Let $\bar{\Sigma}_k := 1 + \sup\{|\bar{\omega}_k(s)|:0<\twonorm{s}\leq\TRALk\}$ where
  \begin{equation*}
    \bar{\omega}_k(s) := \frac{-\Deltait\tilde{q}(s;x_k,y_k,\mu_k)-s\T \Grad_x \Lscr(x_k,y_k,\mu_k)}{\twonorm{s}^2} \words{for all} s \in \Re^n.
  \end{equation*}
  Using \eqref{the-alpha-and-the-omega}, we have in a similar manner as above that
  \begin{align*}
    \bar{\omega}_k(s)
      &=    \frac{\max\{\mu_k s\T\Hess_{xx} \ell(x_k,y_k,\mu_k) s +s\T J_k\T J_k s,0\}}{2\twonorm{s}^2} \\
      &\leq \frac{\twonorm{\mu_k \Hess_{xx} \ell(x_k,y_k,\mu_k)+J_k\T J_k}\twonorm{s}^2}{2\twonorm{s}^2} \leq \Omega.
  \end{align*}
  Thus, $\bar{\Sigma}_k \leq 1+\Omega$.  The requirement \eqref{c2} and \cite[Theorem 4.4]{Mor88} then yield, for some $\bar\kappa_5 \in (0,1)$, that
  \begin{equation*}
    \Deltait\tilde{q}(\sCk;x_k,y_k,\mu_k) \geq \frac{\epsilon_k+\epsilon_r}{2}\bar\kappa_5 \twonorm{F\sub{AL}(x_k,y_k,\mu_k)} \min\left\{\TRALk,\frac{1}{\bar{\Sigma}_k}\twonorm{F\sub{AL}(x_k,y_k,\mu_k)} \right\},
  \end{equation*}
  which, with \eqref{sandy-AL} and Lemma~\ref{lem-algo}(i), implies that \eqref{decrease-AL} follows with $\kappa_5 := \tfrac12 \epsilon_r \bar\kappa_5$.
\end{proof}
}{
\begin{lemma}[\protect{\cite[Lemma~\ref{keylemma1}]{CJJR-arxiv}}]\label{keylemma1}
 \lemkeylemmaone
 \end{lemma}
}

\newcommand{\lemkeylemmatwo}{
   The \textbf{while} loop that begins at line~\ref{while-2} of 
   Algorithm~\ref{alg-aal} terminates finitely.
}

The next lemma shows that the \textbf{while} loop at
line~\ref{while-2}, which is responsible for
ensuring that our adaptive steering conditions in \eqref{conds} are
satisfied, terminates finitely.
\iftoggle{arxivpaper}{
The proof of this is similar to the second part of~\cite[Theorem~3.6]{curtis12}
\begin{lemma}\label{keylemma2}
 \lemkeylemmatwo
 \end{lemma}
\begin{proof}
  Since Lemma~\ref{lem-wd-while-1} ensures that the latter condition in
  the \textbf{while} loop is satisfied for all sufficiently small $\mu_k
  > 0$, it suffices to show that $s_k = \sCk$ and $r_k = \rCk$ satisfy
  \eqref{cond2b} for all sufficiently small $\mu_k > 0$.  To see this,
  we may borrow notation from Lemma~\ref{lem-models}---i.e., to consider
  $s_k = \sCk$ as a quantity dependent on a parameter $\mu>0$---and
  observe that \eqref{modvals-converge} implies
  \begin{equation}\label{tired}
    \lim_{\mu\to 0} \delmodv(s_k(\mu);x_k) = \lim_{\mu\to 0} 
       \delmodv(\sCk(\mu);x_k) = \delmodv(\rCk;x_k) = \delmodv(r_k;x_k).
  \end{equation}
  If $\delmodv(r_k;x_k) > 0$, then~\eqref{tired} implies that \eqref{cond2b} 
  is satisfied for sufficiently small $\mu_k>0$.  Otherwise,
  \begin{equation}\label{snore}
    \delmodv(r_k;x_k) = \delmodv(\rCk;x_k)= 0,
  \end{equation}
  which along with~\eqref{decrease-feas} implies that
  $F\sub{FEAS}(x_k)=0$.  We may now consider two cases depending on
  whether $x_k$ is feasible for \eqref{nep}.  If $c_k \neq 0$, then
  Algorithm~\ref{alg-aal} would have terminated in line~\ref{term-isp},
  meaning that the \textbf{while} loop at line~\ref{while-2} would not have
  been reached.  On the other hand, if $c_k = 0$, then  \eqref{snore} implies
  \begin{equation}\label{snore2}
    \min\{\kappa_3\delmodv(r_k;x_k), v_k - \half(\kappa_tt_j)^2\} 
       = -\half(\kappa_tt_j)^2 < 0.
  \end{equation}
  This last strict inequality follows since $t_j > 0$ by construction
  and $\kappa_t \in (0,1)$ by choice.  Therefore, we can deduce
  that~\eqref{cond2b} will be satisfied for sufficiently small $\mu_k >
  0$ by observing~\eqref{tired}, \eqref{snore} and \eqref{snore2}.
\end{proof} 
}{
\begin{lemma}[\protect{\cite[Lemma~\ref{keylemma2}]{CJJR-arxiv}}]\label{keylemma2}
 \lemkeylemmatwo
 \end{lemma}
}

\newcommand{\lemisdescent}{
  At line~\ref{line-alpha} of Algorithm~\ref{alg-aal}, the search direction 
  $s_k$ is a strict descent direction for $\Lscr(\cdot,y_k,\mu_k)$ from $x_k$.  
  In particular,
  \begin{equation} \label{eq-descent}
    \Grad_x\Lscr(x_k,y_k,\mu_k)^T s_k \leq 
      -\Deltait\tilde{q}(s_k;x_k,y_k,\mu_{k}) \leq 
        -\kappa_1 \Deltait\tilde{q}(\sCk;x_k,y_k,\mu_{k})<0.
  \end{equation}
}

The final lemma of this section shows that $s_k$ is a strict descent direction
for the AL function.  The conclusion of this lemma is the primary
motivation for our use of the convexified model $\tilde{q}$.
\iftoggle{arxivpaper}{
\begin{lemma}\label{lem-is-descent}
 \lemisdescent
 \end{lemma}
\begin{proof}
  From the definition of $\tilde{q}$, we find
  \begin{align*}
        \Deltait\tilde{q}(s_k;x_k,y_k,\mu_k)
       &= \tilde{q}(0;x_k,y_k,\mu_k)-\tilde{q}(s_k;x_k,y_k,\mu_k)\\
       &= -\Grad_x\Lscr(x_k,y_k,\mu_k)\T s_k- \max\{\half s_k\T(\mu_k\Hess_{xx}\ell(x_k,y_k) + J_k^TJ_k)s_k,0\}\\
    &\leq -\Grad_x\Lscr(x_k,y_k,\mu_k)\T s_k.
  \end{align*}
  It follows from this inequality and~\eqref{cond1} that
  \begin{equation*}
    \Grad_x\Lscr(x_k,y_k,\mu_k)\T s_k \leq -\Deltait\tilde{q}(s_k;x_k,y_k,\mu_k) \leq -\kappa_1 \Deltait\tilde{q}(\sCk;x_k,y_k,\mu_k) < 0,
  \end{equation*}
  as desired.
\end{proof}
}{
\begin{lemma}[\protect{\cite[Lemma~\ref{lem-is-descent}]{CJJR-arxiv}}]\label{lem-is-descent}
 \lemisdescent
 \end{lemma}
}


\subsection{Proof of well-posedness result}

\begin{proof}[Proof of Theorem~\ref{thm-wd}]
  If, during the $k$th iteration, Algorithm~\ref{alg-aal} terminates in line~\ref{term-kkt} or~\ref{term-isp}, then there is nothing to prove.  Thus, to proceed in the proof, we may assume that line~\ref{while-1} is reached.  Lemma~\ref{lem-wd-while-1} then ensures that 
  \begin{equation} \label{got-1}
    F\sub{AL}(x_k,y_k,\mu) \neq 0 \words{for all sufficiently small} \mu > 0.
  \end{equation}
  Consequently, the \textbf{while} loop in line~\ref{while-1} will terminate for a sufficiently small $\mu_k > 0$.  Next, by construction, conditions~\eqref{cond1} and \eqref{cond2} are satisfied for any $\mu_k > 0$ by $s_k = \sCk$ and $r_k = \rCk$.  Lemma \ref{keylemma2} then shows that for a sufficiently small $\mu_k > 0$, \eqref{cond2b} is also satisfied by $s_k = \sCk$ and $r_k = \rCk$.  Therefore, line \ref{line-alpha} will be reached.  Finally, Lemma \ref{lem-is-descent} ensures that $\alpha_k$ in line \ref{line-alpha} is well-defined.  This completes the proof as all remaining lines in the $k$th iteration are explicit.
\end{proof}  

\section{Global Convergence} \label{sec-convergence}

We shall tacitly presume that Assumption~\ref{ass-0} holds throughout 
this section, and not state it explicitly.
This assumption and the bound on the multipliers enforced in line~\ref{line-y-bd} of Algorithm~\ref{alg-aal} imply that there exists a positive monotonically increasing sequence $\{\Omega_j\}_{j\geq 1}$ such that for all $k_j \leq k < k_{j+1}$ we have 
\begin{subequations}\label{asses}
  \begin{align}
    \twonorm{\Hess_{xx}\Lscr(\sigma,y_k,\mu_k)}
    &\leq \Omega_j \words{for all $\sigma$ on the segment $[x_k,x_k+s_k]$,} \label{ass-2}\\
    \twonorm{\mu_k\Hess_{xx}\ell(x_k,y_k) + J_k^TJ_k}
    &\leq \Omega_j, \label{ass-3}\\
    \words{and} \twonorm{J_k^TJ_k}
    &\leq \Omega_j. \label{ass-4}
  \end{align}
\end{subequations}
In the subsequent analysis, we make use of the subset of iterations for which line~\ref{line-kj} of Algorithm~\ref{alg-aal} is reached.  For this purpose, we define the iteration index set
\begin{equation}\label{def-y-iter}
  \Yscr := \left\{k_j : \twonorm{c_{k_j}} \leq t_j,\ \min\{\twonorm{F\sub{L}(x_{k_j},\yhat_{k_j})}, \twonorm{F\sub{AL}(x_{k_j},y_{k_j-1},\mu_{k_j-1})}\} \leq T_j\right\}.
\end{equation}

\subsection{Preliminary results}

The following result provides critical bounds on differences in (components of) the augmented Lagrangian summed over sequences of iterations.  We remark that the proof in \cite{curtis12} essentially relies on Assumption~\ref{ass-0} and Dirichlet's Test \cite[\S3.4.10]{DavD10}.

\begin{lemma}[\protect{\cite[Lemma~3.7]{curtis12}}.]\label{lem-sums}
  The following hold true.
  \begin{itemize}
  \item[(i)] If $\mu_k = \mu$ for some $\mu > 0$ and all sufficiently large $k$, then there exist positive constants $M_f$, $M_c$, and $M_\Lscr$ such that for all integers $p \geq 1$ we have 
  \begin{align}
    &\sum_{k=0}^{p-1} \mu_k(f_k-f_{k+1}) < M_f,  \label{f-fin1} \\
    &\sum_{k=0}^{p-1} \mu_ky_k\T (c_{k+1}-c_k) < M_c, \label{c-fin1} \\
    \words{and} &\sum_{k=0}^{p-1} (\Lscr(x_k,y_k,\mu_k) - \Lscr(x_{k+1},y_k,\mu_k)) \label{L-fin1} < M_\Lscr.
  \end{align}
  \item[(ii)] If $\mu_k \to 0$, then the sums
  \begin{align}
    &\sum_{k=0}^\infty \mu_k(f_k-f_{k+1}), \label{f-fin2} \\
    &\sum_{k=0}^\infty \mu_ky_k\T (c_{k+1}-c_k), \label{c-fin2} \\
    \words{and} &\sum_{k=0}^\infty (\Lscr(x_k,y_k,\mu_k) - \Lscr(x_{k+1},y_k,\mu_k)) \label{L-fin2}
  \end{align}
  converge and are finite, and
  \begin{equation}
    \lim_{k\to\infty} \twonorm{c_k} = \cbar \words{for some $\cbar\geq0$.}\label{c-yay}
 \end{equation}
  \end{itemize}
\end{lemma}
\iftoggle{fullpaper}
{
\begin{proof}
  Under Assumption~\ref{ass-0} we may conclude that for some constant $\overline{M}_f>0$ and all integers $p\geq 1$ we have
  \begin{equation*}
    \sum_{k= 0}^{p-1}(f_k-f_{k+1}) = f_0-f_p < \overline{M}_f.
  \end{equation*}
  If $\mu_k = \mu$ for all sufficiently large $k$, then this implies
  that \eqref{f-fin1} clearly holds for some sufficiently large $M_f$.
  Otherwise, if $\mu_k \to 0$, then it follows from Dirichlet's
  Test~\cite[\S3.4.10]{DavD10} and the fact that $\{\mu_k\}_{k\geq 0}$
  is a monotonically decreasing sequence that converges to zero
  that~\eqref{f-fin2} converges and is finite. 

  Next, we show that for some constant $\overline{M}_c>0$ and all integers $p \geq 1$ we have
  \begin{equation} \label{c-mbar}
    \sum_{k=0}^{p-1} y_k\T (c_{k+1}-c_k) < \overline{M}_c.
  \end{equation}
  First, suppose that $\Yscr$ defined in~\eqref{def-y-iter} is finite.  It follows that there exists $k' \geq 0$ and $y$ such that $y_k = y$ for all $k \geq k'$.  Moreover, under Assumption~\ref{ass-0} there exists a constant $\Mhat_c>0$ such that for all $p \geq k'+1$ we have
  \begin{equation*}
    \sum_{k=k'}^{p-1} y_k\T (c_{k+1}-c_k) = y^T \sum_{k=k'}^{p-1} (c_{k+1}-c_k) = y^T (c_p-c_{k'}) \leq \twonorm{y}\twonorm{c_p-c_{k'}} < \Mhat_c.
  \end{equation*}
  It is now clear that~\eqref{c-mbar} holds in this case.  Second,
  suppose that $|\Yscr|=\infty$ so that the sequence $\{k_j\}_{j\geq 1}$
  in Algorithm~\ref{alg-aal} is infinite.  By construction $t_j \to 0$, so for some $j' \geq 1$ we have
  \begin{equation} \label{t-event}
    t_j = t_{j-1}^{1+\epsilon} \words{and} Y_{j+1} = t_{j-1}^{-\epsilon} \words{for all $j \geq j'$.}
  \end{equation}
  From the definition of the sequence $\{k_j\}_{j\geq 1}$,
  \eqref{y-combo}, and \eqref{new-y-2}, we know that
  \begin{align*}
    \sum_{k=k_j}^{k_{j+1}-1} y_k\T (c_{k+1}-c_k)
      &= y_{k_j}\T \sum_{k=k_j}^{k_{j+1}-1} (c_{k+1}-c_k) = y_{k_j}\T (c_{k_{j+1}} - c_{k_j}) \\ &\leq \twonorm{y_{k_j}} \twonorm{c_{k_{j+1}} - c_{k_j}} \leq 2Y_{j+1}t_j = 2t_{j-1} \words{for all $j\geq j'$,}
  \end{align*}
  where the last equality follows from~\eqref{t-event}.  Using these
  relationships, summing over all $j' \leq j \leq j'+q$ for an
  arbitrary integer $q \geq 1$, and using the fact that $t_{j+1} \leq
  \gamma_t t_j$ by construction, leads to 
  \begin{align*}
    \sum_{j=j'}^{j'+q}\left(\sum_{k= k_j}^{k_{j+1}-1} y_k\T (c_{k+1}-c_k)\right)
      &\leq 2\sum_{j=j'}^{j'+q} t_{j-1} \nonumber \\
      &\leq 2t_{j'-1}\sum_{l=0}^{q} \gamma_t^l = 2t_{j'-1}\frac{1-\gamma_t^{q+1}}{1-\gamma_t} \leq \frac{2t_{j'-1}}{1-\gamma_t}.
  \end{align*}
  It is now clear that~\eqref{c-mbar} holds in this case as well.

  We have shown that~\eqref{c-mbar} always holds.  Thus, if $\mu_k = \mu$ for all sufficiently large $k$, then \eqref{c-fin1} holds for some sufficiently large $M_c$.  Otherwise, if $\mu_k \to 0$, then it follows from Dirichlet's Test~\cite[\S3.4.10]{DavD10}, \eqref{c-mbar} and the fact that $\{\mu_k\}_{k\geq 0}$ is a monotonically decreasing sequence that converges to zero that~\eqref{c-fin2} converges and is finite.

  Finally, observe that
  \begin{align}
     & \sum_{k=0}^{p-1} \Lscr(x_k,y_k,\mu_k) - \Lscr(x_{k+1},y_k,\mu_k)\nonumber \\
    =& \sum_{k=0}^{p-1} \mu_k(f_k-f_{k+1}) + \sum_{k=0}^{p-1} \mu_k y_k\T(c_{k+1}-c_k) + \half\sum_{k=0}^{p-1}(\twonorm{c_k}^2-\twonorm{c_{k+1}}^2)\nonumber \\
    =& \sum_{k=0}^{p-1} \mu_k(f_k-f_{k+1}) +\sum_{k=0}^{p-1} \mu_k y_k\T(c_{k+1}-c_k) +\half( \twonorm{c_0}^2-\twonorm{c_p}^2).\label{l-done}
  \end{align}
  If $\mu_k = \mu$ for all sufficiently large $k$, then it follows from~Assumption~\ref{ass-0}, \eqref{f-fin1}, \eqref{c-fin1}, and~\eqref{l-done} that~\eqref{L-fin1} will hold for some sufficiently large $M_\Lscr$.  Otherwise, consider when $\mu_k \to 0$.  Taking the limit of~\eqref{l-done} as $p \to \infty$, we have from Assumption~\ref{ass-0} and conditions \eqref{f-fin2} and~\eqref{c-fin2} that
  \begin{equation*}
    \sum_{k=0}^\infty (\Lscr(x_k,y_k,\mu_k) - \Lscr(x_{k+1},y_k,\mu_k))<\infty.
  \end{equation*}
  Since the terms in this sum are all nonnegative, it follows from the Monotone Convergence Theorem that~\eqref{L-fin2} converges and is finite.  Moreover, we may again take the limit of~\eqref{l-done} as $p \to \infty$ and use~\eqref{f-fin2}, \eqref{c-fin2}, and~\eqref{L-fin2} to conclude that~\eqref{c-yay} holds.\qed
\end{proof}
}



We also need the following lemma that bounds the step-size sequence $\{\alpha_k\}$ below.

\begin{lemma} \label{lem-alphabound}
  There exists a positive monotonically decreasing sequence $\{C_j\}_{j \geq 1}$ such that, with the sequence $\{k_j\}$ computed in Algorithm~\ref{alg-aal}, the step-size sequence $\{\alpha_k\}$ satisfies
  \begin{equation*}
    \alpha_k \geq C_j > 0 \words{for all} k_j \leq k < k_{j+1}.
  \end{equation*}
\end{lemma}
\begin{proof}
  By Taylor's Theorem and Lemma~\ref{lem-is-descent}, it follows under Assumption~\ref{ass-0} that there exists $\tau>0$ such that for all sufficiently small $\alpha>0$ we have
  \begin{equation}
    \Lscr(x_k+\alpha s_k,y_k,\mu_{k}) - \Lscr(x_k,y_k,\mu_{k}) \leq -\alpha\Deltait\tilde{q}(s_k;x_k,y_k,\mu_{k}) + \tau\alpha^2\|s_k\|^2. \label{eq.ls1}
  \end{equation}
  On the other hand, during the line search implicit in line~\ref{line-alpha} of Algorithm~\ref{alg-aal}, a step-size $\alpha$ is rejected if
  \begin{equation}\label{eq.ls2}
    \Lscr(x_k+\alpha s_k,y_k,\mu_{k}) - \Lscr(x_k,y_k,\mu_{k}) > -\eta_s\alpha\Deltait\tilde{q}(s_k;x_k,y_k,\mu_{k}).
  \end{equation}
  Combining \eqref{eq.ls1}, \eqref{eq.ls2}, and~\eqref{cond1}  we have that a rejected step-size $\alpha$ satisfies
  \begin{equation*}
    \alpha > \frac{(1-\eta_s)\Deltait\tilde{q}(s_k;x_k,y_k,\mu_{k})}{\tau\twonorm{s_k}^2} \geq \frac{(1-\eta_s)\Deltait\tilde{q}(s_k;x_k,y_k,\mu_{k})}{\tau\TRALk^2}.
  \end{equation*}
  From this bound, the fact that if the line search rejects a step-size it multiplies it by $\gamma_{\alpha} \in (0,1)$, \eqref{cond1}, 
\eqref{decrease-AL}, \eqref{ass-3}, \eqref{sandy-AL}, and $\Gamma_k \in (1,2]$ (see Lemma~\ref{lem-algo}) it follows that, for all $k \in [k_j,k_{j+1})$,
\begin{align*}
  \alpha_k
    &\geq \frac{\gamma_\alpha(1-\eta_s)\Deltait\tilde{q}(s_k;x_k,y_k,\mu_k)}{\tau\TRALk^2} \\
    &\geq \frac{\gamma_\alpha(1-\eta_s)\kappa_1\kappa_5\twonorm{F\sub{AL}(x_k,y_k,\mu_k)}^2} {\tau \Gamma_k^2 \delta^2\twonorm{F\sub{AL}(x_k,y_k,\mu_k)}^2}\min\left\{\delta,\frac{1}{1 + \Omega_j}\right\} \\
    &\geq \frac{\gamma_\alpha(1-\eta_s)\kappa_1\kappa_5}{4\tau \delta^2} \min\left\{\delta,\frac{1}{1 + \Omega_j}\right\} =: C_j > 0,
\end{align*}
  as desired.
\end{proof}

We break the remainder of the analysis into two cases depending on whether there are a finite or an infinite number of modifications of the Lagrange multiplier estimate.

\subsection{A finite number of multiplier updates}

In this section, we suppose that the set $\Yscr$ in \eqref{def-y-iter} is finite in that the counter $j$ in Algorithm~\ref{alg-aal} satisfies
\begin{equation} \label{j-finite}
 j\in\{1,2, \dots \jbar\} \words{for some finite $\jbar$.}
\end{equation}
This allows us to define, and consequently use in our analysis, the quantities
\begin{equation} \label{cons-j-finite}
  t := t_\jbar > 0 \words{and} T := T_\jbar > 0.
\end{equation}

We provide two lemmas in this subsection.  The first considers cases when the penalty parameter converges to zero, and the second considers cases when the penalty parameter remains bounded away from zero.  This first case---in which the multiplier estimate is only modified a finite number of times and the penalty parameter vanishes---may be expected to occur when \eqref{nep} is infeasible.  Indeed, in this case, we show that every limit point of the primal iterate sequence is an infeasible stationary point.

\newcommand{\lemtwentysix}{
  If $|\Yscr| < \infty$ and $\mu_k \to 0$, then there exist a vector $y$ and integer $\kbar \geq 0$ such that 
  \begin{equation}\label{zero-1}
    y_k = y \words{for all} k \geq \kbar,
  \end{equation}
  and for some constant $\cbar > 0$, we have the limits
  \begin{equation}\label{zero-2}
    \lim_{k\to\infty} \twonorm{c_k} = \cbar > 0 \words{and} \lim_{k\to\infty} F\sub{FEAS}(x_k) = 0.
  \end{equation}
  Therefore, every limit point of $\{x_k\}_{k\geq 0}$ is an infeasible stationary point.
}

\iftoggle{arxivpaper}{
\begin{lemma}\label{lem-26}
 \lemtwentysix
 \end{lemma}
\begin{proof}
  It follows from \eqref{j-finite}, \eqref{cons-j-finite}, and the
  manner in which the multiplier estimates are updated in
  Algorithm~\ref{alg-aal} that there exists $y$ and a scalar $\kbar \geq
  k_\jbar$ such that~\eqref{zero-1} holds.  Thus, all that remains is to
  prove that \eqref{zero-2} holds for some $\cbar > 0$.

  From~\eqref{c-yay} and the supposition that $\mu_k \to 0$, it follows
  that $\twonorm{c_k} \to \cbar$ for some $\cbar \geq 0$.  If $\cbar =
  0$, then by Assumption~\ref{ass-0}, \eqref{zero-1}, and the fact that
  $\mu_k \to 0$ it follows that $\lim_{k\to \infty}
  \Grad_x\Lscr(x_k,y,\mu_k) = \lim_{k\to\infty} J_k\T c_k = 0$, which
  implies that $\lim_{k \to \infty} F\sub{AL}(x_k,y,\mu_k) =
  \lim_{k\to\infty} F\sub{FEAS}(x_k) = 0$.  This would imply that for
  some $k \geq \kbar$ the algorithm would set $j \gets \jbar + 1$, thus
  violating~\eqref{j-finite}.  Consequently, we may conclude that $\cbar
  > 0$, which proves the first limit in \eqref{zero-2}.  Now, to reach a
  contradiction to the second limit in \eqref{zero-2}, suppose that
  $F\sub{FEAS}(x_k) \nrightarrow 0$.  This, together with
  Assumption~\ref{ass-0}, \eqref{zero-1}, and the supposition that
  $\mu_k \to 0$, implies that there exist a positive constant $\epsilon$
  and an infinite index set $\mathcal{K}$ such that
  \begin{equation}\label{AL-positive}
    \twonorm{F\sub{AL}(x_k,y_k,\mu_k)} \geq \epsilon \text{ for all } k \in \mathcal{K}.
  \end{equation}
  It follows from \eqref{dec-L-ls}, \eqref{cond1}, \eqref{decrease-AL},
  \eqref{AL-positive}, and Lemma \ref{lem-alphabound} that for all $k
  \in \mathcal{K}$ we have
  \begin{align*}
    \Lscr(x_{k+1},y_k,\mu_k) 
      &= \Lscr(x_k + \alpha_k s_k,y_k,\mu_k) \\
      &\leq \Lscr(x_k,y_k,\mu_k) - \eta_s \alpha_k \Deltait\tilde{q}(s_k;x_k,y_k,\mu_k)\\
      &\leq \Lscr(x_k,y_k,\mu_k) - \eta_s \alpha_k \kappa_1 \Deltait\tilde{q}(\sCk;x_k,y_k,\mu_k)\\
      &\leq \Lscr(x_k,y_k,\mu_k) - \eta_s \alpha_k \kappa_1 \kappa_5 \twonorm{F\sub{AL}(x_k,y_k,\mu_k)}^2 \min\left\{\delta, \frac{1}{1+\Omega_\jbar}\right\}\\
      &\leq \Lscr(x_k,y_k,\mu_k) - \eta_s C_{\jbar} \kappa_1 \kappa_5 \epsilon^2 \min\left\{\delta, \frac{1}{1+\Omega_\jbar}\right\}.
  \end{align*}
  This implies that, for all $k\in \mathcal{K}$, the reduction
  $\Lscr(x_k,y_k,\mu_k)-\Lscr(x_{k+1},y_k,\mu_k)$ is greater than or
  equal to a positive constant. In the meantime, we know from Lemma
  \ref{lem-is-descent} and the way we update $x_k$ at each iteration
  that $\Lscr(x_k,y_k,\mu_k)-\Lscr(x_{k+1},y_k,\mu_k) \geq 0$ for all
  $k$.  Therefore, we have reached a contradiction to \eqref{L-fin2}.
  This implies that our supposition that $F\sub{FEAS}(x_k) \nrightarrow
  0$ cannot be true, so we have \eqref{zero-2}.
 \end{proof}
}{
\begin{lemma}[\protect{\cite[Lemma~\ref{lem-26}]{CJJR-arxiv}}]\label{lem-26}
 \lemtwentysix
 \end{lemma}
}

The next lemma considers the case when $\mu$ stays bounded away from
zero.  This is possible, for example, if the algorithm converges to an
infeasible stationary point that is stationary for the AL function for
the final Lagrange multiplier estimate and penalty parameter computed in
the algorithm.

\newcommand{\lemcartoons}{ 
  If $|\Yscr| < \infty$ and $\mu_k = \mu$ for
  some $\mu > 0$ for all sufficiently large $k$, then with $t$ defined
  in \eqref{cons-j-finite} there exist a vector $y$ and integer $\kbar
  \geq 0$ such that
  \begin{equation}\label{fini-1}
    y_k = y \words{and} \twonorm{c_k} \geq t \words{for all} k \geq \kbar,
  \end{equation}
  and we have the limit
  \begin{equation}\label{fini-2}
    \lim_{k\to\infty} F\sub{FEAS}(x_k) = 0.
  \end{equation}
  Therefore, every limit point of $\{x_k\}_{k\geq 0}$ is an infeasible 
  stationary point.
}

\iftoggle{arxivpaper}{
\begin{lemma}\label{cartoons}
 \lemcartoons
 \end{lemma}
\begin{proof}
  Since $|\Yscr| < \infty$, we know that~\eqref{j-finite}
  and~\eqref{cons-j-finite} hold for some $\jbar \geq 0$, and since we
  suppose that $\mu_k = \mu > 0$ for all sufficiently large $k$, it
  follows by the mechanisms for updating the Lagrange multiplier
  estimates in Algorithm~\ref{alg-aal} that there exists $y$ and a
  scalar $k' \geq k_\jbar$ such that
  \begin{equation} \label{mu-fx}
    \mu_k = \mu \words{and} y_k = y \words{for all} k \geq k'.
  \end{equation}
  Our next goal is to prove that
  \begin{equation} \label{lim-a}
    \lim_{k \to \infty} \twonorm{F\sub{AL}(x_k,y,\mu)} = 0.
  \end{equation}
  Indeed, to reach a contradiction, suppose that \eqref{lim-a} does not
  hold.  It then follows that there exist a positive number $\zeta$ and
  an infinite index set $\mathcal{K}'$ with all elements greater than or
  equal to $k'$ such that
  \begin{equation}\label{AL-positive-2}
    \twonorm{F\sub{AL}(x_k,y,\mu)} \geq \zeta \text{ for all } k \in \mathcal{K}'.
  \end{equation}
  Similar to the proof of Lemma~\ref{lem-26}, it then follows from
  \eqref{AL-positive-2}, \eqref{dec-L-ls}, \eqref{cond1},
  \eqref{decrease-AL}, and Lemma \ref{lem-alphabound} that for all $k
  \in \mathcal{K}'$ we have
  \begin{align*}
    \Lscr(x_{k+1},y,\mu) 
      &= \Lscr(x_k + \alpha_k s_k,y,\mu) \\
      &\leq \Lscr(x_k,y,\mu) - \eta_s \alpha_k \Deltait\tilde{q}(s_k;x_k,y,\mu)\\
      &\leq \Lscr(x_k,y,\mu) - \eta_s \alpha_k \kappa_1 \Deltait\tilde{q}(\sCk;x_k,y,\mu)\\
      &\leq \Lscr(x_k,y,\mu) - \eta_s \alpha_k \kappa_1 \kappa_5 \twonorm{F\sub{AL}(x_k,y,\mu)}^2 \min\left\{\delta, \frac{1}{1+\Omega_\jbar}\right\}\\
      &\leq \Lscr(x_k,y,\mu) - \eta_s C_{\jbar} \kappa_1 \kappa_5 \zeta^2 \min\left\{\delta, \frac{1}{1+\Omega_\jbar}\right\}.
  \end{align*}
  This implies that, for all $k\in \mathcal{K}'$, the reduction
  $\Lscr(x_k,y,\mu)-\Lscr(x_{k+1},y,\mu)$ is greater than or equal to a
  positive constant.  However, we know from Assumption \ref{ass-0} that
  $\Lscr(x_k,y,\mu)$ is bounded below.  Therefore, we have reached a
  contradiction, so \eqref{lim-a} must hold.

  The first consequence of \eqref{lim-a} is that it allows us to prove
  \eqref{fini-1}.  Indeed, it follows that there exists $\kbar \geq k'$
  such that $\twonorm{c_k} \geq t$ for all $k \geq \kbar$, since
  otherwise \eqref{lim-a} would imply that, for some $k \geq \kbar$,
  Algorithm~\ref{alg-aal} would set $j \gets \jbar + 1$, which
  violates~\eqref{j-finite}.  Thus, along with \eqref{mu-fx}, we have
  proved \eqref{fini-1}.

  The second consequence of \eqref{lim-a} is that it allows us to prove
  \eqref{fini-2}, which is all that remains to complete the proof of the
  lemma.  It follows from \eqref{cond1}, \eqref{lim-a}, and part (i) of
  Lemma~\ref{lem-algo} that
  \begin{equation} \label{s-z}
    \lim_{k\to\infty} \twonorm{s_k} \leq \lim_{k\to\infty} \TRALk = \lim_{k\to\infty} \Gamma_k \delta\twonorm{F\sub{AL}(x_k,y,\mu)} =  0. 
  \end{equation}
  Furthermore, from~\eqref{s-z} and Assumption~\ref{ass-0}, we have
  \begin{equation} \label{need-1}
    \lim_{k\to\infty} \delmodv(s_k;x_k) = 0,
  \end{equation}
  and, along with \eqref{cons-j-finite} and \eqref{fini-1}, we have
  \begin{equation} \label{v-pos}
    v_k - \half(\kappa_t t_\jbar)^2 \geq \half t^2 -  \half(\kappa_t t)^2 = \half(1-\kappa_t^2) t^2 > 0 \words{for all} k \geq \kbar.
  \end{equation}
  We may use these facts to prove $F\sub{FEAS}(x_k) \to 0$.  In
  particular, in order to derive a contradiction, suppose that
  $F\sub{FEAS}(x_k) \nrightarrow 0$.  Then, there exist a positive
  number $\xi$ and an infinite index set $\mathcal{K}''$ such that
  \begin{equation}\label{FEAS-positive}
    \twonorm{F\sub{FEAS}(x_k)} \geq \xi \words{for all} k \in \mathcal{K}''.
  \end{equation}
  Using~\eqref{cond2}, \eqref{decrease-feas}, \eqref{ass-4}, and \eqref{j-finite}, we then find for $k \in \mathcal{K}''$ that
  \begin{equation}\label{finally}
    \delmodv(r_{k};x_{k}) \geq \kappa_2 \delmodv(\rC_{k};x_{k}) \geq \kappa_2 \kappa_4\xi^2\min\left\{\frac{1}{1+\Omega_\jbar}, \delta \right\} =: \zeta' > 0.
  \end{equation}
  We may now combine~\eqref{finally}, \eqref{v-pos}, and~\eqref{need-1}
  to state that~\eqref{cond2b} must be violated for sufficiently large
  $k \in \mathcal{K}''$ and, consequently, the penalty parameter will be
  decreased.  However, this is a contradiction to \eqref{mu-fx}, so we
  conclude that $F\sub{FEAS}(x_k) \to 0$.  The fact that every limit
  point of $\{x_k\}_{k\geq 0}$ is an infeasible stationary point follows
  since $\twonorm{c_k} \geq t$ for all $k \geq \kbar$ from
  \eqref{fini-1} and $F\sub{FEAS}(x_k) \to 0$.
 \end{proof}
}{
\begin{lemma}[\protect{\cite[Lemma~\ref{cartoons}]{CJJR-arxiv}}]\label{cartoons}
 \lemcartoons
 \end{lemma}
}

This completes the analysis for the case that the set $\Yscr$ is finite.

\subsection{An infinite number of multiplier updates}

We now suppose that $|\Yscr| = \infty$.  In this case, it follows from the procedures for updating the Lagrange multiplier estimate and target values in Algorithm~\ref{alg-aal} that
\begin{equation} \label{ts-zero}
  \lim_{j\to\infty} t_j = \lim_{j\to\infty} T_j = 0.
\end{equation}

As in the previous subsection, we split the analysis in this subsection
into two results.  This time, we begin by considering the case when the
penalty parameter remains bounded below and away from zero.  In this
scenario, we state the following result that a subsequence of the
iterates converges to a first-order stationary point.  
\iftoggle{arxivpaper}{
The proof of the
corresponding result in \cite{curtis12} applies for
Algorithm~\ref{alg-aal}, so we do not provide it here for the sake of
brevity; the proof is relatively straightforward, essentially relying on
the fact that \eqref{ts-zero} squeezes the constraint violation and
stationarity measure error to zero to yield \eqref{goods}.
}

\begin{lemma}[\protect{\cite[Lemma~3.10]{curtis12}}.] \label{lem-good}
  If $|\Yscr| = \infty$ and $\mu_k = \mu$ for some $\mu > 0$ for all 
  sufficiently large $k$, then 
  \begin{subequations}\label{goods}
    \begin{align}
      \lim_{j\to\infty} c_{k_j} &= 0 \label{good-1} \\
      \words{and} \lim_{j\to\infty} F\sub{L}(x_{k_j},\yhat_{k_j}) &= 0. \label{good-2}
    \end{align}
  \end{subequations}
  Thus, any limit point $(\xstar,\ystar)$ of $\{(x_{k_j},\yhat_{k_j})\}_{j 
   \geq 0}$ is first-order stationary for \eqref{nep}. 
\end{lemma}
\iftoggle{fullpaper}
{
\begin{proof}
  Since $|\Yscr|=\infty$, it follows that the condition in line \ref{line-c-small} holds an infinite number of times.  The limit~\eqref{good-1} then follows by \eqref{ts-zero} since line \ref{line-kj} sets $k_j \gets k+1$ for all $k_j \in \Yscr$.
  
  To prove~\eqref{good-2}, we first define 
  \begin{equation*}
    \Yscr'
    = \{k_j\in\Yscr : \twonorm{F\sub{L}(x_{k_j},\yhat_{k_j})}
    \leq \twonorm{F\sub{AL}(x_{k_j},y_{k_j-1},\mu_{k_j-1})} \}.
  \end{equation*}
  It follows from~\eqref{ts-zero} and line~\ref{crunchy} of
  Algorithm~\ref{alg-aal} that if $\Yscr'$ is infinite, then 
  \begin{equation} \label{Fl-lim}
    \lim_{k_j\in\Yscr'} F\sub{L}(x_{k_j},\yhat_{k_j})=0.
  \end{equation}
  Meanwhile, it follows from~\eqref{ts-zero} and line~\ref{crunchy} of
  Algorithm~\ref{alg-aal} that if $\Yscr\backslash\Yscr'$ is infinite,
  then 
  \begin{equation} \label{dual-1}
    \lim_{k_j\in\Yscr\backslash\Yscr'} F\sub{AL}(x_{k_j},y_{k_j-1},\mu_{k_j-1}) = 0.
  \end{equation}
  Under Assumption~\ref{ass-0}, \eqref{dual-1} may be combined
  with~\eqref{good-1} and the fact that $\mu_k = \mu$ for some $\mu >
  0$ to deduce that if $\Yscr\backslash\Yscr'$ is infinite, then 
  \begin{equation} \label{dual-2-pre}
    \lim_{k_j\in\Yscr\backslash\Yscr'} F\sub{L}(x_{k_j}, y_{k_j-1}) = 0.
  \end{equation}
  We may now combine~\eqref{dual-2-pre} with~\eqref{new-y} to state
  that if $\Yscr\backslash\Yscr'$ is infinite, then 
  \begin{equation} \label{dual-2}
    \lim_{k_j\in\Yscr\backslash\Yscr'} F\sub{L}(x_{k_j},\yhat_{k_j}) = 0.
  \end{equation}
  The desired result~\eqref{good-2} now follows from \eqref{Fl-lim},
  \eqref{dual-2}, and the supposition that $|\Yscr|=\infty$.\qed 
\end{proof}
}

Finally, we consider the case when the penalty parameter converges to
zero.  
\iftoggle{arxivpaper}{
Again, we do not provide a proof of the following lemma since
that of the corresponding proof in \cite{curtis12} suffices here as
well.
}

\begin{lemma}[\protect{\cite[Lemma~3.13]{curtis12}}] \label{lem-29}
  If $|\Yscr| = \infty$ and $\mu_k \to 0$, then
  \begin{equation}\label{randy}
    \lim_{k\to\infty} c_k = 0.
  \end{equation}
  If, in addition, there exists a positive integer $p$ such that $\mu_{k_j-1} \geq \gamma_\mu^p\mu_{k_{j-1}-1}$ for all sufficiently large $j$, then there exists an infinite ordered set $\Jscr \subseteq\mathbb{N}$ such that
  \begin{equation}\label{dandy}
    \lim_{j\in\Jscr,j\to\infty} \|F\sub{L}(x_{k_j},\yhat_{k_j})\|_2 = 0 \words{or} \lim_{j\in\Jscr,j\to\infty} \|F\sub{L}(x_{k_j},\pi(x_{k_j},y_{k_j-1},\mu_{k_j-1}))\|_2 = 0.
  \end{equation}
  In such cases, if the first (respectively, second) limit in \eqref{dandy} holds, then along with \eqref{randy} it follows that any limit point of $\{(x_{k_j},\yhat_{k_j})\}_{j\in\Jscr}$ (respectively, $\{(x_{k_j},y_{k_j-1})\}_{j\in\Jscr}$)
is a first-order stationary point for \eqref{nep}.
\end{lemma}
\iftoggle{fullpaper}
{
\begin{proof}
  It follows from~\eqref{c-yay} that
  \begin{equation} \label{bug1}
    \lim_{k\to\infty} \twonorm{c_k} = \cbar \geq 0.
  \end{equation}
  However, it also follows from \eqref{ts-zero} and line~\ref{line-c-small} of Algorithm~\ref{alg-aal} that
  \begin{equation} \label{bug2}
    \lim_{j\to\infty} \twonorm{c_{k_j}} \leq \lim_{j\to\infty} t_j = 0.
  \end{equation}
  The desired result now follows from~\eqref{bug1} and~\eqref{bug2}.\qed
\end{proof}
}


\subsection{Proof of global convergence result}

\begin{proof}[Proof of Theorem~\ref{thm-main}]
  Lemmas~\ref{lem-26}, \ref{cartoons}, \ref{lem-good}
  and \ref{lem-29}
  cover the only four possible outcomes of Algorithm~\ref{alg-aal}; the
  result follows from those described in these lemmas.
\end{proof}

\iftoggle{arxivpaper}{
\section{Numerical Results} \label{app-results}

In this appendix, we provide detailed results of our experiments described in \S\ref{sec-numerical}.  The problems listed in the tables in this appendix are only those that were used in the performance data provided in \S\ref{sec-numerical}.  Also, note that we provide problem size information for the reformulated model, i.e., that resulting after transformations were employed to create a model of the form \eqref{nep}.  For example, a general inequality constraint with lower and upper bounds would have been augmented with two auxiliary variables to form two equality constraints with lower bounds on the auxiliary variables.

In Tables~\ref{tab.matlab_first}--\ref{tab.matlab_last} for our \Matlab{} software, we indicate the name (Name) along with the numbers of variables ($n$), equality constraints ($m_e$), and bound constraints ($m_b$) of each problem solved.  Then, for each algorithm, we indicate the termination flag (Flag) along with the numbers of iterations (Iter.), function evaluations (Func.), and gradient evaluations (Grad.) required before termination.  The flags indicate whether a first-order stationary point was found (Opt.), an infeasible stationary point was found (Inf.), the iteration limit was reached (Itr.), or the time limit was reached (Time).

\begin{tiny}
\rowcolors{1}{lightgray}{}
\begin{longtable}{|cccc|cccc|cccc|cccc|}
\rowcolor{white}
\caption{\Matlab{} line search algorithms, results on \CUTEst{} problems}\label{tab.matlab_first}\\
  \hline \rowcolor{white}
       &     &     &     & \multicolumn{4}{c|}{\balls} & \multicolumn{4}{c|}{\aalls} & \multicolumn{4}{c|}{\aallssafe} \\ \rowcolor{white}
  Name & $n$ & $m_e$ & $m_b$ & Flag & Iter. & Func. & Grad. & Flag & Iter. & Func. & Grad. & Flag & Iter. & Func. & Grad. \\
  \hline
    ACOPP30 &    154 &    142 &    166 & Itr. &  10000 &  10032 &  10002 & Itr. &  10000 &  10011 &  10002 & Itr. &  10000 &  10009 &  10002\\
  ACOPR14 &    106 &     96 &     88 & Opt. &   5976 &   5998 &   5978 & Opt. &     72 &     78 &     74 & Opt. &     67 &     74 &     69\\
 AIRCRFTA &      5 &      5 &      0 & Opt. &    151 &    152 &    153 & Opt. &    151 &    152 &    153 & Opt. &    151 &    152 &    153\\
  AIRPORT &    126 &     42 &    210 & Opt. &    176 &    226 &    178 & Opt. &     64 &     80 &     66 & Opt. &    113 &    129 &    115\\
 ALSOTAME &      2 &      1 &      4 & Opt. &     10 &     11 &     12 & Opt. &      9 &     10 &     11 & Opt. &      9 &     10 &     11\\
\hline
\end{longtable}
\end{tiny}

\newpage

\begin{tiny}
\rowcolors{1}{lightgray}{}
\begin{longtable}{|cccc|cccc|cccc|cccc|}
\rowcolor{white}
\caption{\Matlab{} trust region algorithms, results on \CUTEst{} problems}\\
  \hline \rowcolor{white}
       &     &     &     & \multicolumn{4}{c|}{\baltr} & \multicolumn{4}{c|}{\aaltr} & \multicolumn{4}{c|}{\aaltrsafe} \\ \rowcolor{white}
  Name & $n$ & $m_e$ & $m_b$ & Flag & Iter. & Func. & Grad. & Flag & Iter. & Func. & Grad. & Flag & Iter. & Func. & Grad. \\
  \hline
    ACOPP30 &    154 &    142 &    166 & Itr. &  10000 &  10001 &   9978 & Itr. &  10000 &  10001 &   9992 & Opt. &     96 &     97 &     90\\
  ACOPR14 &    106 &     96 &     88 & Opt. &   8211 &   8212 &   7954 & Itr. &  10000 &  10001 &   9641 & Itr. &  10000 &  10001 &   9625\\
 AIRCRFTA &      5 &      5 &      0 & Opt. &      5 &      6 &      7 & Opt. &      5 &      6 &      7 & Opt. &      5 &      6 &      7\\
  AIRPORT &    126 &     42 &    210 & Opt. &    275 &    276 &    221 & Opt. &     71 &     72 &     59 & Opt. &    150 &    151 &    138\\
 ALSOTAME &      2 &      1 &      4 & Opt. &     10 &     11 &     12 & Opt. &     14 &     15 &     16 & Opt. &     14 &     15 &     16\\
\hline
\end{longtable}
\end{tiny}

\begin{tiny}
\rowcolors{1}{lightgray}{}
\begin{longtable}{|cccc|cccc|cccc|cccc|}
\rowcolor{white}
\caption{\Matlab{} line search algorithms, results on \COPS{} problems}\\
  \hline \rowcolor{white}
       &     &     &     & \multicolumn{4}{c|}{\balls} & \multicolumn{4}{c|}{\aalls} \\ \rowcolor{white}
  Name & $n$ & $m_e$ & $m_b$ & Flag & Iter. & Func. & Grad. & Flag & Iter. & Func. & Grad. \\
  \hline
   bearing1 &   2500 &      0 &   2500 & Opt. &      4 &      5 &      6 & Opt. &      4 &      5 &      6\\
 camshape1 &   3198 &   2398 &   3998 & Opt. &     30 &     31 &     32 & Opt. &     24 &     25 &     26\\
  catmix1 &   1098 &    798 &    600 & Opt. &    366 &    672 &    368 & Itr. &  10000 &  10201 &  10002\\
 channel1 &   1598 &   1598 &      0 & Opt. &    357 &    690 &    359 & Opt. &    357 &    690 &    359\\
    elec1 &    150 &     50 &      0 & Opt. &     65 &    119 &     67 & Opt. &    200 &    205 &    202\\
\hline
\end{longtable}
\end{tiny}

\begin{tiny}
\rowcolors{1}{lightgray}{}
\begin{longtable}{|cccc|cccc|cccc|cccc|}
\rowcolor{white}
\caption{\Matlab{} trust region algorithms, results on \COPS{} problems}\\
  \hline \rowcolor{white}
       &     &     &     & \multicolumn{4}{c|}{\baltr} & \multicolumn{4}{c|}{\aaltr} \\ \rowcolor{white}
  Name & $n$ & $m_e$ & $m_b$ & Flag & Iter. & Func. & Grad. & Flag & Iter. & Func. & Grad. \\
  \hline
   bearing1 &   2500 &      0 &   2500 & Opt. &      3 &      4 &      5 & Opt. &      3 &      4 &      5\\
 camshape1 &   3198 &   2398 &   3998 & Inf. &   4726 &   4727 &   4728 & Opt. &     16 &     17 &     18\\
  catmix1 &   1098 &    798 &    600 & Inf. &    745 &    746 &    496 & Time &   8021 &   8022 &   7908\\
 channel1 &   1598 &   1598 &      0 & Opt. &    610 &    611 &    402 & Opt. &    573 &    574 &    393\\
    elec1 &    150 &     50 &      0 & Opt. &    108 &    109 &     75 & Opt. &     80 &     81 &     59\\
\hline
\end{longtable}
\end{tiny}

\newpage

\begin{tiny}
\rowcolors{1}{lightgray}{}
\begin{longtable}{|cccc|cccc|cccc|cccc|}
\rowcolor{white}
\caption{\Matlab{} line search algorithms, results on OPF problems}\\
  \hline \rowcolor{white}
       &     &     &     & \multicolumn{4}{c|}{\balls} & \multicolumn{4}{c|}{\aalls} \\ \rowcolor{white}
  Name & $n$ & $m_e$ & $m_b$ & Flag & Iter. & Func. & Grad. & Flag & Iter. & Func. & Grad. \\
  \hline
   OPFapproxIEEE014 &     76 &     68 &     84 & Opt. &     16 &     19 &     18 & Opt. &     16 &     19 &     18\\
 OPFcartIEEE014 &    108 &    100 &     98 & Opt. &     13 &     14 &     15 & Opt. &     12 &     13 &     14\\
 OPFdcIEEE014 &     49 &     49 &     40 & Opt. &      2 &      3 &      4 & Opt. &      2 &      3 &      4\\
 OPFpolarIEEE014 &     76 &     68 &     84 & Opt. &     14 &     15 &     16 & Opt. &     15 &     16 &     17\\
 OPFapproxIEEE030 &    151 &    142 &    160 & Opt. &     22 &     23 &     24 & Opt. &     18 &     19 &     20\\
\hline
\end{longtable}
\end{tiny}

\begin{tiny}
\rowcolors{1}{lightgray}{}
\begin{longtable}{|cccc|cccc|cccc|cccc|}
\rowcolor{white}
\caption{\Matlab{} trust region algorithms, results on OPF problems}\label{tab.matlab_last}\\
  \hline \rowcolor{white}
       &     &     &     & \multicolumn{4}{c|}{\baltr} & \multicolumn{4}{c|}{\aaltr} \\ \rowcolor{white}
  Name & $n$ & $m_e$ & $m_b$ & Flag & Iter. & Func. & Grad. & Flag & Iter. & Func. & Grad. \\
  \hline
   OPFapproxIEEE014 &     76 &     68 &     84 & Opt. &     24 &     25 &     22 & Opt. &     41 &     42 &     31\\
 OPFcartIEEE014 &    108 &    100 &     98 & Opt. &     23 &     24 &     21 & Opt. &     21 &     22 &     23\\
 OPFdcIEEE014 &     49 &     49 &     40 & Opt. &      2 &      3 &      4 & Opt. &      2 &      3 &      4\\
 OPFpolarIEEE014 &     76 &     68 &     84 & Opt. &     14 &     15 &     16 & Opt. &     22 &     23 &     21\\
 OPFapproxIEEE030 &    151 &    142 &    160 & Opt. &     22 &     23 &     24 & Opt. &     18 &     19 &     20\\
\hline
\end{longtable}
\end{tiny}

In Table~\ref{tab.lancelot} for our \lancelot{} software, we indicate the name (Name) along with the numbers of variables ($n$), equality constraints ($m_e$), and bound constraints ($m_b$) of each problem solved.  Then, for each algorithm, we indicate the termination flag (Flag) along with the numbers of iterations (Iter.), gradient evaluations (Grad.), and time (Time) required before termination.  The flags indicate whether a first-order stationary point was found (Opt.), the problem was suspected to be infeasible (Inf.), the iteration limit was reached (Itr.), or the algorithm determined that no more progress could be made.  In this last situation, 
we verified whether the final point was approximately stationary based on a loosened threshold criteria (Thr.) or not (Ter.).

\newpage

\begin{tiny}
\rowcolors{1}{lightgray}{}
\begin{longtable}{|cccc|cccc|cccc|cccc|}
\rowcolor{white}
\caption{\lancelot{} algorithms, results on \CUTEst{} problems}\label{tab.lancelot}\\
  \hline \rowcolor{white}
       &     &     &     & \multicolumn{4}{c|}{\Lancelot{}} & \multicolumn{4}{c|}{\LancelotS{}} & \multicolumn{4}{c|}{\LancelotSS{}} \\ \rowcolor{white}
  Name & $n$ & $m_e$ & $m_b$ & Flag & Iter. & Grad. & Time & Flag & Iter. & Grad. & Time & Flag & Iter. & Grad. & Time \\
  \hline
     A0NSDSDS &  6012 &  2004 &  1000 & Opt. &    15 &    19 & 27.03 & Opt. &    19 &    23 & 33.04 & Opt. &    19 &    23 & 33.04\\
      A4X12 &   514 &   385 &   384 & Thr. &  1849 &  1790 & 445.10 & Opt. &  1307 &  1265 & 200.37 & Opt. &  2647 &  2597 & 713.78\\
   A5NSDSDM &  6012 &  2004 &  1000 & Opt. &    15 &    19 & 26.95 & Opt. &    19 &    23 & 33.05 & Opt. &    19 &    23 & 33.02\\
   A5NSSNSM &  6012 &  2004 &  1000 & Opt. &    15 &    19 & 27.03 & Opt. &    19 &    23 & 33.03 & Opt. &    19 &    23 & 32.94\\
    ACOPP30 &   154 &   142 &   166 & Thr. &    92 &    91 & 0.53 & Opt. &    99 &    93 & 0.59 & Thr. &   106 &   100 & 0.57\\
\hline
\end{longtable}
\end{tiny}

}

\end{document}